\pdfoutput=1
\RequirePackage{ifpdf}
\ifpdf 
\documentclass[pdftex]{sigma}
\else
\documentclass{sigma}
\fi

\usepackage{tikz,tikz-cd}
\usetikzlibrary{arrows,shapes.geometric,intersections,decorations.markings}
\usetikzlibrary{calc,intersections,through,backgrounds}
\usetikzlibrary{patterns}

\newlength\friezelen
\usepackage{array} 
\newcolumntype{Q}{>{\centering}p{\friezelen}<{}}
\newcommand{\til}[0]{\widetilde{\theta}}

\tikzset {->-/.style={decoration={markings, mark=at position .5 with {\arrow{latex}}}, postaction={decorate}}}
\newcommand{\midarrow}{\tikz \draw[-triangle 90] (0,0) -- +(.1,0);}

\numberwithin{equation}{section}

\newtheorem{Theorem}{Theorem}[section]
\newtheorem{Corollary}[Theorem]{Corollary}
\newtheorem{Lemma}[Theorem]{Lemma}
\newtheorem{Proposition}[Theorem]{Proposition}
 {\theoremstyle{Definition}
\newtheorem{Definition}[Theorem]{Definition}
\newtheorem{Example}[Theorem]{Example}
\newtheorem{Remark}[Theorem]{Remark}
\newtheorem{Question}[Theorem]{Question} }

\DeclareMathOperator{\wt}{wt}
\DeclareMathOperator{\inv}{inv}

\newcommand{\s}[1]{{\color{red}#1}}
\renewcommand{\phi}{\varphi}

\renewcommand{\l}[2]{\lambda_{#1#2}}

\begin{document}

\newcommand{\arXivNumber}{2102.09143}

\renewcommand{\PaperNumber}{080}

\FirstPageHeading

\ShortArticleName{An Expansion Formula for Decorated Super-Teichm\"uller Spaces}

\ArticleName{An Expansion Formula\\ for Decorated Super-Teichm\"uller Spaces}

\Author{Gregg MUSIKER, Nicholas OVENHOUSE and Sylvester W. ZHANG}
\AuthorNameForHeading{G.~Musiker, N.~Ovenhouse and S.W.~Zhang}
\Address{School of Mathematics, University of Minnesota, Minneapolis, MN 55455, USA}
\Email{\href{mailto:musiker@math.umn.edu}{musiker@umn.edu}, \href{mailto:ovenh001@umn.edu}{ovenh001@umn.edu}, \href{mailto:swzhang@umn.edu}{swzhang@umn.edu}}

\ArticleDates{Received March 31, 2021, in final form August 27, 2021; Published online September 01, 2021}

\Abstract{Motivated by the definition of super-Teichm\"uller spaces, and Penner--Zeitlin's recent extension of this definition to decorated super-Teichm\"uller space, as examples of super Riemann surfaces, we use the super Ptolemy relations to obtain formulas for super $\lambda$-lengths associated to arcs in a bordered surface. In the special case of a disk, we are able to give combinatorial expansion formulas for the super $\lambda$-lengths associated to diagonals of a polygon in the spirit of Ralf Schiffler's $T$-path formulas for type $A$ cluster algebras. We further connect our formulas to the super-friezes of Morier-Genoud, Ovsienko, and Tabachnikov, and obtain partial progress towards defining super cluster algebras of type $A_n$. In particular, following Penner--Zeitlin, we are able to get formulas (up to signs) for the $\mu$-invariants associated to triangles in a triangulated polygon, and explain how these provide a step towards understanding odd variables of a super cluster algebra.}

\Keywords{cluster algebras; Laurent polynomials; decorated Teichm\"uller spaces; supersymmetry}

\Classification{13F60; 17A70; 30F60}

\section{Introduction}
Cluster algebras were introduced in \cite{fz_02} as certain commutative algebras whose generators are defined by
a combinatorial recursive procedure called ``mutation''. They were originally conceived to study certain problems
in Lie theory and quantum groups, but have since found many surprising and deep connections to other areas of mathematics and physics.

In \cite{fst_08}, a class of cluster algebras was defined starting from the data of a surface with boundary, together with
a collection of punctures and marked points on the boundary. It was shown in \cite{gsv_05} and \cite{fg_06} that these
cluster algebras could be interpreted geometrically as functions on decorated Teichm\"{u}ller spaces.
Specifically, the cluster variables are coordinate functions known as \emph{$\lambda$-lengths} or \emph{Penner coordinates} \cite{penner_12}.

One of the earliest results in the subject is the \emph{Laurent phenomenon}, which says that all cluster variables
can be expressed as Laurent polynomials in terms of a fixed set of initial cluster variables. Over the years, several
explicit formulas have been given for these Laurent expressions in the case of cluster algebras coming from surfaces,
with the Laurent monomials being indexed by various combinatorial objects.
Some of these include Schiffler's ``$T$-paths'' \cite{schiffler_08}, perfect matchings of the snake graphs of Musiker, Schiffler, and Williams \cite{msw_11},
and Yurikusa's ``angle matchings'' \cite{yurikusa_19}.

Going beyond the commutative case, Berenstein and Zelevinsky defined quantum cluster algebras \cite{berenstein2005quantum}
where cluster variables (of the same cluster) quasi-commute with one another, meaning that exchanging the order of multiplication
in a cluster monomial could alter the expression by yielding a power of $q$ out in front of such a term.
More recently, Berenstein and Retakh \cite{berenstein2018noncommutative} defined in the case of surfaces a (completely) non-commutative
model of cluster variables and obtained non-commutative Laurent expansions analogous to $T$-paths.
Such non-commutative expressions could also be defined as quasi-Pl\"ucker coordinates.

Recently, Penner and Zeitlin defined the notion of the \emph{decorated super-Teichm\"{u}ller space} associated to a bordered marked surface \cite{pz_19}.
This work builds off of earlier work on super-Teichm\"uller spaces for super Riemann surfaces \cite{cr_88}.
The coordinates on a super-space are broken into two classes: namely even coordinates and odd coordinates.
Even coordinates are ordinary commutative variables but odd coordinates anti-commute with one another.
Odd coordinates are also commonly known as Grassmann variables. As in the classical commutative case, the coordinates correspond to arcs in a fixed triangulation of the surface.
They described a super version of the Ptolemy relation, which is an expression for how the coordinates change
when changing the choice of triangulation.

Our main result in this paper is an explicit formula for the super $\lambda$-lengths in the case of marked disks, generalizing the $T$-path formulation of Schiffler.
 Like Schiffler's formula, the terms in our formula are also indexed by objects which closely resemble the $T$-paths from the classical case.

\section{Decorated Teichm\"uller theory}

First we review some background on decorated Teichm\"{u}ller spaces. For a detailed reference, see~\cite{penner_12}.
Let $S$ be a surface with boundary, and let $p_1,\dots,p_n$ be a collection of marked points on the boundary,
such that each boundary component contains at least one marked point. More generally, we can also have a collection of interior marked points (or \emph{punctures}),
but we will not be concerned with this case in this paper. We also equip the surface with a triangulation, where the arcs terminate
at the marked boundary points. The \emph{Teichm\"{u}ller space} of~$S$, denoted~$\mathcal{T}(S)$, is the space of (equivalence classes of) hyperbolic metrics on $S$ with constant negative
curvature, with cusps at the marked points. Because of the cusps, any geodesic between marked points has infinite hyperbolic length.

The \emph{decorated Teichm\"{u}ller space} of $S$, written $\widetilde{\mathcal{T}}(S)$, is a trivial vector bundle over $\mathcal{T}(S)$,
with fiber ${\mathbb R}_{> 0}^n$. The fibers represent a choice of a positive real number associated to each marked point.
At each marked point, we draw a horocycle whose size (or \emph{height}) is determined by the corresponding positive number.
Truncating the geodesics using these horocycles, it now makes sense to talk about their lengths. If $\ell$ is the truncated
length of one of these geodesic segments, then the $\lambda$-length (or \emph{Penner coordinate}) associated to that geodesic arc
is defined to be
\[ \lambda := \exp(\ell/2). \]

Fixing a triangulation of the marked surface, the collection of $\lambda$-lengths corresponding to the arcs in the triangulation (including segments
of the boundary) form a system of coordinates for~$\widetilde{\mathcal{T}}(S)$. Choosing a different triangulation results in a different
system of coordinates, but they are related by simple transformations which are a hyperbolic analogue of Ptolemy's theorem from classical Euclidean geometry.
If two triangulations differ by the flip of a single arc as in Figure~\ref{fig:ptolemy},
then the $\lambda$-lengths are related by $ef = ac + bd$.

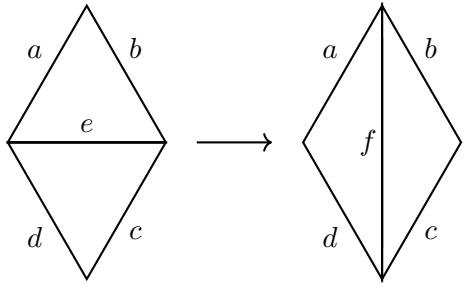
\begin{figure}[h!]
\centering

\begin{tikzpicture}[scale=0.7, baseline, thick]

    \draw (0,0)--(3,0)--(60:3)--cycle;
    \draw (0,0)--(3,0)--(-60:3)--cycle;

    \draw (0,0) -- (3,0);

    \draw node[above]      at (70:1.5){$a$};
    \draw node[above]      at (30:2.8){$b$};
    \draw node[below]      at (-30:2.8){$c$};
    \draw node[below=-0.1] at (-70:1.5){$d$};
    \draw node[above] at (1.5, 0){$e$};

    \draw node[left] at (0,0) {};
    \draw node[above] at (60:3) {};
    \draw node[right] at (3,0) {};
    \draw node[below] at (-60:3) {};

\end{tikzpicture}
\begin{tikzpicture}[baseline]
    \draw[->, thick](0,0)--(1,0);
    \node[above]  at (0.5,0) {};
\end{tikzpicture}
\begin{tikzpicture}[scale=0.7, baseline, thick,every node/.style={sloped,allow upside down}]
    \draw (0,0)--(60:3)--(-60:3)--cycle;
    \draw (3,0)--(60:3)--(-60:3)--cycle;

    \draw node[above]      at (70:1.5)  {$a$};
    \draw node[above]      at (30:2.8)  {$b$};
    \draw node[below]      at (-30:2.8) {$c$};
    \draw node[below=-0.1] at (-70:1.5) {$d$};
    \draw node[left]       at (1.6,0)   {$f$};

    \draw (1.5,-2) --  (1.5,2);

    \draw node[left] at (0,0) {};
    \draw node[above] at (60:3) {};
    \draw node[right] at (3,0) {};
    \draw node[below] at (-60:3) {};

\end{tikzpicture}
\caption{Ptolemy transformation.}\label{fig:ptolemy}
\end{figure}

\section[Laurent expression for lambda-lengths]{Laurent expression for $\boldsymbol{\lambda}$-lengths}

In this paper, we will only be concerned with the case that the surface $S$ is a disk with marked points on its boundary (which we will
picture as a convex polygon). So we restrict to that case now.

Fix a triangulation of an $n$-gon, and label the vertices $1$ through $n$.
Schiffler~\cite{schiffler_08} defined a~\emph{$T$-path} from $i$ to $j$ to be a sequence
\[ \alpha = \big(a_0,\dots,a_{\ell(\alpha)} \,|\, t_1,\dots,t_{\ell(\alpha)}\big) \]
such that
\begin{itemize}\itemsep=0pt
 \item[(T1)] $a_0,\dots,a_{\ell(\alpha)}$ are vertices of the polygon,
 \item[(T2)] $t_k$ is an arc in the triangulation connecting $a_{k-1}$ to $a_k$,
 \item[(T3)] no arc is used more than once,
 \item[(T4)] $\ell(\alpha)$ is odd,
 \item[(T5)] if $k$ is even, then $t_k$ crosses the arc connecting vertices $i$ and $j$,
 \item[(T6)] if $k < l$ and both $t_k$ and $t_l$ cross the arc from $i$ to $j$, then
 the point of intersection with $t_k$ is closer to $i$, and the point of intersection
 with $t_l$ is closer to $j$.
\end{itemize}

An example of a $T$-path in a hexagon is shown in Figure~\ref{fig:t_path}.
The even-numbered edges are colored blue, and the odd edges red, for emphasis.

\begin{figure}[h!] \centering

 \begin {tikzpicture}[scale=1.3]
 \coordinate (A) at (-1,0);
 \coordinate (B) at (-0.5,-0.866);
 \coordinate (C) at (0.5,-0.866);
 \coordinate (D) at (1,0);
 \coordinate (E) at (0.5,0.866);
 \coordinate (F) at (-0.5,0.866);

 \draw (A) -- (B) -- (C) -- (D) -- (E) -- (F) -- cycle;

 \draw (B) -- (F) -- (C) -- (E);

 \draw (A) node[left] {$i$};
 \draw (D) node[right] {$j$};

 \draw[red, line width = 1.2] (A) -- (B);
 \draw[red, line width = 1.2] (F) -- (E);
 \draw[red, line width = 1.2] (C) -- (D);

 \draw[blue, line width = 1.2] (B) -- (F);
 \draw[blue, line width = 1.2] (E) -- (C);
 \end {tikzpicture}

 \caption{A $T$-path from $i$ to $j$.} \label {fig:t_path}
\end{figure}
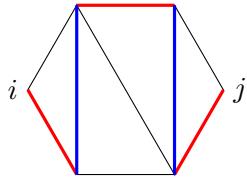

We denote the set of all $T$-paths from $i$ to $j$ by $T_{ij}$. Given a $T$-path $\alpha$,
we define a Laurent monomial $x_\alpha$ (in terms of the $\lambda$-lengths of a fixed triangulation) as
the product of the $\lambda$-lengths used in the $T$-path, with the even-numbered ones inverted. That is,
if the ordered sequence of edges in $\alpha$ is $e_1, e_2, \dots, e_{2m+1}$, then
\[ x_\alpha := \frac{\prod_{k=0}^m x_{e_{2k+1}}}{\prod_{k=1}^m x_{e_{2k}}}. \]

Schiffler proved the following theorem relating $\lambda$-lengths and $T$-paths.

\begin{Theorem}[{\cite[Theorem 1.2]{schiffler_08}}] Let $x_{ij}$ be the $\lambda$-length corresponding to the geodesic arc
 connecting vertices $i$ and $j$ in a triangulated polygon. Then
 \[ x_{ij} = \sum_{\alpha \in T_{ij}} x_\alpha. \]
\end{Theorem}

\begin{Remark} \label{rem:subtriang}\rm
As a consequence of the definition of $T$-paths, to compute $x_{ij}$ it is sufficient to consider only the sub-polygon consisting of the
triangles which the geodesic arc connecting vertices $i$ and $j$ crosses.
\end{Remark}

\section{Decorated super-Teichm\"{u}ller theory}\label{sec:dec_super}

Super-Teichm\"{u}ller spaces have been studied for several years now (see for example~\cite{cr_88}).
Recently, Penner and Zeitlin introduced a decorated version of super-Teichm\"{u}ller spaces~\cite{pz_19}.
It is generated by even variables corresponding to the
$\lambda$-lengths of a triangulation, as well as odd variables (called \emph{$\mu$-invariants}) corresponding to the triangles. They give
a super version of the Ptolemy relation, which reads as follows (see Figure~\ref{fig:super_ptolemy} for the meanings of the variables)
\begin{gather*}
 ef = (ac+bd)\left( 1+{{\sigma\theta\sqrt{\chi}}\over{1+\chi}}\right),\qquad
 \sigma' = {{\sigma-\sqrt{\chi}\theta}\over{\sqrt{1+\chi}}},\qquad
 \theta' = {{\theta+\sqrt{\chi}\sigma}\over{\sqrt{1+\chi}}},
\end{gather*}
where $ \chi = \frac{ac}{bd}$.

\begin{figure}[h!]
\centering

\begin{tikzpicture}[scale=0.7, baseline, thick]

    \draw (0,0)--(3,0)--(60:3)--cycle;
    \draw (0,0)--(3,0)--(-60:3)--cycle;

    \draw (0,0)--node {\midarrow} (3,0);

    \draw node[above]      at (70:1.5){$a$};
    \draw node[above]      at (30:2.8){$b$};
    \draw node[below]      at (-30:2.8){$c$};
    \draw node[below=-0.1] at (-70:1.5){$d$};
    \draw node[above] at (1,-0.12){$e$};

    \draw node[left] at (0,0) {};
    \draw node[above] at (60:3) {};
    \draw node[right] at (3,0) {};
    \draw node[below] at (-60:3) {};

    \draw node at (1.5,1){$\theta$};
    \draw node at (1.5,-1){$\sigma$};
\end{tikzpicture}
\begin{tikzpicture}[baseline]
    \draw[->, thick](0,0)--(1,0);
    \node[above]  at (0.5,0) {};
\end{tikzpicture}
\begin{tikzpicture}[scale=0.7, baseline, thick,every node/.style={sloped,allow upside down}]
    \draw (0,0)--(60:3)--(-60:3)--cycle;
    \draw (3,0)--(60:3)--(-60:3)--cycle;

    \draw node[above]      at (70:1.5)  {$a$};
    \draw node[above]      at (30:2.8)  {$b$};
    \draw node[below]      at (-30:2.8) {$c$};
    \draw node[below=-0.1] at (-70:1.5) {$d$};
    \draw node[left]       at (1.7,1)   {$f$};

    \draw (1.5,-2) --node {\midarrow} (1.5,2);

    \draw node[left] at (0,0) {};
    \draw node[above] at (60:3) {};
    \draw node[right] at (3,0) {};
    \draw node[below] at (-60:3) {};

    \draw node at (0.8,0){$\theta'$};
    \draw node at (2.2,0){$\sigma'$};
\end{tikzpicture}
\caption{Super Ptolemy transformation.}\label{fig:super_ptolemy}
\end{figure}
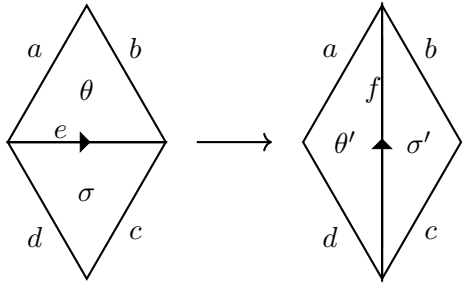

We will usually find it convenient to re-write these equations without $\chi$, as follows
\begin{gather}
 ef = ac+bd+\sqrt{abcd} \sigma\theta, \label{eq:mutation}\\
 \sigma' = {{\sigma\sqrt{bd}-\theta\sqrt{ac}}\over{\sqrt{ac+bd}}}, \label{eq:mu_mutation_neg} \\
 \theta' = {{ \theta\sqrt{bd}+\sigma \sqrt{ac}}\over{\sqrt{ac+bd}}}.\label{eq:mu_mutation_pos}
\end{gather}
An important corollary of these equations is the following
\begin{equation}\label{eq:mu_cross}
	\sigma\theta=\sigma'\theta',
\end{equation}
which will be used frequently in our proofs.

To understand the oriented arrows in Figure~\ref{fig:super_ptolemy} and the minus sign in equation~\eqref{eq:mu_mutation_neg},
we need the combinatorial data of a \emph{spin structure}. In~\cite{cr_07} and~\cite{cr_08}, an isomorphism was shown between
the set of equivalence classes of spin structures on a surface and the set of isomorphism classes of Kasteleyn orientations
of a~fatgraph spine of the surface. Dual to any fatgraph spine is a~triangulation of the surface, and so an orientation of
the fatgraph corresponds to an orientation of a~triangulation (requiring the dual edge to cross from left to right). For our
purposes a~\emph{spin structure} will be a~choice of orientation of the edges of a~triangulation, modulo a~certain equivalence relation,
which we now describe.

Fix a triangulation $\Delta$ and an orientation $\tau$ of the edges in that triangulation. For any triangle~$t$,
consider the transformation which reverses the orientation of the three sides of~$t$. Define an equivalence relation on orientations by declaring that $\tau \sim \tau'$ if they differ by a~sequence of these transformations. In~\cite{pz_19}, a spin
structure is represented combinatorially by an equivalence class of orientations.

In \cite{pz_19}, the authors did not consider surfaces with boundary, as we do here. Because the orientation of boundary segments
will play no role in our formulas for $\lambda$-lengths (but not $\mu$-invariants), we can often ignore the boundary orientations.
Note that the hypothesis of Proposition~\ref{prop:equiv_rel} on the triangulation is sufficient because of Remark~\ref{rem:subtriang}.

\begin {Proposition} \label{prop:equiv_rel}
 Fix a triangulation of a polygon in which every triangle has at least one boundary edge. Then there is a unique spin structure after ignoring the boundary edges. In particular, this means that, from any representative orientation of a fixed spin structure, one can obtain all other orientations on the interior diagonals, without leaving this spin structure.
\end {Proposition}
\begin {proof}
 Because every triangle has at least one boundary edge, one can naturally sequence the triangles (and internal diagonals) from left-to-right.
 We may picture the polygon as follows, to emphasize this:
 
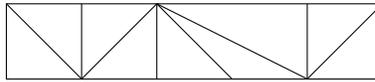
\begin{center}
 \begin{tikzpicture}
 \draw (0,0) -- (0,1) -- (5,1) -- (5,0) -- cycle;

 \draw (0,1) -- (1,0) -- (1,1);
 \draw (1,0) -- (2,1);

 \draw (2,1) -- (2,0);
 \draw (2,1) -- (3,0);
 \draw (2,1) -- (4,0);

 \draw (4,0) -- (4,1);
 \draw (4,0) -- (5,1);
 \end{tikzpicture}
 \end{center}
 We will demonstrate that we may change the orientation of a single edge while remaining in the same equivalence class. Label
 the triangles from left to right as $t_1,t_2,\dots$ and the internal diagonals as $d_1,d_2,\dots$. We will argue that we can change
 the orientation of just $d_k$, by induction on $k$.

 If $k=1$, then $d_1$ is an edge of $t_1$, and the other two edges of $t_1$ are on the boundary. Because we ignore boundary orientations,
 the equivalence relation allows us to reverse the orientation of just $d_1$.

 Now, for arbitrary $d_k$, reverse the orientations around triangle~$t_k$. This affects both $d_k$ and~$d_{k-1}$. But by induction, we may
 change just $d_{k-1}$ while staying in the equivalence class. This proves the claim.
\end {proof}

In Figure~\ref{fig:super_ptolemy}, the arrows on the edges labelled $e$ and $f$ represent the choice of orientation. In the figure,
the orientations of the edges around the boundary of the quadrilateral are not indicated. Three of the four edges are unchanged
in the super Ptolemy transformation, and only the orientation of the edge labelled~$b$ is changed, see Figure~\ref{fig:flip_spin}. The sign in equation~\eqref{eq:mu_mutation_neg} is determined by the relative positions of the triangular faces with respect to the chosen orientation.

\begin{figure}[h!]

\centering

\begin{tikzpicture}[scale=0.7, baseline, thick]

\draw (0,0)--(3,0)--(60:3)--cycle;

\draw (0,0)--(3,0)--(-60:3)--cycle;

\draw (0,0)--node {\midarrow} (3,0);

\draw node[above] at (70:1.5){$\epsilon_a$};

\draw node[above] at (30:2.8){$\epsilon_b$};

\draw node[below] at (-30:2.8){$\epsilon_c$};

\draw node[below=-0.1] at (-70:1.5){$\epsilon_d$};


\draw node[left] at (0,0) {};

\draw node[above] at (60:3) {};

\draw node[right] at (3,0) {};

\draw node[below] at (-60:3) {};

\draw node at (1.5,1){$\theta$};

\draw node at (1.5,-1){$\sigma$};

\end{tikzpicture}
\begin{tikzpicture}[baseline]

\draw[->, thick](0,0)--(1,0);

\node[above]  at (0.5,0) {};

\end{tikzpicture}
\begin{tikzpicture}[scale=0.7, baseline, thick,every node/.style={sloped,allow upside down}]
\draw (0,0)--(60:3)--(-60:3)--cycle;
\draw (3,0)--(60:3)--(-60:3)--cycle;

\draw node[above] at (70:1.5){$\epsilon_a$};

\draw node[above] at (30:2.8){$-\epsilon_b$};

\draw node[below] at (-30:2.8){$\epsilon_c$};

\draw node[below=-0.1] at (-70:1.5){$\epsilon_d$};

\draw (1.5,-2) --node {\midarrow} (1.5,2);


\draw node[left] at (0,0) {};

\draw node[above] at (60:3) {};

\draw node[right] at (3,0) {};

\draw node[below] at (-60:3) {};

\draw node at (0.8,0){$\theta'$};

\draw node at (2.2,0){$\sigma'$};

\end{tikzpicture}
\begin{tikzpicture}[baseline]

\draw[->, thick](0,0)--(1,0);

\node[above]  at (0.5,0) {};

\end{tikzpicture}
\begin{tikzpicture}[scale=0.7, baseline, thick]

\draw (0,0)--(3,0)--(60:3)--cycle;

\draw (0,0)--(3,0)--(-60:3)--cycle;

\draw (3,0) --node {\reflectbox{\midarrow}} (0,0);

\draw node[above] at (70:1.5){$-\epsilon_a$};

\draw node[above] at (30:2.8){$-\epsilon_b$};

\draw node[below] at (-30:2.8){$\epsilon_c$};

\draw node[below=-0.1] at (-70:1.5){$\epsilon_d$};


\draw node[left] at (0,0) {};

\draw node[above] at (60:3) {};

\draw node[right] at (3,0) {};

\draw node[below] at (-60:3) {};

\draw node at (1.5,1){$-\theta$};

\draw node at (1.5,-1){$\sigma$};

\end{tikzpicture}
\caption{Flip effect on spin structures. Here $\epsilon_x$ denotes the orientation of an edge $x$.}
\label{fig:flip_spin}
\end{figure}
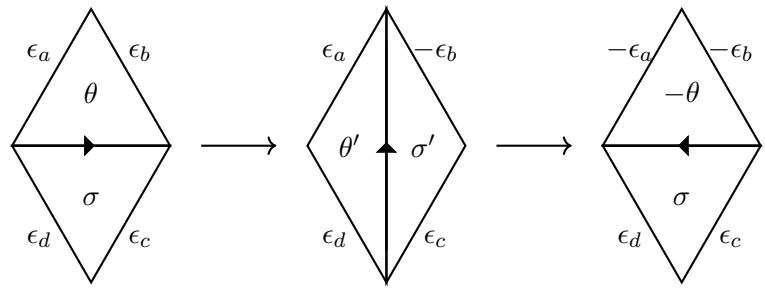

As illustrated in Figure~\ref{fig:flip_spin}, the super Ptolemy relation is \emph{not} an involution. Performing a flip
twice results in reversing the orientations around the top triangle. This leads us to the following observation that $\mu$-invariants are well-defined only up to sign. Consequently, in the main results of our paper regarding $\mu$-invariants (Theorems~\ref{thm:mu_single_fan} and~\ref{thm:proof_zigzag}), we must specify a mutation sequence to obtain a formula for the $\mu$-invariants.

\begin{Remark} \label{rem:mu-invariants}\rm
The above equivalence relation, i.e., as used in Proposition~\ref{prop:equiv_rel}, guarantees that the result after flipping twice represents the same spin structure, but algebraically it has the effect of negating the $\mu$-invariant of that triangle
($\theta \mapsto -\theta$ in the figure). This means that the specific $\mu$-invariants are not a feature of the triangulation
and spin structure alone, but also the choice of representative orientation. Choosing a different orientation corresponds to
changing the sign of some of the $\mu$-invariants.
\end{Remark}

\begin{Remark} \label{rem:lambda}\rm
Unlike the $\mu$-invariants, the expressions of $\lambda$-lengths in terms of an initial triangulation (as will be described more fully in Corollary~\ref{cor:laurent}) are independent of the orientation of the arc as part of a spin structure, and of the flip sequence used to obtain a triangulation containing that arc. This is proven in \cite{pz_19} in the case of surfaces without boundaries. The fact that $\lambda$-lengths are well-defined can also be seen as a consequence of the well-definedness of super-freizes, as we will describe in Sections~\ref{sec:super-friezes} and~\ref{sec:subfrieze}, based on \cite{M-GOT15,pz_19}. We provide a more direct proof in Appendix~\ref{appendix}.
\end{Remark}

\section[Super T-paths]{Super $\boldsymbol{T}$-paths}
Let $P$ be an $(n+3)$-gon (a disk with $n+3$ marked points on the boundary), and $T = \{t_1,\dots,t_{2n+3}\}$ the set of arcs in some triangulation.
We will denote by $V_0$ the set of vertices (marked points).
Let~$a$ and~$b$ be two non-adjacent vertices on the boundary and let $(a,b)$ be the arc that connects~$a$ and~$b$.

\subsection{Fans of a triangulation and their centers}

We call a triangulation a \emph{fan} if all the internal diagonals meet at a common vertex.
We will define a canonical way to break any triangulation $T$ of $P$ into smaller polygons with fan triangulations.
For this purpose, certain vertices in $P$ will be distinguished as \emph{centers of fans}.

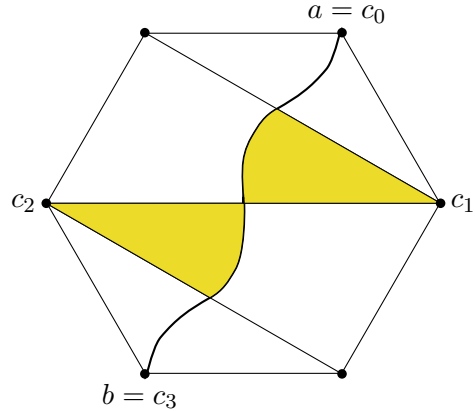
\begin{figure}[h]
	\centering
	\begin{tikzpicture}[scale=1.1]
	\tikzstyle{every path}=[draw] 
		\path
    node[
      regular polygon,
      regular polygon sides=6,
      draw,
      inner sep=1.6cm,
    ] (hexagon) {}
    %
    (hexagon.corner 1) node[above] {$\ a=c_0$}
    (hexagon.corner 3) node[left] {$c_2$}
    (hexagon.corner 4) node[below] {$b=c_3\ $}
    (hexagon.corner 6) node[right] {$c_1$}

  ;
  \coordinate (m1) at (1,1.62);
  \coordinate (m2) at (0.1,0.8);
  \coordinate (m3) at (-0.1,-0.8);
  \coordinate (m4) at (-1,-1.62);
  
  \draw [name path = m ,thin,opacity=0] (hexagon.corner 1) to [out=90,in=50] (m1) to [out=-180+50,in=60] (m2) to [out=-180+60,in=60] (m3) to [out=-180+60,in=50] (m4) to [out=-180+50,in=-180+90](hexagon.corner 4);%
  
  \draw [name path = L1] (hexagon.corner 6) to (hexagon.corner 2);
  \draw [name path = L2] (hexagon.corner 6) to (hexagon.corner 3);
  \draw [name path = L3] (hexagon.corner 5) to (hexagon.corner 3);
  
  \path [name intersections={of= m and L1,by=X1}];
  \path [name intersections={of= m and L2,by=X2}];
  \path [name intersections={of= m and L3,by=X3}];
  
  \draw [name path = m ,thick, black] (hexagon.corner 1) to [out=90,in=50] (m1) to [out=-180+50,in=30] (X1) to [out=-180+30,in=60] (m2) to [out=-180+60,in=90] (X2) to [out=90,in=60] (m3) to [out=-180+60,in=30] (X3) to [out=-180+30,in=50] (m4) to [out=-180+50,in=-180+90](hexagon.corner 4);

  \draw [fill=yellow!90!black] (X1) to [out=-180+30,in=60] (m2) to [out=-180+60,in=90] (X2) to (hexagon.corner 6) to (X1);
  
  \draw[fill=yellow!90!black](X2) to [out=90,in=60] (m3) to [out=-180+60,in=30] (X3) to (hexagon.corner 3)--(X2);

  

  \draw (hexagon.corner 1) node [fill,circle,scale=0.35] {};
  \draw (hexagon.corner 2) node [fill,circle,scale=0.35] {};
  \draw (hexagon.corner 3) node [fill,circle,scale=0.35] {};
  \draw (hexagon.corner 4) node [fill,circle,scale=0.35] {};
  \draw (hexagon.corner 5) node [fill,circle,scale=0.35] {};
  \draw (hexagon.corner 6) node [fill,circle,scale=0.35] {};
  
  \end{tikzpicture}
		\caption{Centers of fan segments.}
		\label{fig:fan_centers}
	\end{figure}

Let $P$ be a polygon and $T$ a triangulation. Following Remark~\ref{rem:subtriang}, without loss of generality, we may assume that
$(a,b)$ crosses all internal diagonals of $T$.
Consider the intersection of $(a,b)$ with triangles in $T$ which do not contain $a$ or $b$.
These intersections create small triangles (colored yellow in Figure~\ref{fig:fan_centers}) whose vertices in $P$ we call \emph{fan centers}.
We set $a=c_0$ and $b=c_{N+1}$ and as a convention we will name these centers $c_1,\dots,c_N$ such that
\begin{enumerate}\itemsep=0pt
 \item[1)] for $1\leq i\leq N-1$, the edge $(c_i,c_{i+1})$ is in $T$ which crosses $(a,b)$,
 \item[2)] the intersection $(c_i,c_{i+1}) \cap (a,b)$ is closer to $a$ than $(c_j,c_{j+1}) \cap (a,b)$ if $i<j$.
\end{enumerate}

Now the edges $(c_i,c_{i+1})$ naturally break the triangulation $T$ into $N$ smaller polygons, each of which
comes with an induced fan triangulation. Let $F_i$ denote the subgraph of $T$ bounded by $c_{i-1}$, $c_{i}$ and $c_{i+1}$,
which are called the \emph{fan segments} of $T$. We say that $c_i$ is the \emph{center} of $F_i$.
 See Figure~\ref{fig:bigger_auxiliary_graph} for an illustration, where the fan segments are indicated by different colors.

\subsection{The auxiliary graph}

We shall now define an auxiliary graph associated to $(T,a,b)$, which will be used to define the super $T$-paths from $a$ to $b$.
	
For a triangulation $T$ and a pair of vertices $a$ and $b$, we define the graph $\Gamma_T^{a,b}$ to be the graph of the triangulation $T$ with some additional vertices and edges.
\begin{enumerate}\itemsep=0pt
 \item For each face of the triangulation $T$, we place an internal vertex, which lies on the arc $(a,b)$.
 We denote the internal vertices $V_1=\{\theta_1,\dots,\theta_{n+1}\}$, such that~$\theta_i$ is closer to $a$ than~$\theta_j$ if and only if $i<j$.

 \item For each face of~$T$, we add an edge $\sigma_i := (\theta_i, c_j)$ connecting the internal vertex $\theta_i$ to the center of the fan segment which contains~$\theta_i$.
 We denote by~$\sigma$ the set of all such edges.

 \item For each $\theta_i$ and $\theta_j$ with $i<j$, we add an edge connecting $\theta_i$ and $\theta_j$.
 We denote the collection of these edges as $\tau = \{\tau_{ij}\colon i<j\}$. For simplicity the $\tau$-edges are drawn to be overlapping.
\end{enumerate}
See Figure~\ref{fig:auxiliary_graph} for example.
\begin{figure}[h]
	\centering
	\begin{tikzpicture}[scale=1]
	\tikzstyle{every path}=[draw] 
		\path
    node[
      regular polygon,
      regular polygon sides=6,
      draw,
      inner sep=1.6cm,
    ] (hexagon) {}
    %
    (hexagon.corner 1) node[above] {$\ \ a$}
    (hexagon.corner 2) node[above] {$x\ \ $}
    (hexagon.corner 3) node[left ] {$c_2$}
    (hexagon.corner 4) node[below] {$b\ \ $}
    (hexagon.corner 5) node[below] {$\ \ y$}
    (hexagon.corner 6) node[right] {$c_1$}

  ;
  \coordinate (m1) at (1,1.62);
  \coordinate (m2) at (0.1,0.8);
  \coordinate (m3) at (-0.1,-0.8);
  \coordinate (m4) at (-1,-1.62);
  
  \draw [name path = m ,opacity=1] (m1) to [out=-180+50,in=60] (m2) to [out=-180+60,in=60] (m3) to [out=-180+60,in=50] (m4);%
  
  \draw [name path = L1] (hexagon.corner 6) to (hexagon.corner 2);
  \draw [name path = L2] (hexagon.corner 6) to (hexagon.corner 3);
  \draw [name path = L3] (hexagon.corner 5) to (hexagon.corner 3);

  \draw (m1) node [fill,circle,scale=0.3] {};
  \draw (m2) node [fill,circle,scale=0.3] {};
  \draw (m3) node [fill,circle,scale=0.3] {};
  \draw (m4) node [fill,circle,scale=0.3] {};
  
  \node[left] at (m1) {$\theta_1$};
  \node[left] at (m2) {$\theta_2$};
  \node[right] at (m3) {$\theta_3$}; 
  \node[left] at (m4) {$\theta_4$};
  
  \node[right] at (1.2,1.4) {$\sigma_1$};
  \node[right] at (0.4,0.4) {$\sigma_2$};
  \node[right] at (-1,-0.4) {$\sigma_3$};
  \node[right] at (-1.7,-1) {$\sigma_4$};

  \draw (hexagon.corner 1) node [fill,circle,scale=0.35] {};
  \draw (hexagon.corner 2) node [fill,circle,scale=0.35] {};
  \draw (hexagon.corner 3) node [fill,circle,scale=0.35] {};
  \draw (hexagon.corner 4) node [fill,circle,scale=0.35] {};
  \draw (hexagon.corner 5) node [fill,circle,scale=0.35] {};
  \draw (hexagon.corner 6) node [fill,circle,scale=0.35] {};
  
  \draw [] (m1)--(hexagon.corner 6);
  \draw [] (m2)-- (hexagon.corner 6);
  \draw [] (m3)--(hexagon.corner 3);
  \draw [] (m4)-- (hexagon.corner 3);

	\end{tikzpicture}
		\caption{The auxiliary graph $\Gamma_T^{a,b}$.}
		\label{fig:auxiliary_graph}
	\end{figure}
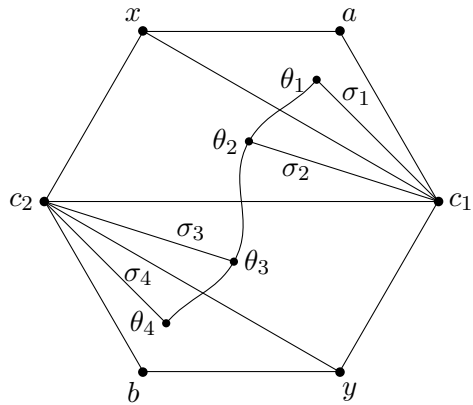

\begin{Remark}\rm
 Note that the arc $(a,b)$ divides the first and last triangle into two triangles, as opposed to into a triangle and a quadrilateral
 like the case of every other triangle of~$T$. Consequently, the convention of the yellow coloring as in Figure~\ref{fig:fan_centers}
 can be extended to the first and last triangles of $T$ in multiple ways.
 Thus we may define the first and last $\sigma$-edge in a~different way: $\sigma_1=(\theta_1,x)$ and $\sigma_4=(\theta_4,y)$ in Figure~\ref{fig:auxiliary_graph},
 in which case the (not yet defined) super $T$-paths produces the same weight. We make this choice for sake of consistency of doing induction.
 This means that when $P$ is a quadrilateral, we can view its triangulation~$T$ as either a single fan or having 2 fans, and define the auxiliary graph in $4$ different ways.
\end{Remark}
In Figure~\ref{fig:bigger_auxiliary_graph}, we give another example of constructing the auxiliary graph.

\begin{figure}[h!]
\centering
 \begin{tikzpicture}[decoration={
    markings,
    mark=at position 0.5 with {\arrow{>}}},
    scale=0.8
    ] 

\tikzstyle{every path}=[draw] 
		\path
    node[
      regular polygon,
      regular polygon sides=12,
      draw,
      inner sep=2cm,
    ] (hexagon) {}

    (hexagon.corner 1) node[above] {$c_0$}
    (hexagon.corner 2) node[above] {}
    (hexagon.corner 3) node[left] {}
    (hexagon.corner 4) node[left] {$c_2$}
    (hexagon.corner 5) node[below] {}
    (hexagon.corner 6) node[right] {}
    (hexagon.corner 7) node[below] {$c_4$}
    (hexagon.corner 8) node[below] {$c_5$}
    (hexagon.corner 9) node[left] {}
    (hexagon.corner 10) node[right] {$c_3$}
    (hexagon.corner 11) node[below] {}
    (hexagon.corner 12) node[right] {$c_1$}
;

\coordinate (m 1) at  ($0.38*(hexagon.corner 2)+0.38*(hexagon.corner 1)+0.24*(hexagon.corner 12)$);
\coordinate (m 2) at  ($0.28*(hexagon.corner 2)+0.28*(hexagon.corner 3)+0.44*(hexagon.corner 12)$);
\coordinate (m 3) at  ($0.245*(hexagon.corner 4)+0.245*(hexagon.corner 3)+0.51*(hexagon.corner 12)$);
\coordinate (m 4) at  ($0.52*(hexagon.corner 4)+0.24*(hexagon.corner 11)+0.24*(hexagon.corner 12)$);
\coordinate (m 5) at  ($0.56*(hexagon.corner 4)+0.22*(hexagon.corner 11)+0.22*(hexagon.corner 10)$);
\coordinate (m 6) at  ($0.27*(hexagon.corner 4)+0.27*(hexagon.corner 5)+0.46*(hexagon.corner 10)$);
\coordinate (m 7) at  ($0.23*(hexagon.corner 6)+0.23*(hexagon.corner 5)+0.44*(hexagon.corner 10)$);
\coordinate (m 8) at  ($0.26*(hexagon.corner 6)+0.26*(hexagon.corner 7)+0.48*(hexagon.corner 10)$);
\coordinate (m 9) at  ($0.25*(hexagon.corner 9)+0.5*(hexagon.corner 7)+0.25*(hexagon.corner 10)$);
\coordinate (m 10) at  ($0.4*(hexagon.corner 9)+0.3*(hexagon.corner 7)+0.3*(hexagon.corner 8)$);


\draw [name path = M ,opacity=0] (hexagon.corner 1) to [out=15,in=40] (m 1) to [out=-180+40,in=55] (m 2) to [out=-180+55,in=75] (m 3) to [out=-180+75,in=80] (m 4) to [out=-180+80,in=90] (m 5) to [out=-180+90,in=110] (m 6) to [out=-180+110,in=125] (m 7) to [out=-180+125,in=125] (m 8) to [out=-180+125,in=105] (m 9) to [out=-180+105,in=80] (m 10) to [out=-180+80,in=80] (hexagon.corner 8);%


\draw[name path = L 2] (hexagon.corner 12)--(hexagon.corner 4);
\draw[name path = L 1] (hexagon.corner 12)--(hexagon.corner 2);
\draw[] (hexagon.corner 12)--(hexagon.corner 3);
\draw[name path = L 3](hexagon.corner 4)--(hexagon.corner 10);
\draw[](hexagon.corner 4)--(hexagon.corner 11);
\draw[name path = L 4](hexagon.corner 10)--(hexagon.corner 7);
\draw[](hexagon.corner 10)--(hexagon.corner 5);
\draw[](hexagon.corner 10)--(hexagon.corner 6);
\draw[name path = L 5](hexagon.corner 7)--(hexagon.corner 9);

\path [name intersections={of= M and L 1,by=X 1}];
\path [name intersections={of= M and L 2,by=X 2}];
\path [name intersections={of= M and L 3,by=X 3}];
\path [name intersections={of= M and L 4,by=X 4}];
\path [name intersections={of= M and L 5,by=X 5}];

\draw [fill=yellow!90!black, draw=none] (X 1) to (hexagon.corner 12) to (X 2) to [out=180-100,in=180+42] (X 1);

\draw [fill=yellow!90!black, draw=none] (X 2) to (hexagon.corner 4) to (X 3) to  [out=180-85,in=180+79] (X 2);
\draw [fill=yellow!90!black, draw=none] (X 3) to (hexagon.corner 10) to (X 4) to  [out=180-63,in=180+96] (X 3);
\draw [fill=yellow!90!black, draw=none] (X 4) to (hexagon.corner 7) to (X 5) to  [out=180-100,in=180+110] (X 4);

\draw[name path = L 2] (hexagon.corner 12)--(hexagon.corner 4);
\draw[name path = L 1] (hexagon.corner 12)--(hexagon.corner 2);
\draw[] (hexagon.corner 12)--(hexagon.corner 3);
\draw[name path = L 3](hexagon.corner 4)--(hexagon.corner 10);
\draw[](hexagon.corner 4)--(hexagon.corner 11);
\draw[name path = L 4](hexagon.corner 10)--(hexagon.corner 7);
\draw[](hexagon.corner 10)--(hexagon.corner 5);
\draw[](hexagon.corner 10)--(hexagon.corner 6);
\draw[name path = L 5](hexagon.corner 7)--(hexagon.corner 9);

\draw [opacity=1, thick] (hexagon.corner 1) to [out=180+40,in=40] (m 1) to [out=-180+40,in=55] (m 2) to [out=-180+55,in=75] (m 3) to [out=-180+75,in=80] (m 4) to [out=-180+80,in=90] (m 5) to [out=-180+90,in=110] (m 6) to [out=-180+110,in=125] (m 7) to [out=-180+125,in=125] (m 8) to [out=-180+125,in=105] (m 9) to [out=-180+105,in=80] (m 10) to [out=-180+80,in=80] (hexagon.corner 8);%



\end{tikzpicture}\ \ \begin{tikzpicture}[decoration={
    markings,
    mark=at position 0.5 with {\arrow{>}}},scale=0.8
    ] 

\tikzstyle{every path}=[draw] 
		\path
    node[
      regular polygon,
      regular polygon sides=12,
      draw,
      inner sep=2cm,
    ] (hexagon) {}
    %
    (hexagon.corner 1) node[above] {$c_0$}
    (hexagon.corner 2) node[above] {}
    (hexagon.corner 3) node[left] {}
    (hexagon.corner 4) node[left] {$c_2$}
    (hexagon.corner 5) node[below] {}
    (hexagon.corner 6) node[right] {}
    (hexagon.corner 7) node[below] {$c_4$}
    (hexagon.corner 8) node[below] {$c_5$}
    (hexagon.corner 9) node[left] {}
    (hexagon.corner 10) node[right] {$c_3$}
    (hexagon.corner 11) node[below] {}
    (hexagon.corner 12) node[right] {$c_1$}
;

\coordinate (m 1) at  ($0.38*(hexagon.corner 2)+0.38*(hexagon.corner 1)+0.24*(hexagon.corner 12)$);
\coordinate (m 2) at  ($0.28*(hexagon.corner 2)+0.28*(hexagon.corner 3)+0.44*(hexagon.corner 12)$);
\coordinate (m 3) at  ($0.24*(hexagon.corner 4)+0.24*(hexagon.corner 3)+0.52*(hexagon.corner 12)$);
\coordinate (m 4) at  ($0.52*(hexagon.corner 4)+0.24*(hexagon.corner 11)+0.24*(hexagon.corner 12)$);
\coordinate (m 5) at  ($0.56*(hexagon.corner 4)+0.22*(hexagon.corner 11)+0.22*(hexagon.corner 10)$);
\coordinate (m 6) at  ($0.27*(hexagon.corner 4)+0.27*(hexagon.corner 5)+0.46*(hexagon.corner 10)$);
\coordinate (m 7) at  ($0.23*(hexagon.corner 6)+0.23*(hexagon.corner 5)+0.44*(hexagon.corner 10)$);
\coordinate (m 8) at  ($0.26*(hexagon.corner 6)+0.26*(hexagon.corner 7)+0.48*(hexagon.corner 10)$);
\coordinate (m 9) at  ($0.25*(hexagon.corner 9)+0.5*(hexagon.corner 7)+0.25*(hexagon.corner 10)$);
\coordinate (m 10) at  ($0.4*(hexagon.corner 9)+0.3*(hexagon.corner 7)+0.3*(hexagon.corner 8)$);

\foreach \x in {1,...,10}{
\draw (m \x) node [fill,circle,scale=0.2] {};}

\foreach \x in {1,2,...,12}{
\draw (hexagon.corner \x) node [fill,circle,scale=0.2] {};}

\draw[] (hexagon.corner 12)--(hexagon.corner 4);
\draw[] (hexagon.corner 12)--(hexagon.corner 2);
\draw[] (hexagon.corner 12)--(hexagon.corner 3);

\draw[](hexagon.corner 4)--(hexagon.corner 10);
\draw[](hexagon.corner 4)--(hexagon.corner 11);

\draw[](hexagon.corner 10)--(hexagon.corner 7);
\draw[](hexagon.corner 10)--(hexagon.corner 5);
\draw[](hexagon.corner 10)--(hexagon.corner 6);

\draw[](hexagon.corner 7)--(hexagon.corner 9);

\draw[fill=yellow!50!white, nearly transparent] (hexagon.corner 12)--(hexagon.corner 4)--(hexagon.corner 3)--(hexagon.corner 2)--(hexagon.corner 1)--cycle;
\draw[fill=red!50!white, nearly transparent] (hexagon.corner 12)--(hexagon.corner 4)--(hexagon.corner 10)--(hexagon.corner 11)--cycle;

\draw[fill=green!50!white, nearly transparent] (hexagon.corner 10)-- (hexagon.corner 4)--(hexagon.corner 5)--(hexagon.corner 6)--(hexagon.corner 7)--cycle;
\draw[fill=blue!50!white, nearly transparent] (hexagon.corner 10)--(hexagon.corner 7)--(hexagon.corner 8)--(hexagon.corner 9)--cycle;   

\draw [opacity=1]  (m 1) to [out=-180+40,in=55] (m 2) to [out=-180+55,in=75] (m 3) to [out=-180+75,in=80] (m 4) to [out=-180+80,in=90] (m 5) to [out=-180+90,in=110] (m 6) to [out=-180+110,in=125] (m 7) to [out=-180+125,in=125] (m 8) to [out=-180+125,in=105] (m 9) to [out=-180+105,in=80] (m 10);%

\foreach \i in {1,2,3}{
\draw [] (hexagon.corner 12)--(m \i);}

\foreach \i in {4,5}{
\draw [] (hexagon.corner 4)--(m \i);}

\foreach \i in {6,7,8}{
\draw [] (hexagon.corner 10)--(m \i);}

\foreach \i in {9,10}{
\draw [] (hexagon.corner 7)--(m \i);}

\end{tikzpicture}
\caption{Left: yellow shading indicates the fan centers. Right: the auxiliary graph where different fans are indicated by different colors.}
\label{fig:bigger_auxiliary_graph}
\end{figure}
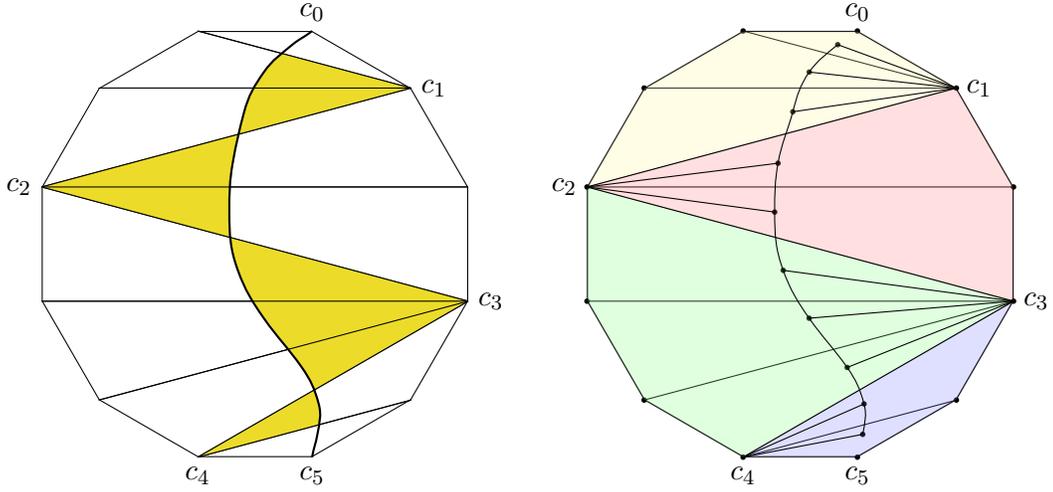

\subsection[Super T-paths]{Super $\boldsymbol{T}$-paths}

Now we define the super $T$-paths from $a$ to $b$ to be paths on edges of the auxiliary graph $\Gamma_T^{a,b}$ satisfying certain axioms.

\begin{Definition}[super $T$-paths]
\label{def:super-T}
A \emph{super $T$-path} $t$ from $a$ to $b$ is a sequence
 \[t=(a_0,a_1,\dots,a_{\ell(t)} \,|\, t_1,t_2,\dots,t_{\ell(t)})\]
 such that
 \begin{enumerate}\itemsep=0pt
 \item[\rm (T1)] $a=a_0,a_1,\dots,a_{\ell(t)}=b\in V_0\cup V_1$ are vertices on $\Gamma_T^{a,b}$,
 \item[\rm (T2)] for each $1\leq i\leq \ell(t)$, $t_i$ is an edge in $\Gamma_T^{a,b}$ connecting $a_{i-1}$ and $a_i$,
 \item[\rm (T3)] $t_i\neq t_j$ if $i\neq j$,
 \item[\rm (T4)] $\ell(t)$ is odd,
 \item[\rm (T5)] $t_i$ crosses $(a,b)$ if $i$ is even. The $\sigma$-edges are considered to cross $(a,b)$,
 \item[\rm (T6)] $t_i\in\sigma$ only if $i$ is even, $t_i\in\tau$ only if $i$ is odd,
 \item[\rm (T7)] if $i<j$ and both $t_i$ and $t_j$ cross the arc $(a,b)$, then the intersection $t_i\cap (a,b)$ is closer
 to the vertex $a$ than the intersection $t_j\cap (a,b)$.
 \end{enumerate}

 We let $\mathcal T_{a,b}$ denote the set of super $T$-paths from $a$ to $b$.
 Furthermore, let $\mathcal T_{a,b}^0$ be the set of super $T$-paths from $a$ to $b$ which do not use $\sigma$ or $\tau$ edges.
 We naturally identify $\mathcal{T}_{a,b}^0$ with $T_{a,b}$, the set of ordinary $T$-paths from $a$ to $b$.
 We also define $\mathcal T_{a,b}^1 := \mathcal T_{a,b} - \mathcal T_{a,b}^0$ as the set of super $T$-paths which do not
 correspond to ordinary $T$-paths.
\end{Definition}

An immediate observation is that every super $T$-path must have an even number of $\sigma$-edges.
More specifically, they always appear as a sequence of $\sigma$-edge, $\tau$-edge, and $\sigma$-edge.
Here $\tau$- stands for \emph{teleportation}: instead of following along an edge of the triangulation $T$, a $\tau$-step teleports
from one internal vertex to another. We call a subsequence of a super $T$-path of the form $(\dots,\theta_i,\theta_j,\dots|\dots,\sigma_i,\tau_{ij},\sigma_j,\dots)$ a \emph{super step}.
In other words, a super $T$-path is a~concatenation of certain ordinary $T$-paths and super steps.

\begin{Example} \label{example:14some}\rm
 Figure~\ref{fig:examplepath} illustrates several examples of super $T$-paths from $1$ to $4$.
 Odd-numbered edges are colored red and even-numbered edges are colored blue.

\begin{figure}[t]

\centering
\begin{tabular}{c}
\begin{tikzpicture}[scale=0.9]
	\tikzstyle{every path}=[draw]
		\path
    node[
      regular polygon,
      regular polygon sides=6,
      draw,
      inner sep=1.6cm,
    ] (hexagon) {}
    %
    (hexagon.corner 1) node[above] {$1$}
    (hexagon.corner 2) node[above] {$2$}
    (hexagon.corner 3) node[left] {$3$}
    (hexagon.corner 4) node[below] {$4 $}
    (hexagon.corner 5) node[below] {$5$}
    (hexagon.corner 6) node[right] {$6$}

  ;
  \coordinate (m1) at (1,1.62);
  \coordinate (m2) at (0.1,0.8);
  \coordinate (m3) at (-0.1,-0.8);
  \coordinate (m4) at (-1,-1.62);

  \draw [name path = m ,opacity=1] (m1) to [out=-180+50,in=60] (m2) to [out=-180+60,in=60] (m3) to [out=-180+60,in=50] (m4);%

  \draw [name path = L1] (hexagon.corner 6) to (hexagon.corner 2);
  \draw [name path = L2] (hexagon.corner 6) to (hexagon.corner 3);
  \draw [name path = L3] (hexagon.corner 5) to (hexagon.corner 3);

  \draw (m1) node [fill,circle,scale=0.3] {};
  \draw (m2) node [fill,circle,scale=0.3] {};
  \draw (m3) node [fill,circle,scale=0.3] {};
  \draw (m4) node [fill,circle,scale=0.3] {};

  \node[left] at (m1) {$\theta_1$};
  \node[left] at (m2) {$\theta_2$};
  \node[right] at (m3) {$\theta_3$};
  \node[left] at (m4) {$\theta_4$};

  \node at(2.2,1.5 ){$x_1$};
  \node at(0,2.5 ){$x_2$};
  \node at(-2.2,1.5 ){$x_3$};
  \node at(-2.2,-1.5 ){$x_4$};
  \node at(2.2,-1.5 ){$x_6$};
  \node at(0,-2.5 ){$x_5$};
  \node at(-0.4,1.5 ){$x_7$};
  \node at(-0.4,-1.5 ){$x_9$};
  \node at(-0.7,0.27 ){$x_8$};

  \draw (hexagon.corner 1) node [fill,circle,scale=0.35] {};
  \draw (hexagon.corner 2) node [fill,circle,scale=0.35] {};
  \draw (hexagon.corner 3) node [fill,circle,scale=0.35] {};
  \draw (hexagon.corner 4) node [fill,circle,scale=0.35] {};
  \draw (hexagon.corner 5) node [fill,circle,scale=0.35] {};
  \draw (hexagon.corner 6) node [fill,circle,scale=0.35] {};

  \draw [] (m1)--(hexagon.corner 6);
  \draw [] (m2)-- (hexagon.corner 6);
  \draw [] (m3)--(hexagon.corner 3);
  \draw [] (m4)-- (hexagon.corner 3);

  \draw [ultra thick, red] (hexagon.corner 1) to (hexagon.corner 6) ;
  \draw [ultra thick,blue] (hexagon.corner 6) to (m1);
  \draw [ultra thick,red]  (m1) to [out=-180+50,in=60] (m2) to (hexagon.corner 6);
   \draw [ultra thick,blue] (m2) to (hexagon.corner 6);
  \draw [ultra thick, red] (hexagon.corner 6) to (hexagon.corner 3) ;
  \draw [ultra thick, blue] (hexagon.corner 3) to (m3);
  \draw [ultra thick,red] (m3) to [out=-180+60,in=50] (m4) to (hexagon.corner 3);
  \draw [ultra thick,blue]  (m4) to (hexagon.corner 3);
  \draw [ultra thick, red] (hexagon.corner 4) to (hexagon.corner 3) ;

	\end{tikzpicture}

  \\
  \( t=(1,6,\theta_1,\theta_2,6,3,\theta_3,\theta_4,3,4\,|\,x_1,\sigma_1,\tau_{12},\sigma_2,x_8,\sigma_3,\tau_{34},\sigma_4,x_4)\)
\end{tabular}
\begin{tabular}{c}
	
 \begin{tikzpicture}[scale=0.9]
	\tikzstyle{every path}=[draw]
		\path
    node[
      regular polygon,
      regular polygon sides=6,
      draw,
      inner sep=1.6cm,
    ] (hexagon) {}
    %
    (hexagon.corner 1) node[above] {$1$}
    (hexagon.corner 2) node[above] {$2$}
    (hexagon.corner 3) node[left] {$3$}
    (hexagon.corner 4) node[below] {$4$}
    (hexagon.corner 5) node[below] {$5$}
    (hexagon.corner 6) node[right] {$6$}

  ;
  \coordinate (m1) at (1,1.62);
  \coordinate (m2) at (0.1,0.8);
  \coordinate (m3) at (-0.1,-0.8);
  \coordinate (m4) at (-1,-1.62);
  \node at(2.2,1.5 ){$x_1$};
  \node at(0,2.5 ){$x_2$};
  \node at(-2.2,1.5 ){$x_3$};
  \node at(-2.2,-1.5 ){$x_4$};
  \node at(2.2,-1.5 ){$x_6$};
  \node at(0,-2.5 ){$x_5$};
  \node at(-0.4,1.5 ){$x_7$};
  \node at(-0.4,-1.5 ){$x_9$};
  \node at(-0.7,0.27 ){$x_8$};

  \draw [name path = m ,opacity=1] (m1) to [out=-180+50,in=60] (m2) to [out=-180+60,in=60] (m3) to [out=-180+60,in=50] (m4);%

  \draw [name path = L1] (hexagon.corner 6) to (hexagon.corner 2);
  \draw [name path = L2] (hexagon.corner 6) to (hexagon.corner 3);
  \draw [name path = L3] (hexagon.corner 5) to (hexagon.corner 3);

  \draw (m1) node [fill,circle,scale=0.3] {};
  \draw (m2) node [fill,circle,scale=0.3] {};
  \draw (m3) node [fill,circle,scale=0.3] {};
  \draw (m4) node [fill,circle,scale=0.3] {};

  \node[left] at (m1) {$\theta_1$};
  \node[left] at (m2) {$\theta_2$};
  \node[right] at (m3) {$\theta_3$};
  \node[left] at (m4) {$\theta_4$};

  \draw (hexagon.corner 1) node [fill,circle,scale=0.35] {};
  \draw (hexagon.corner 2) node [fill,circle,scale=0.35] {};
  \draw (hexagon.corner 3) node [fill,circle,scale=0.35] {};
  \draw (hexagon.corner 4) node [fill,circle,scale=0.35] {};
  \draw (hexagon.corner 5) node [fill,circle,scale=0.35] {};
  \draw (hexagon.corner 6) node [fill,circle,scale=0.35] {};

  \draw [] (m1)--(hexagon.corner 6);
  \draw [] (m2)-- (hexagon.corner 6);
  \draw [] (m3)--(hexagon.corner 3);
  \draw [] (m4)-- (hexagon.corner 3);

  \draw [ultra thick, red] (hexagon.corner 1) to (hexagon.corner 6) ;
  \draw [ultra thick, blue] (hexagon.corner 6) to (m1) ;
  \draw [ultra thick, red] (m1) to [out=-180+50,in=60] (m2) to (hexagon.corner 6);
  \draw [ultra thick, blue] (m2) to (hexagon.corner 6);
  \draw [ultra thick, red] (hexagon.corner 6) to (hexagon.corner 3) ;
  \draw [ultra thick, blue] (hexagon.corner 3) to (hexagon.corner 5) ;

  \draw [ultra thick, red] (hexagon.corner 4) to (hexagon.corner 5) ;

	\end{tikzpicture}
	\\
	$t=(1,6,\theta_1,\theta_2,6,3,5,4\,|\,x_1,\sigma_1,\tau_{12},\sigma_2,x_8,x_9,x_5)$
	\end{tabular}
	\begin{tabular}{c}
	\begin{tikzpicture}[scale=0.9]
	\tikzstyle{every path}=[draw]
		\path
    node[
      regular polygon,
      regular polygon sides=6,
      draw,
      inner sep=1.6cm,
    ] (hexagon) {}
    %
    (hexagon.corner 1) node[above] {$1$}
    (hexagon.corner 2) node[above] {$2$}
    (hexagon.corner 3) node[left] {$3$}
    (hexagon.corner 4) node[below] {$4$}
    (hexagon.corner 5) node[below] {$5$}
    (hexagon.corner 6) node[right] {$6$}

  ;
  \coordinate (m1) at (1,1.62);
  \coordinate (m2) at (0.1,0.8);
  \coordinate (m3) at (-0.1,-0.8);
  \coordinate (m4) at (-1,-1.62);

  \node at(2.2,1.5 ){$x_1$};
  \node at(0,2.5 ){$x_2$};
  \node at(-2.2,1.5 ){$x_3$};
  \node at(-2.2,-1.5 ){$x_4$};
  \node at(2.2,-1.5 ){$x_6$};
  \node at(0,-2.5 ){$x_5$};
  \node at(-0.4,1.5 ){$x_7$};
  \node at(-0.4,-1.5 ){$x_9$};
  \node at(-0.7,0.27 ){$x_8$};

  \draw [name path = m ,opacity=1] (m1) to [out=-180+50,in=60] (m2) to [out=-180+60,in=60] (m3) to [out=-180+60,in=50] (m4);%

  \draw [name path = L1] (hexagon.corner 6) to (hexagon.corner 2);
  \draw [name path = L2] (hexagon.corner 6) to (hexagon.corner 3);
  \draw [name path = L3] (hexagon.corner 5) to (hexagon.corner 3);

  \draw (m1) node [fill,circle,scale=0.3] {};
  \draw (m2) node [fill,circle,scale=0.3] {};
  \draw (m3) node [fill,circle,scale=0.3] {};
  \draw (m4) node [fill,circle,scale=0.3] {};

  \node[left] at (m1) {$\theta_1$};
  \node[left] at (m2) {$\theta_2$};
  \node[right] at (m3) {$\theta_3$};
  \node[left] at (m4) {$\theta_4$};

  \draw (hexagon.corner 1) node [fill,circle,scale=0.35] {};
  \draw (hexagon.corner 2) node [fill,circle,scale=0.35] {};
  \draw (hexagon.corner 3) node [fill,circle,scale=0.35] {};
  \draw (hexagon.corner 4) node [fill,circle,scale=0.35] {};
  \draw (hexagon.corner 5) node [fill,circle,scale=0.35] {};
  \draw (hexagon.corner 6) node [fill,circle,scale=0.35] {};

  \draw [] (m1)--(hexagon.corner 6);
  \draw [] (m2)-- (hexagon.corner 6);
  \draw [] (m3)--(hexagon.corner 3);
  \draw [] (m4)-- (hexagon.corner 3);

  \draw [ultra thick, red] (hexagon.corner 1) to (hexagon.corner 6) ;
  \draw [ultra thick, blue] (hexagon.corner 6) to (m1) to [out=-180+50,in=60] (m2) to [out=-180+60,in=60] (m3) to (hexagon.corner 3);
   \draw [ultra thick, red] (m1) to [out=-180+50,in=60] (m2) to [out=-180+60,in=60] (m3) to (hexagon.corner 3);

   \draw [ultra thick, red] (m1) to [out=-180+50,in=60] (m2) to [out=-180+60,in=60] (m3);
   \draw [ultra thick, blue](m3) to (hexagon.corner 3);

  \draw [ultra thick, red] (hexagon.corner 3) to (hexagon.corner 4) ;
\end{tikzpicture}
\\
$t=(1,6,\theta_1,\theta_3,3,4\,|\,x_1,\sigma_1,\tau_{13},\sigma_3,x_4)$

\end{tabular}
\caption{Examples of super $T$-paths.}
\label{fig:examplepath}
\end{figure}
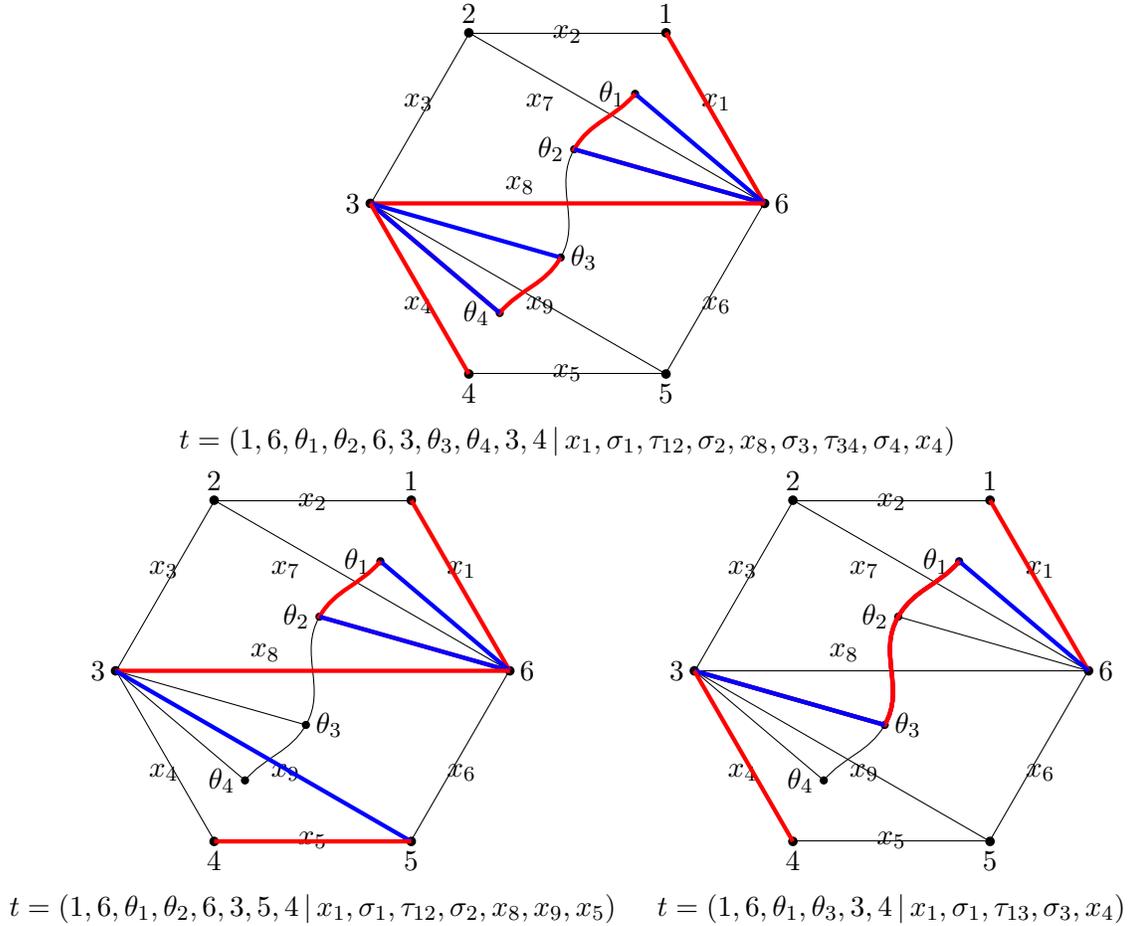
\end{Example}

\subsection{Default orientation and positive order}	

{Fix an arc $(a,b)$, and as mentioned in the previous section, we assume that this arc crosses all diagonals in the chosen triangulation.
We also choose a direction $a \to b$ for this arc.}
Based on these choices, we will define a \emph{default orientation}, which guarantees an ordering of the $\mu$-invariants
in which the coefficients in the $\lambda$-length expansion have positive coefficients.
Accordingly, we will call this ordering the \emph{positive ordering}.
Notice that only the orientation of interior edges affects our calculation of $\lambda$-lengths, therefore the orientation of boundary edges will be omitted.

Recall the convention for labelling the vertices of a polygon: given two vertices $a$ and $b$ and a chosen
direction $a \to b$, we label $c_0 = a$, $c_{N+1} = b$, and the fan centers are labelled $c_1,\dots,c_N$ in such a way that $c_i$
is closer to $a$ than $c_{i+1}$ is. See Figure~\ref{fig:default_spin} for an illustration.

\begin{Definition}[default orientation]\rm
 When the triangulation is a single fan with $c_1$ being the center, every interior edge is oriented away from $c_1$.
 When $T$ is a triangulation with $N>1$ fans, where $c_1,\dots,c_N$ are the centers, the interior edges within each fan segment are oriented away
 from its center. The edges where two fans meet each other are oriented as
 \( c_1\rightarrow c_2\rightarrow\cdots\rightarrow c_{N-1}\rightarrow c_N \).
 See Figures~\ref{fig:default_spin} and \ref{fig:default_spin_more}.
\end{Definition}

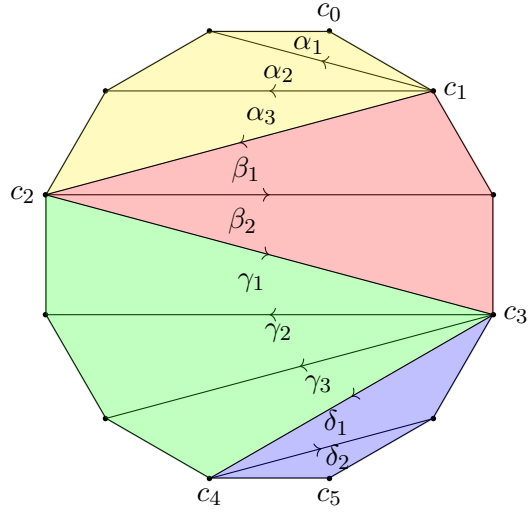
\begin{figure}[h]
\centering
\begin{tikzpicture}[decoration={
    markings,
    mark=at position 0.5 with {\arrow{>}}},
    scale=1.2] 

\tikzstyle{every path}=[draw] 
		\path
    node[
      regular polygon,
      regular polygon sides=12,
      draw,
      inner sep=2.1cm,
    ] (hexagon) {}
    %
    (hexagon.corner 1) node[above] {$c_0$}
    (hexagon.corner 2) node[above] {}
    (hexagon.corner 3) node[left] {}
    (hexagon.corner 4) node[left] {$c_2$}
    (hexagon.corner 5) node[below] {}
    (hexagon.corner 6) node[right] {}
    (hexagon.corner 7) node[below] {$c_4$}
    (hexagon.corner 8) node[below] {$c_5$}
    (hexagon.corner 9) node[left] {}
    (hexagon.corner 10) node[right] {$c_3$}
    (hexagon.corner 11) node[below] {}
    (hexagon.corner 12) node[right] {$c_1$}
;

\foreach \x in {1,2,...,12}{
\draw (hexagon.corner \x) node [fill,circle,scale=0.2] {};}

\draw[postaction={decorate}] (hexagon.corner 12)--(hexagon.corner 4);
\draw[postaction={decorate}] (hexagon.corner 12)--(hexagon.corner 2);
\draw[postaction={decorate}] (hexagon.corner 12)--(hexagon.corner 3);

\draw[postaction={decorate}](hexagon.corner 4)--(hexagon.corner 10);
\draw[postaction={decorate}](hexagon.corner 4)--(hexagon.corner 11);

\draw[postaction={decorate}](hexagon.corner 10)--(hexagon.corner 7);
\draw[postaction={decorate}](hexagon.corner 10)--(hexagon.corner 5);
\draw[postaction={decorate}](hexagon.corner 10)--(hexagon.corner 6);

\draw[postaction={decorate}](hexagon.corner 7)--(hexagon.corner 9);

\draw[fill=yellow, nearly transparent] (hexagon.corner 12)--(hexagon.corner 4)--(hexagon.corner 3)--(hexagon.corner 2)--(hexagon.corner 1)--cycle;
\draw[fill=red, nearly transparent] (hexagon.corner 12)--(hexagon.corner 4)--(hexagon.corner 10)--(hexagon.corner 11)--cycle;

\draw[fill=green, nearly transparent] (hexagon.corner 10)-- (hexagon.corner 4)--(hexagon.corner 5)--(hexagon.corner 6)--(hexagon.corner 7)--cycle;
\draw[fill=blue, nearly transparent] (hexagon.corner 10)--(hexagon.corner 7)--(hexagon.corner 8)--(hexagon.corner 9)--cycle;

\coordinate (m 1) at  ($0.38*(hexagon.corner 2)+0.38*(hexagon.corner 1)+0.24*(hexagon.corner 12)$);
\coordinate (m 2) at  ($0.28*(hexagon.corner 2)+0.28*(hexagon.corner 3)+0.44*(hexagon.corner 12)$);
\coordinate (m 3) at  ($0.24*(hexagon.corner 4)+0.24*(hexagon.corner 3)+0.52*(hexagon.corner 12)$);
\coordinate (m 4) at  ($0.52*(hexagon.corner 4)+0.24*(hexagon.corner 11)+0.24*(hexagon.corner 12)$);
\coordinate (m 5) at  ($0.56*(hexagon.corner 4)+0.22*(hexagon.corner 11)+0.22*(hexagon.corner 10)$);
\coordinate (m 6) at  ($0.27*(hexagon.corner 4)+0.27*(hexagon.corner 5)+0.46*(hexagon.corner 10)$);
\coordinate (m 7) at  ($0.23*(hexagon.corner 6)+0.23*(hexagon.corner 5)+0.44*(hexagon.corner 10)$);
\coordinate (m 8) at  ($0.26*(hexagon.corner 6)+0.26*(hexagon.corner 7)+0.48*(hexagon.corner 10)$);
\coordinate (m 9) at  ($0.25*(hexagon.corner 9)+0.5*(hexagon.corner 7)+0.25*(hexagon.corner 10)$);
\coordinate (m 10) at  ($0.36*(hexagon.corner 9)+0.24*(hexagon.corner 7)+0.4*(hexagon.corner 8)$);

\node at (m 1) {$\alpha_1$};
\node at (m 2) {$\alpha_2$};
\node at (m 3) {$\alpha_3$};
\node at (m 4) {$\beta_1$};
\node at (m 5) {$\beta_2$};
\node at (m 6) {$\gamma_1$};
\node at (m 7) {$\gamma_2$};
\node at (m 8) {$\gamma_3$};
\node at (m 9) {$\delta_1$};
\node at (m 10) {$\delta_2$};

\end{tikzpicture}
\caption{The default orientation of a generic triangulation where each fan segment is colored differently. The faces are labelled by their $\mu$-invariants.}

\label{fig:default_spin}
	
\end{figure}
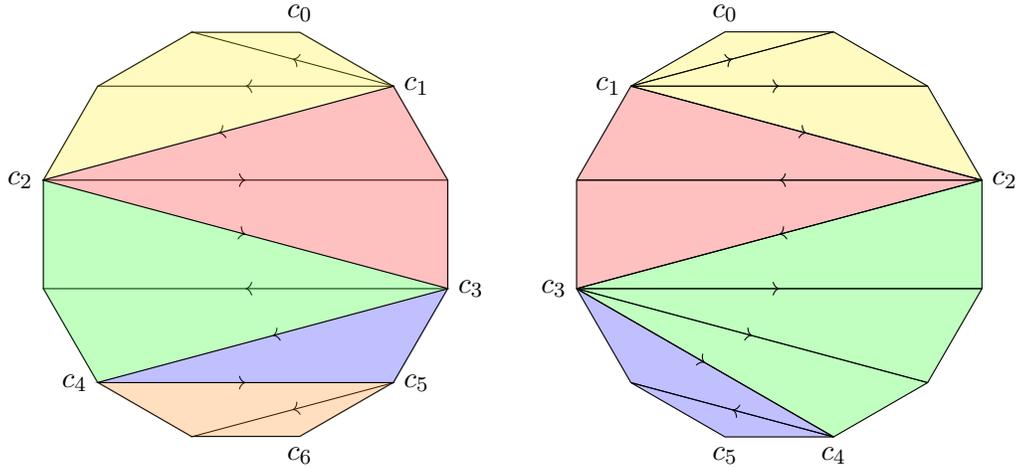
\begin{figure}[h]
\centering
\begin{tikzpicture}[decoration={
    markings,
    mark=at position 0.5 with {\arrow{>}}}
    ,scale=1.1] 

\tikzstyle{every path}=[draw] 
		\path
    node[
      regular polygon,
      regular polygon sides=12,
      draw,
      inner sep=1.9cm,
    ] (hexagon) {}
    %
    (hexagon.corner 1) node[above] {$c_0$}
    (hexagon.corner 2) node[above] {}
    (hexagon.corner 3) node[left] {}
    (hexagon.corner 4) node[left] {$c_2$}
    (hexagon.corner 5) node[below] {}
    (hexagon.corner 6) node[left] {$c_4$}
    (hexagon.corner 7) node[below] {}
    (hexagon.corner 8) node[below] {$c_6$}
    (hexagon.corner 9) node[right] {$c_5$}
    (hexagon.corner 10) node[right] {$c_3$}
    (hexagon.corner 11) node[below] {}
    (hexagon.corner 12) node[right] {$c_1$}
;

\draw[postaction={decorate}] (hexagon.corner 12)--(hexagon.corner 4);
\draw[postaction={decorate}] (hexagon.corner 12)--(hexagon.corner 2);
\draw[postaction={decorate}] (hexagon.corner 12)--(hexagon.corner 3);

\draw[postaction={decorate}](hexagon.corner 4)--(hexagon.corner 10);
\draw[postaction={decorate}](hexagon.corner 4)--(hexagon.corner 11);

\draw[postaction={decorate}](hexagon.corner 6)--(hexagon.corner 9);
\draw[postaction={decorate}](hexagon.corner 10)--(hexagon.corner 5);
\draw[postaction={decorate}](hexagon.corner 10)--(hexagon.corner 6);

\draw[postaction={decorate}](hexagon.corner 9)--(hexagon.corner 7);

\draw[fill=yellow, nearly transparent] (hexagon.corner 12)--(hexagon.corner 4)--(hexagon.corner 3)--(hexagon.corner 2)--(hexagon.corner 1)--cycle;
\draw[fill=red, nearly transparent] (hexagon.corner 12)--(hexagon.corner 4)--(hexagon.corner 10)--(hexagon.corner 11)--cycle;

\draw[fill=green, nearly transparent] (hexagon.corner 10)-- (hexagon.corner 4)--(hexagon.corner 5)--(hexagon.corner 6)--cycle;
\draw[fill=blue, nearly transparent] (hexagon.corner 10)--(hexagon.corner 6)--(hexagon.corner 9)--cycle;   
\draw[fill=orange, nearly transparent] (hexagon.corner 6)--(hexagon.corner 9)--(hexagon.corner 8)--(hexagon.corner 7)--cycle;

\coordinate (m 1) at  ($0.38*(hexagon.corner 2)+0.38*(hexagon.corner 1)+0.24*(hexagon.corner 12)$);
\coordinate (m 2) at  ($0.28*(hexagon.corner 2)+0.28*(hexagon.corner 3)+0.44*(hexagon.corner 12)$);
\coordinate (m 3) at  ($0.24*(hexagon.corner 4)+0.24*(hexagon.corner 3)+0.52*(hexagon.corner 12)$);
\coordinate (m 4) at  ($0.52*(hexagon.corner 4)+0.24*(hexagon.corner 11)+0.24*(hexagon.corner 12)$);
\coordinate (m 5) at  ($0.56*(hexagon.corner 4)+0.22*(hexagon.corner 11)+0.22*(hexagon.corner 10)$);
\coordinate (m 6) at  ($0.27*(hexagon.corner 4)+0.27*(hexagon.corner 5)+0.46*(hexagon.corner 10)$);
\coordinate (m 7) at  ($0.23*(hexagon.corner 6)+0.23*(hexagon.corner 5)+0.44*(hexagon.corner 10)$);
\coordinate (m 8) at  ($0.26*(hexagon.corner 6)+0.26*(hexagon.corner 7)+0.48*(hexagon.corner 10)$);
\coordinate (m 9) at  ($0.25*(hexagon.corner 9)+0.5*(hexagon.corner 7)+0.25*(hexagon.corner 10)$);
\coordinate (m 10) at  ($0.36*(hexagon.corner 9)+0.24*(hexagon.corner 7)+0.4*(hexagon.corner 8)$);

\end{tikzpicture}
\quad
\begin{tikzpicture}[decoration={
    markings,
    mark=at position 0.5 with {\arrow{>}}}
    ,scale=1.1] 
\begin{scope}[xscale=1,yscale=1]
\tikzstyle{every path}=[draw] 
		\path
    node[
      regular polygon,
      regular polygon sides=12,
      draw=none,
      inner sep=1.9cm,
    ] (T) {}

    (T.corner 1) node[above] {}
    (T.corner 2) node[above] {$c_0$}
    (T.corner 3) node[left] {$c_1$}
    (T.corner 4) node[left] {}
    (T.corner 5) node[left] {$c_3$}
    (T.corner 6) node[right] {}
    (T.corner 7) node[below] {$c_{5}$}
    (T.corner 8) node[below] {$c_{4}$}
    (T.corner 9) node[left] {}
    (T.corner 10) node[right] {}
    (T.corner 11) node[right] {$c_2$}
    (T.corner 12) node[right] {}
    
;

\draw [-] (T.corner 1) to (T.corner 2) to (T.corner 3) to (T.corner 4) to (T.corner 5);
\draw [ ] (T.corner 1) to (T.corner 12);
\draw [ ] (T.corner 5) to (T.corner 6);
\draw [-] (T.corner 12)--(T.corner 11)--(T.corner 10);
\draw [ ] (T.corner 10) to (T.corner 9)to (T.corner 8);
\draw [-] (T.corner 6)--(T.corner 7)--(T.corner 8);

\coordinate (m 1) at  ($0.38*(T.corner 1)+0.38*(T.corner 2)+0.24*(T.corner 3)$);
\coordinate (m 2) at  ($0.28*(T.corner 1)+0.28*(T.corner 12)+0.44*(T.corner 3)$);
\coordinate (m 3) at  ($0.245*(T.corner 11)+0.245*(T.corner 12)+0.51*(T.corner 3)$);
\coordinate (m 4) at  ($0.52*(T.corner 11)+0.24*(T.corner 4)+0.24*(T.corner 3)$);
\coordinate (m 5) at  ($0.56*(T.corner 11)+0.22*(T.corner 4)+0.22*(T.corner 5)$);
\coordinate (m 6) at  ($0.27*(T.corner 11)+0.27*(T.corner 10)+0.46*(T.corner 5)$);
\coordinate (m 7) at  ($0.23*(T.corner 9)+0.23*(T.corner 10)+0.44*(T.corner 5)$);
\coordinate (m 8) at  ($0.26*(T.corner 9)+0.26*(T.corner 8)+0.48*(T.corner 5)$);
\coordinate (m 9) at  ($0.25*(T.corner 6)+0.5*(T.corner 8)+0.25*(T.corner 5)$);
\coordinate (m 10) at  ($0.4*(T.corner 6)+0.3*(T.corner 8)+0.3*(T.corner 7)$);


\draw [name path = M ,opacity=0] (T.corner 1) to [out=15,in=40] (m 1) to [out=-180+40,in=55] (m 2) to [out=-180+55,in=75] (m 3) to [out=-180+75,in=80] (m 4) to [out=-180+80,in=90] (m 5) to [out=-180+90,in=110] (m 6) to [out=-180+110,in=125] (m 7) to [out=-180+125,in=125] (m 8) to [out=-180+125,in=105] (m 9) to [out=-180+105,in=80] (m 10) to [out=-180+80,in=80] (T.corner 8);%

\draw[name path = L 2] (T.corner 3)--(T.corner 11);
\draw[name path = L 1] (T.corner 3)--(T.corner 1);
\draw[] (T.corner 3)--(T.corner 12);
\draw[name path = L 3](T.corner 11)--(T.corner 5);
\draw[](T.corner 4)--(T.corner 11);
\draw[name path = L 4, ](T.corner 5)--(T.corner 8);
\draw[](T.corner 10)--(T.corner 5);
\draw[ ](T.corner 5)--(T.corner 9);
\draw[name path = L 5](T.corner 8)--(T.corner 6);

\path [name intersections={of= M and L 1,by=X 1}];
\path [name intersections={of= M and L 2,by=X 2}];
\path [name intersections={of= M and L 3,by=X 3}];
\path [name intersections={of= M and L 4,by=X 4}];
\path [name intersections={of= M and L 5,by=X 5}];

\draw[fill=yellow, nearly transparent] (T.corner 11)--(T.corner 12)--(T.corner 1)--(T.corner 2)--(T.corner 3)--cycle;
\draw[fill=red, nearly transparent] (T.corner 5)--(T.corner 4)--(T.corner 3)--(T.corner 11)--cycle;
\draw[fill=green, nearly transparent] (T.corner 11)--(T.corner 5)--(T.corner 8)--(T.corner 9)--(T.corner 10)--cycle;
\draw[fill=blue, nearly transparent] (T.corner 5)--(T.corner 8)--(T.corner 7)--(T.corner 6)--cycle;

\draw[postaction={decorate}] (T.corner 3)--(T.corner 1);
\draw[postaction={decorate}] (T.corner 3)--(T.corner 12);
\draw[postaction={decorate}] (T.corner 3)--(T.corner 11);
\draw[postaction={decorate}] (T.corner 11)--(T.corner 4);
\draw[postaction={decorate}] (T.corner 11)--(T.corner 5);
\draw[postaction={decorate}] (T.corner 5)--(T.corner 10);
\draw[postaction={decorate}] (T.corner 5)--(T.corner 9);
\draw[postaction={decorate}] (T.corner 5)--(T.corner 8);
\draw[postaction={decorate}] (T.corner 8)--(T.corner 6);

\end{scope}
\end{tikzpicture}

\caption{More examples of default orientation.}
\label{fig:default_spin_more}
\end{figure}

\begin{Remark}\rm
 As mentioned above, the definition of default orientation depends on the choice of direction $a \to b$.
 In particular, choosing the opposite direction $b \to a$ would change the labelling so that $c_i$ becomes $c_{N-i}$.
 The effect is that the orientation of the diagonals within a fan are unchanged, but the diagonals connecting
 two fan centers would have the reverse orientation.
\end{Remark}

\begin{Definition}[positive ordering]\label{def:pos_order}\rm
 For $F$ a single fan triangulation with center $c_1$, let $\theta_1,\dots,\theta_k$ be its faces. The \emph{positive ordering} is defined to be
 \[\theta_1>\theta_2>\cdots>\theta_k, \]
 where $\theta_1,\dots,\theta_k$ are ordered counterclockwise around $c_1$.

 For a triangulation $T$ with fans $F_1,\dots,F_N$, we order the fans as follows in two different cases
 \begin{enumerate}\itemsep=0pt
 	\item If $c_{N-1},c_{N},c_{N+1}$ are oriented {counterclockwise}, then we order the fans as follows
\[\begin{cases}
	F1>F3>\cdots>F_{N-1}>F_N>F_{N-2}>\cdots>F4>F2&\text{if }N\text{ is even,}\\
	F2>F4>\cdots>F_{N-1}>F_N>F_{N-2}>\cdots>F3>F1&\text{if }N\text{ is odd.}
\end{cases}
\]
\item If $c_{N-1}$, $c_{N}$, $c_{N+1}$ are oriented {clockwise}, then we order the fans as follows
\[\begin{cases}
	F1>F3>\cdots>F_{N-2}>F_N>F_{N-1}>\cdots>F4>F2&\text{if }N\text{ is odd,}\\
	F2>F4>\cdots>F_{N-2}>F_N>F_{N-1}>\cdots>F3>F1&\text{if }N\text{ is even.}
\end{cases}
\]
\end{enumerate}

Then the \emph{positive ordering} on faces of $T$ is induced by the ordering on fans and the positive ordering within each fan.
\end{Definition}

\begin{Remark} \label{rem:pos_order}\rm
 The positive ordering may also be described inductively, triangle-by-triangle, as follows:
 Recall from the definition of auxiliary graph that the triangles are labelled $\theta_1, \theta_2, \dots, \theta_n$ in order from $a$ to $b$.
 For each triangle $\theta_k$, look at the edge separating $\theta_k$ and $\theta_{k+1}$. If the edge is
 oriented so that $\theta_k$ is to the right, then we declare that $\theta_k > \theta_i$ for all $i > k$.
 On the other hand, if $\theta_k$ is to the left, we declare that $\theta_k < \theta_i$ for all $i > k$.
\end{Remark}

For example, in Figure~\ref{fig:default_spin}, the positive ordering on the faces is
\[\alpha_1>\alpha_2>\alpha_3>\gamma_1>\gamma_2>\gamma_3> \delta_2>\delta_1>\beta_2>\beta_1.\]

\subsection{Expansion formula}

\begin{Definition}[weight] \label{def:weight}\rm
 Let $t \in \mathcal{T}_{ab}$ be a super $T$-path which uses edges $t_1,t_2,\dots$ in the auxiliary graph $\Gamma^T_{a,b}$.
 We will assign to each edge $t_i$ a weight, which will be an element in the super algebra
 $\mathbb R\bigl[x_1^{\pm\frac{1}{2}},\dots,x_{2n+3}^{\pm\frac{1}{2}} \,|\, \theta_1,\dots,\theta_{n+1}\bigr]$ (where $\theta_i$'s are the odd generators) as follows.
 For the parity of edges $t_i \in \sigma$ or $\tau$, we recall axiom (T6) of Definition~\ref{def:super-T}:
 \[
 \wt(t_i) := \begin{cases}
 x_j &\text{if } t_i \in T \text{, with }\lambda\text{-length }x_j, \text{ and }i \text{ is odd}, \\
 x_j^{-1} &\text{if } t_i \in T \text{, with }\lambda\text{-length }x_j, \text{ and }i \text{ is even}, \\
 \sqrt{\frac{x_j}{x_kx_l}} \, \theta_s &\text{if } t_i \in\sigma \text{ and the face containing }t_i \\
 & \text{is as pictured below} ~ (i \text{ must be even}), \\
 1&\text{if }t_i\in\tau~ (i \text{ must be odd}).
 \end{cases} \]

 In keeping with the intuition mentioned after Definition~\ref{def:super-T}, teleportation is unweighted:
 
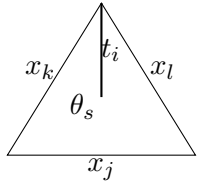
\begin{center}
 \begin{tikzpicture}[scale=1.25]
 \coordinate (1) at (1,0);
 \coordinate (2) at (-1,0);
 \coordinate (3) at (0,1.618);
 \coordinate (4) at (0, 0.618);
 \draw (1)--(2)--(3)--cycle;
 \draw [thick](3)--(4);
 \node at (-0.2, 0.518) {$\theta_s$};
 \node at (0.1,1.118) {$t_i$};
 \node at (0.65,0.9) {$x_l$};
 \node at (-0.65,0.9) {$x_k$};
 \node at (0,-0.15) {$x_j$};
 \end{tikzpicture}
 \end{center}
 Here, $\theta_s$ is the $\mu$-invariant associated to the face containing $t_i$.
	Finally, we define the weight of a super $T$-path to be the product of the weights of its edges
 \[\wt(t)=\prod_{t_i\in t}\wt(t_i),\]
 where the product of $\mu$-invariants is taken under the positive ordering.
\end{Definition}

 In what follows, we will use $\lambda_{ab}$ to denote the $\lambda$-length of the arc $(a,b)$,
 and we will write $\boxed{ijk}$ to denote the $\mu$-invariant associated to the triple of
 ideal points $(i,j,k)$, subject to Remark~\ref{rem:mu-invariants}.
	
The following theorem is the main result of the current paper, giving an explicit expression for arbitrary super $\lambda$-lengths
in terms of the $\lambda$-lengths and $\mu$-invariants of a fixed triangulation.

\begin{Theorem}\label{thm:main}
Under the default orientation,
 the $\lambda$-length of $(a,b)$ is given by
 \[\lambda_{a,b} = \sum_{t\in\mathcal T_{a,b}} \wt(t).
 \]
\end{Theorem}

When starting with a generic orientation, one can first apply a sequence of equivalence relations (reversing the arrows around a triangle and negating the $\mu$-invariant)
to get to the default orientation, with some of the $\mu$-invariants having signs changed. This is always possible due to Proposition~\ref{prop:equiv_rel}. See Example~\ref{example:main}.

Equivalently, we can state the main theorem with respect to an arbitrary choice of orientations (not necessarily the default one) as follows.

\begin{Corollary}\label{cor:main}
 With a generic choice of orientation, the $\lambda$-length of $(a,b)$ is given by
 \[ \lambda_{a,b} = \sum_{t\in\mathcal T_{a,b}}(-1)^{\inv(t)} \wt(t), \]
 where ${\rm inv}(t)$ is the number of edges in the triangulation which cross a $\tau$-edge of $t$ and are oriented opposite the default orientation.
\end{Corollary}
\begin {proof}
 Label the internal diagonals $d_1,d_2,\dots,d_{n-1}$ in order of their proximity from $c_0$ to $c_N$, and similarly label
 the $\mu$-invariants of the triangles $\theta_1,\theta_2,\dots,\theta_n$.
 Suppose $d_k$ is the last diagonal whose orientation disagrees with the default orientation. Proposition~\ref{prop:equiv_rel} describes how we can
 find another orientation, representing the same spin structure, where $d_k$ is the only internal diagonal whose orientation is changed.
 From Remark~\ref{rem:mu-invariants}, we see that in doing so, $\theta_i$ must be replaced by $-\theta_i$ for all $i \leq k$.

 This process may then be repeated for all diagonals whose orientation differs from the default one. The end result is as follows: if $d_{i_1},d_{i_2},\dots,d_{i_k}$
 are all the internal diagonals whose orientation disagrees with the default orientation (and $i_1 < i_2 < \dots < i_k$), then the $\mu$-invariants of
 triangles between $d_{i_{k-1}}$ and $d_{i_k}$ are negated, those between $d_{i_{k-2}}$ and $d_{i_{k-1}}$ are not, those between $d_{i_{k-3}}$ and
 $d_{i_{k-2}}$ are negated, those between $d_{i_{k-4}}$ and $d_{i_{k-3}}$ are not, etc.

 Any super $T$-path $t$ contains some number of super-steps. Suppose $t$ contains a super-step $(\dots,\theta_i,\theta_j,\dots\,|\,\dots,\sigma_i,\tau_{ij},\sigma_j,\dots)$.
 Let $m_{ij}$ be the number of diagonals between $\theta_i$ and $\theta_j$ whose orientation disagrees with the default. If $m_{ij}$ is even, then when we change from
 the given orientation to the default one, either both $\theta_i$ and $\theta_j$ are negated, or both stay the same. In this case, the product $\theta_i \theta_j$
 is unchanged when passing to the default orientation. On the other hand, if $m_{ij}$ is odd, then
 one of them is negated, and the other stays the same, in which case the product $\theta_i \theta_j$ is negated.

 The number ${\rm inv}(t)$ is simply the sum of these $m_{ij}$'s over all super steps in the path $t$.
\end {proof}

It is apparent from the super $T$-path formulation in Theorem~\ref{thm:main} that these $\lambda$-lengths satisfy something analogous
to the Laurent phenomenon exhibited by ordinary cluster algebras. This is summarized in the following corollary.

\begin{Corollary}\label{cor:laurent}
 Let $\tilde{\theta}_i := {\rm wt}(\sigma_i) = \sqrt{\frac{x_j}{x_kx_l}} \theta_i$ $($see Definition {\rm \ref{def:weight})}.
 For any pair of vertices~$a$,~$b$ of the polygon,
 \begin{itemize}\itemsep=0pt
 \item[$(a)$] $\lambda_{ab} \in {\mathbb R}\bigl[x_1^{\pm 1}, \dots, x_{2n+3}^{\pm 1} \,|\, \tilde{\theta}_1, \dots, \tilde{\theta}_{n+1}\bigr]$.
 In other words, each term of $\lambda_{ab}$ is the product of a~Laurent monomial
 in the $x_i$'s and a monomial in the $\tilde{\theta}_i$'s.
 \item[$(b)$] $\lambda_{ab} \in {\mathbb R}\bigl[x_{1}^{\pm \frac{1}{2}}, \dots, x_{2n+3}^{\pm \frac{1}{2}} \,|\, \theta_1, \dots, \theta_{n+1}\bigr]$.
 In other words, each term of~$\lambda_{ab}$ is the product of a~Laurent monomial in the square roots of the~$x_i$'s
 and a monomial in the~$\theta_i$'s.
 \end{itemize}
\end{Corollary}

\input{example_main_theorem.tex}

\begin{Example}\label{example:main}\rm
An example of the expansion formula given in Theorem~\ref{thm:main} and Corollary~\ref{cor:main} is shown in Figure~\ref{fig:main_example}. Continuing from Example~\ref{example:14some},
this figure shows all super $T$-paths in $\mathcal{T}_{1,4}$. For example, to obtain the default orientation, we would need to flip the arrow on edge $(3,5)$.
We can do this by flipping all arrows of the last triangle while negating the $\mu$-invariant to $-\theta_4$.
Keeping the same positive ordering $\theta_1>\theta_2>\theta_4>\theta_3$, this would make all terms in Figure~\ref{fig:main_example} positive.

In this example, the Laurent expansion can also be written in terms of the $\tilde \theta$'s (defined in Corollary~\ref{cor:laurent}) as
\begin{gather*}
\lambda_{14}=\frac{x_1x_4}{x_8}+\frac{x_1 x_3 x_5}{x_7 x_9}+\frac{x_1x_3x_4x_6}{x_7x_8x_9}+\frac{x_2x_5x_8}{x_7x_9} +\frac{x_2x_4x_6}{x_7x_9}+\frac{x_1x_5x_8}{x_9}\tilde\theta_1\tilde\theta_2 +\frac{x_1x_4x_6}{x_9}\tilde\theta_1\tilde\theta_2\\
\hphantom{\lambda_{14}=}{}
+x_1x_4\tilde\theta_1\tilde\theta_3-x_1x_4\tilde\theta_1\tilde\theta_4 +x_1x_4\tilde\theta_2\tilde\theta_3
-x_1x_4\tilde\theta_2\tilde\theta_4
-\frac{x_1x_3x_4}{x_7}\tilde\theta_4\tilde\theta_3
-\frac{x_2x_4x_8}{x_7}\tilde\theta_4\tilde\theta_3\\
\hphantom{\lambda_{14}=}{}
-x_1x_4x_8\tilde\theta_1\tilde\theta_2\tilde\theta_4\tilde\theta_3.
\end{gather*}
\end{Example}

\subsection[Super T-paths in a single fan triangulation]{Super $\boldsymbol{T}$-paths in a single fan triangulation}

Let $P$ be an $n$-gon and $T$ a fan triangulation. Let vertex $1$ be the fan center.

\begin{Lemma}[Schiffler] \label{lemma:Schiffler}
	For $2\leq i\leq n-1$, define
	\[\alpha_{i-1}:=(2,1,i,i+1,1,n\,|\,\dots),\qquad2\leqslant i\leqslant n-1.\]
	Then $\mathcal T^0_{2,n}=\{\alpha_i\colon 1\leq i\leq n-2\}$. See Example~{\rm \ref{exmp:T-Path-fan}}.
	
	Note that in case of $i=2$ or $i=n-1$, $\alpha_{i-1}$ collapses to a shorter $T$-path after removing backtracking. In the special case of $n=3$, there is a unique $T$-path in $\mathcal T^0_{2,n}$, i.e., $\alpha_2 = \alpha_{n-1}$, which consists of the single edge~$(2,3)$.
\end{Lemma}

\begin{Example}[ordinary $T$-paths of a fan]\label{exmp:T-Path-fan}\rm
	The following are the ordinary $T$-paths of a fan triangulation of an octagon. Notice that each of these $T$-paths surrounds (and is in bijection with) one of the triangles of $T$, as illustrated in yellow below:
\begin{center}
\begin{tabular}{cccccc}
 \begin{tikzpicture}[]
 	\tikzstyle{every path}=[draw] 
		\path
    node[
      regular polygon,
      regular polygon sides=8,
      draw,
      inner sep=.7cm,
    ] (T) {}
    %
    (T.corner 1) node[above] {$1$}
    (T.corner 2) node[above] {$2$}
    (T.corner 3) node[left] {$3$}
    (T.corner 4) node[left] {$4$}
    (T.corner 5) node[below] {$5$}
    (T.corner 6) node[below] {$6$}
    (T.corner 7) node[right] {$7$}
    (T.corner 8) node[right] {$8$}
    ;
    
    \draw [fill=yellow, nearly transparent] (T.corner 1) to (T.corner 2) to (T.corner 3) to cycle;
    
    \foreach \i in {3,...,7}{
    \draw[-] (T.corner 1) to (T.corner \i);
    }
   
    \draw [color=red, thick] (T.corner 2) to (T.corner 3);
    \draw [color=blue, thick] (T.corner 1) to (T.corner 3);
    \draw [color=red, thick] (T.corner 1) to (T.corner 8);

    \end{tikzpicture}
 &
  \begin{tikzpicture}[]
 	\tikzstyle{every path}=[draw] 
		\path
    node[
      regular polygon,
      regular polygon sides=8,
      draw,
      inner sep=.7cm,
    ] (T) {}
    %
    (T.corner 1) node[above] {$1$}
    (T.corner 2) node[above] {$2$}
    (T.corner 3) node[left] {$3$}
    (T.corner 4) node[left] {$4$}
    (T.corner 5) node[below] {$5$}
    (T.corner 6) node[below] {$6$}
    (T.corner 7) node[right] {$7$}
    (T.corner 8) node[right] {$8$}
    ;
    
    \draw [fill=yellow, nearly transparent] (T.corner 1) to (T.corner 4) to (T.corner 3) to cycle;
    
    \foreach \i in {3,...,7}{
    \draw[-] (T.corner 1) to (T.corner \i);
    }
   
    \draw [color=red, thick] (T.corner 2) to (T.corner 1);
    \draw [color=blue, thick] (T.corner 1) to (T.corner 3);
    \draw [color=red, thick] (T.corner 3) to (T.corner 4);
    \draw [color=blue, thick] (T.corner 1) to (T.corner 4);
    \draw [color=red, thick] (T.corner 1) to (T.corner 8);
    
    \end{tikzpicture}
 &
\begin{tikzpicture}[]
 	\tikzstyle{every path}=[draw] 
		\path
    node[
      regular polygon,
      regular polygon sides=8,
      draw,
      inner sep=.7cm,
    ] (T) {}
    %
    (T.corner 1) node[above] {$1$}
    (T.corner 2) node[above] {$2$}
    (T.corner 3) node[left] {$3$}
    (T.corner 4) node[left] {$4$}
    (T.corner 5) node[below] {$5$}
    (T.corner 6) node[below] {$6$}
    (T.corner 7) node[right] {$7$}
    (T.corner 8) node[right] {$8$}
    ;
    
    \draw [fill=yellow, nearly transparent] (T.corner 1) to (T.corner 4) to (T.corner 5) to cycle;
    
    \foreach \i in {3,...,7}{
    \draw[-] (T.corner 1) to (T.corner \i);
    }
   
    \draw [color=red, thick] (T.corner 2) to (T.corner 1);
    \draw [color=blue, thick] (T.corner 1) to (T.corner 5);
    \draw [color=red, thick] (T.corner 5) to (T.corner 4);
    \draw [color=blue, thick] (T.corner 1) to (T.corner 4);
    \draw [color=red, thick] (T.corner 1) to (T.corner 8);
    
    \end{tikzpicture}
 &
 \begin{tikzpicture}[]
 	\tikzstyle{every path}=[draw] 
		\path
    node[
      regular polygon,
      regular polygon sides=8,
      draw=none,
      inner sep=.6cm,
    ] (T) {}
    %
    (T.corner 1) node[above] {}
    (T.corner 2) node[above] {}
    (T.corner 3) node[left] {}
    (T.corner 4) node[left] {}
    (T.corner 5) node[below] {}
    (T.corner 6) node[below] {}
    (T.corner 7) node[right] {}
    (T.corner 8) node[right] {}
    ;
    \coordinate (a) at ($0.35*(T.corner 1)+0.15*(T.corner 6)+0.15*(T.corner 5)+0.35*(T.corner 2)$);
    \node at (a) {$\cdots\cdot$};
    \end{tikzpicture}
 &
 \begin{tikzpicture}[]
 	\tikzstyle{every path}=[draw] 
		\path
    node[
      regular polygon,
      regular polygon sides=8,
      draw,
      inner sep=.7cm,
    ] (T) {}
    %
    (T.corner 1) node[above] {$1$}
    (T.corner 2) node[above] {$2$}
    (T.corner 3) node[left] {$3$}
    (T.corner 4) node[left] {$4$}
    (T.corner 5) node[below] {$5$}
    (T.corner 6) node[below] {$6$}
    (T.corner 7) node[right] {$7$}
    (T.corner 8) node[right] {$8$}
    ;
    
    \draw [fill=yellow, nearly transparent] (T.corner 1) to (T.corner 7) to (T.corner 8) to cycle;
    
    \foreach \i in {3,...,7}{
    \draw[-] (T.corner 1) to (T.corner \i);
    }
   
    \draw [color=red, thick] (T.corner 2) to (T.corner 1);
    \draw [color=blue, thick] (T.corner 1) to (T.corner 7);
    \draw [color=red, thick] (T.corner 7) to (T.corner 8);
    
    \end{tikzpicture}
    
\end{tabular}
\end{center}
\end{Example}

\begin{Lemma}\label{lm:super-path-single-fan}
	Every super $T$-path in $\mathcal T^1_{2,n}$ is of the following form
	\[(2,1,\theta_i,\theta_j,1,n\,|\,x_1,\sigma_i,\tau_{ij},\sigma_j,x_n)\]
	for $1\leq i<j\leq n-1$. See Example~{\rm \ref{exmp:super-T-path-fan}}.
\end{Lemma}

We illustrate this lemma with an example before proving it.

\begin{Example}\label{exmp:super-T-path-fan}\rm \allowdisplaybreaks
 The following are the ${4 \choose 2}=6$ different non-ordinary super $T$-paths (elements of $\mathcal{T}^1_{2,6}$) of a single-fan hexagon:
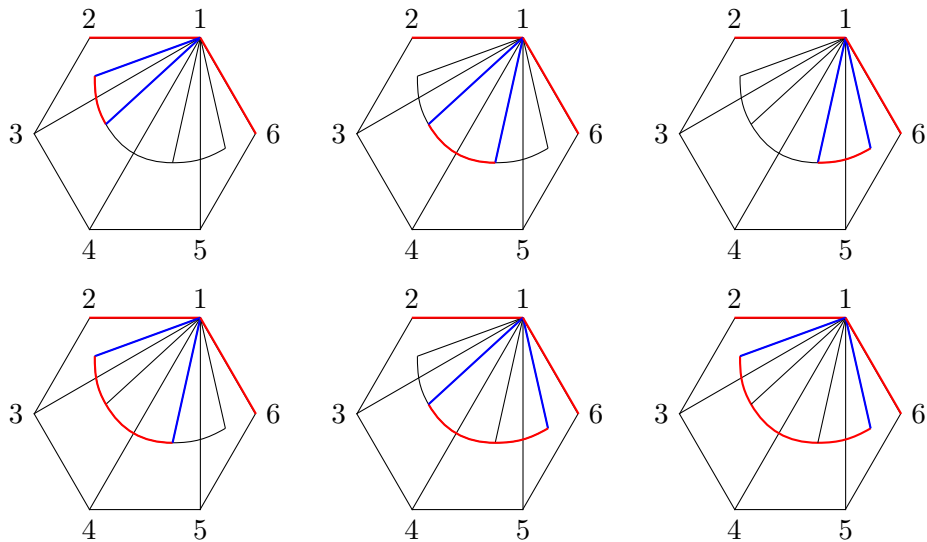
\begin{center}
\begin{center}
\begin{tabular}{ccc}

\begin{tikzpicture}
	\tikzstyle{every path}=[draw] 
		\path
    node[
      regular polygon,
      regular polygon sides=6,
      draw,
      inner sep=0.9cm,
    ] (hexagon) {}
    %
    (hexagon.corner 1) node[above] {$1$}
    (hexagon.corner 2) node[above] {$2$}
    (hexagon.corner 3) node[left] {$3$}
    (hexagon.corner 4) node[below] {$4$}
    (hexagon.corner 5) node[below] {$5$}
    (hexagon.corner 6) node[right] {$6$}
    ;

  \foreach \i in {3,...,5}{
  \draw [-] (hexagon.corner 1) to (hexagon.corner \i);
  }
  
  \coordinate (m1) at ($0.35*(hexagon.corner 2)+0.25*(hexagon.corner 1)+0.4*(hexagon.corner 3)$);
  \coordinate (m2) at ($0.4*(hexagon.corner 3)+0.4*(hexagon.corner 1)+0.3*(hexagon.corner 4)$);
  \coordinate (m3) at ($0.4*(hexagon.corner 5)+0.4*(hexagon.corner 1)+0.3*(hexagon.corner 4)$);
  \coordinate (m4) at ($0.4*(hexagon.corner 5)+0.25*(hexagon.corner 1)+0.4*(hexagon.corner 6)$);

  \draw [name path = m ,opacity=1] (m1) to [out=-90,in=120] (m2) to [out=-180+120,in=180] (m3) to [out=-180+180,in=180+30] (m4);%
  \foreach \i in {1,...,4}{
  \draw [-] (m\i) to (hexagon.corner 1);}

  
  \draw [red,thick] (hexagon.corner 2) to (hexagon.corner 1);
  \draw [blue,thick] (m1) to (hexagon.corner 1);
  \draw [red,thick] (m1) to [out=-90,in=120] (m2);
  \draw [blue,thick] (m2) to (hexagon.corner 1);
  \draw [red,thick] (hexagon.corner 6) to (hexagon.corner 1);

\end{tikzpicture}

&

\begin{tikzpicture}
	\tikzstyle{every path}=[draw] 
		\path
    node[
      regular polygon,
      regular polygon sides=6,
      draw,
      inner sep=0.9cm,
    ] (hexagon) {}
    %
    (hexagon.corner 1) node[above] {$1$}
    (hexagon.corner 2) node[above] {$2$}
    (hexagon.corner 3) node[left] {$3$}
    (hexagon.corner 4) node[below] {$4$}
    (hexagon.corner 5) node[below] {$5$}
    (hexagon.corner 6) node[right] {$6$}
    ;

  \foreach \i in {3,...,5}{
  \draw [-] (hexagon.corner 1) to (hexagon.corner \i);
  }
  
  \coordinate (m1) at ($0.35*(hexagon.corner 2)+0.25*(hexagon.corner 1)+0.4*(hexagon.corner 3)$);
  \coordinate (m2) at ($0.4*(hexagon.corner 3)+0.4*(hexagon.corner 1)+0.3*(hexagon.corner 4)$);
  \coordinate (m3) at ($0.4*(hexagon.corner 5)+0.4*(hexagon.corner 1)+0.3*(hexagon.corner 4)$);
  \coordinate (m4) at ($0.4*(hexagon.corner 5)+0.25*(hexagon.corner 1)+0.4*(hexagon.corner 6)$);

  \draw [name path = m ,opacity=1] (m1) to [out=-90,in=120] (m2) to [out=-180+120,in=180] (m3) to [out=-180+180,in=180+30] (m4);%
  \foreach \i in {1,...,4}{
  \draw [-] (m\i) to (hexagon.corner 1);}

  
  \draw [red,thick] (hexagon.corner 2) to (hexagon.corner 1);
  \draw [blue,thick] (m3) to (hexagon.corner 1);
  \draw [red,thick] (m2) to [out=-180+120,in=180] (m3) ;
  \draw [blue,thick] (m2) to (hexagon.corner 1);
  \draw [red,thick] (hexagon.corner 6) to (hexagon.corner 1);

\end{tikzpicture} 
  
  & 
  
\begin{tikzpicture}
	\tikzstyle{every path}=[draw] 
		\path
    node[
      regular polygon,
      regular polygon sides=6,
      draw,
      inner sep=0.9cm,
    ] (hexagon) {}
    %
    (hexagon.corner 1) node[above] {$1$}
    (hexagon.corner 2) node[above] {$2$}
    (hexagon.corner 3) node[left] {$3$}
    (hexagon.corner 4) node[below] {$4$}
    (hexagon.corner 5) node[below] {$5$}
    (hexagon.corner 6) node[right] {$6$}
    ;

  \foreach \i in {3,...,5}{
  \draw [-] (hexagon.corner 1) to (hexagon.corner \i);
  }
  
  \coordinate (m1) at ($0.35*(hexagon.corner 2)+0.25*(hexagon.corner 1)+0.4*(hexagon.corner 3)$);
  \coordinate (m2) at ($0.4*(hexagon.corner 3)+0.4*(hexagon.corner 1)+0.3*(hexagon.corner 4)$);
  \coordinate (m3) at ($0.4*(hexagon.corner 5)+0.4*(hexagon.corner 1)+0.3*(hexagon.corner 4)$);
  \coordinate (m4) at ($0.4*(hexagon.corner 5)+0.25*(hexagon.corner 1)+0.4*(hexagon.corner 6)$);

  \draw [name path = m ,opacity=1] (m1) to [out=-90,in=120] (m2) to [out=-180+120,in=180] (m3) to [out=-180+180,in=180+30] (m4);%
  \foreach \i in {1,...,4}{
  \draw [-] (m\i) to (hexagon.corner 1);}

  
  \draw [red,thick] (hexagon.corner 2) to (hexagon.corner 1);
  \draw [blue,thick] (m3) to (hexagon.corner 1);
  \draw [red,thick] (m3) to [out=-180+180,in=180+30] (m4) ;
  \draw [blue,thick] (m4) to (hexagon.corner 1);
  \draw [red,thick] (hexagon.corner 6) to (hexagon.corner 1);

\end{tikzpicture} 

\\

\begin{tikzpicture}
	\tikzstyle{every path}=[draw] 
		\path
    node[
      regular polygon,
      regular polygon sides=6,
      draw,
      inner sep=0.9cm,
    ] (hexagon) {}
    %
    (hexagon.corner 1) node[above] {$1$}
    (hexagon.corner 2) node[above] {$2$}
    (hexagon.corner 3) node[left] {$3$}
    (hexagon.corner 4) node[below] {$4$}
    (hexagon.corner 5) node[below] {$5$}
    (hexagon.corner 6) node[right] {$6$}
    ;

  \foreach \i in {3,...,5}{
  \draw [-] (hexagon.corner 1) to (hexagon.corner \i);
  }
  
  \coordinate (m1) at ($0.35*(hexagon.corner 2)+0.25*(hexagon.corner 1)+0.4*(hexagon.corner 3)$);
  \coordinate (m2) at ($0.4*(hexagon.corner 3)+0.4*(hexagon.corner 1)+0.3*(hexagon.corner 4)$);
  \coordinate (m3) at ($0.4*(hexagon.corner 5)+0.4*(hexagon.corner 1)+0.3*(hexagon.corner 4)$);
  \coordinate (m4) at ($0.4*(hexagon.corner 5)+0.25*(hexagon.corner 1)+0.4*(hexagon.corner 6)$);

  \draw [name path = m ,opacity=1] (m1) to [out=-90,in=120] (m2) to [out=-180+120,in=180] (m3) to [out=-180+180,in=180+30] (m4);%
  \foreach \i in {1,...,4}{
  \draw [-] (m\i) to (hexagon.corner 1);}

  
  \draw [red,thick] (hexagon.corner 2) to (hexagon.corner 1);
  \draw [blue,thick] (m1) to (hexagon.corner 1);
  \draw [red,thick] (m1) to [out=-90,in=120] (m2) to [out=-180+120,in=180] (m3);
  \draw [blue,thick] (m3) to (hexagon.corner 1);
  \draw [red,thick] (hexagon.corner 6) to (hexagon.corner 1);

\end{tikzpicture}

  & 

\begin{tikzpicture}
	\tikzstyle{every path}=[draw] 
		\path
    node[
      regular polygon,
      regular polygon sides=6,
      draw,
      inner sep=0.9cm,
    ] (hexagon) {}
    %
    (hexagon.corner 1) node[above] {$1$}
    (hexagon.corner 2) node[above] {$2$}
    (hexagon.corner 3) node[left] {$3$}
    (hexagon.corner 4) node[below] {$4$}
    (hexagon.corner 5) node[below] {$5$}
    (hexagon.corner 6) node[right] {$6$}
    ;

  \foreach \i in {3,...,5}{
  \draw [-] (hexagon.corner 1) to (hexagon.corner \i);
  }
  
  \coordinate (m1) at ($0.35*(hexagon.corner 2)+0.25*(hexagon.corner 1)+0.4*(hexagon.corner 3)$);
  \coordinate (m2) at ($0.4*(hexagon.corner 3)+0.4*(hexagon.corner 1)+0.3*(hexagon.corner 4)$);
  \coordinate (m3) at ($0.4*(hexagon.corner 5)+0.4*(hexagon.corner 1)+0.3*(hexagon.corner 4)$);
  \coordinate (m4) at ($0.4*(hexagon.corner 5)+0.25*(hexagon.corner 1)+0.4*(hexagon.corner 6)$);

  \draw [name path = m ,opacity=1] (m1) to [out=-90,in=120] (m2) to [out=-180+120,in=180] (m3) to [out=-180+180,in=180+30] (m4);%
  \foreach \i in {1,...,4}{
  \draw [-] (m\i) to (hexagon.corner 1);}

  
  \draw [red,thick] (hexagon.corner 2) to (hexagon.corner 1);
  \draw [blue,thick] (m4) to (hexagon.corner 1);
  \draw [red,thick] (m2) to [out=-180+120,in=180] (m3) to [out=-180+180,in=180+30] (m4);
  \draw [blue,thick] (m2) to (hexagon.corner 1);
  \draw [red,thick] (hexagon.corner 6) to (hexagon.corner 1);

\end{tikzpicture}   
  
  & 
  
  \begin{tikzpicture}
	\tikzstyle{every path}=[draw] 
		\path
    node[
      regular polygon,
      regular polygon sides=6,
      draw,
      inner sep=0.9cm,
    ] (hexagon) {}
    %
    (hexagon.corner 1) node[above] {$1$}
    (hexagon.corner 2) node[above] {$2$}
    (hexagon.corner 3) node[left] {$3$}
    (hexagon.corner 4) node[below] {$4$}
    (hexagon.corner 5) node[below] {$5$}
    (hexagon.corner 6) node[right] {$6$}
    ;

  \foreach \i in {3,...,5}{
  \draw [-] (hexagon.corner 1) to (hexagon.corner \i);
  }
  
  \coordinate (m1) at ($0.35*(hexagon.corner 2)+0.25*(hexagon.corner 1)+0.4*(hexagon.corner 3)$);
  \coordinate (m2) at ($0.4*(hexagon.corner 3)+0.4*(hexagon.corner 1)+0.3*(hexagon.corner 4)$);
  \coordinate (m3) at ($0.4*(hexagon.corner 5)+0.4*(hexagon.corner 1)+0.3*(hexagon.corner 4)$);
  \coordinate (m4) at ($0.4*(hexagon.corner 5)+0.25*(hexagon.corner 1)+0.4*(hexagon.corner 6)$);

  \draw [name path = m ,opacity=1, red, thick] (m1) to [out=-90,in=120] (m2) to [out=-180+120,in=180] (m3) to [out=-180+180,in=180+30] (m4);%
  \foreach \i in {1,...,4}{
  \draw [-] (m\i) to (hexagon.corner 1);}

  
  \draw [red,thick] (hexagon.corner 2) to (hexagon.corner 1);
  \draw [blue,thick] (m4) to (hexagon.corner 1);
  \draw [blue,thick] (m1) to (hexagon.corner 1);
  \draw [red,thick] (hexagon.corner 6) to (hexagon.corner 1);

\end{tikzpicture} 
\end{tabular}
\end{center}
\end{center}
In this example, $\lambda_{26}$ can be expressed as
 \begin{gather*}
 \lambda_{26} = \frac{\lambda_{23}\lambda_{16}}{\lambda_{31}}
 + \frac{\lambda_{21}\lambda_{34}\lambda_{16}}{\lambda_{13}\lambda_{41}}
 + \frac{\lambda_{21}\lambda_{45}\lambda_{16}}{\lambda_{14}\lambda_{51}}
 + \frac{\lambda_{21}\lambda_{56}}{\lambda_{15}} \\
 \phantom{\lambda_{26}=} + \lambda_{21} \sqrt{\frac{\lambda_{23}}{\lambda_{12}\lambda_{13}}} \, \boxed{123} \, \sqrt{\frac{\lambda_{34}}{\lambda_{13}\lambda_{14}}} \, \boxed{134} \, \lambda_{16}
+ \lambda_{21} \sqrt{\frac{\lambda_{34}}{\lambda_{13}\lambda_{14}}} \, \boxed{134} \, \sqrt{\frac{\lambda_{45}}{\lambda_{14}\lambda_{15}}} \, \boxed{145} \, \lambda_{16} \\
 \phantom{\lambda_{26}=} + \lambda_{21} \sqrt{\frac{\lambda_{45}}{\lambda_{14}\lambda_{15}}} \, \boxed{145} \, \sqrt{\frac{\lambda_{56}}{\lambda_{15}\lambda_{16}}} \, \boxed{156} \, \lambda_{16}
+ \lambda_{21} \sqrt{\frac{\lambda_{23}}{\lambda_{12}\lambda_{13}}} \, \boxed{123} \, \sqrt{\frac{\lambda_{45}}{\lambda_{14}\lambda_{15}}} \, \boxed{145} \, \lambda_{16} \\
 \phantom{\lambda_{26}=} + \lambda_{21} \sqrt{\frac{\lambda_{34}}{\lambda_{13}\lambda_{14}}} \; \boxed{134} \; \sqrt{\frac{\lambda_{56}}{\lambda_{15}\lambda_{16}}} \, \boxed{156} , \lambda_{16}
+ \lambda_{21} \sqrt{\frac{\lambda_{23}}{\lambda_{12}\lambda_{13}}} \, \boxed{123} \, \sqrt{\frac{\lambda_{56}}{\lambda_{15}\lambda_{16}}} \, \boxed{156} \, \lambda_{16}.
 \end{gather*}
 The first four terms are the weights of all ordinary $T$-paths (as described in Example~\ref{exmp:T-Path-fan}),
 and the remaining six terms correspond to the $T$-paths pictured above (in the same order).
 It can also be written as
 \begin{gather*}
 \lambda_{26} = \frac{\lambda_{23}\lambda_{16}}{\lambda_{31}}
 + \frac{\lambda_{21}\lambda_{34}\lambda_{16}}{\lambda_{13}\lambda_{41}}
 + \frac{\lambda_{21}\lambda_{45}\lambda_{16}}{\lambda_{14}\lambda_{51}}
 + \frac{\lambda_{21}\lambda_{56}}{\lambda_{15}}
 + \lambda_{16}\lambda_{21} \, \widetilde{\boxed{123}} \; \widetilde{\boxed{134}}
 + \lambda_{16}\lambda_{21} \, \widetilde{\boxed{134}} \; \widetilde{\boxed{145} }\\
 \hphantom{\lambda_{26} =}{}
 + \lambda_{21}\lambda_{16}\, \widetilde{\boxed{145}} \; \widetilde{\boxed{156} }
 + \lambda_{16}\lambda_{21} \, \widetilde{\boxed{123}} \; \widetilde{ \boxed{145}}
 + \lambda_{16}\lambda_{21} \, \widetilde{\boxed{134}}\;\widetilde{\boxed{156}}
 + \lambda_{16}\lambda_{21}\,
 \widetilde{ \boxed{123} }\, \widetilde{ \boxed{156}}\,,
 \end{gather*}
 where the $\widetilde{\boxed{ijk}}$ are the weights of the corresponding $\sigma$-edges (following the notation of Corollary~\ref{cor:laurent}).
\end{Example}

\begin{proof}[Proof of Lemma~\ref{lm:super-path-single-fan}]
 A $\mathcal T^1_{2,n}$-path must use one of the $\sigma$-edges, therefore the first step in the super $T$-path needs to be $(2,1)$.
 Then after traveling through $\sigma_i$, what follows must be $\tau_{ij}$ which leads the $T$-path to another internal vertex $\theta_j$.
 The next step is an even step hence has to be a $\sigma$-edge which will take us to vertex $1$:
 \[(2,1,\theta_i,\theta_j,1,\dots\,|\,x_1,\sigma_i,\tau_{ij},\sigma_j,\dots).\]
 Now we are at vertex $1$ and have completed even number of steps, therefore the rest of the this super $T$-path must be a super $T$-path
 from~$1$ to~$n$~-- clearly there is only one possibility which is the single edge $(1,n)$.
 Hence a super $T$-path in $\mathcal T^1_{2,n}$ must have the form
 \begin{gather*}
 (2,1,\theta_i,\theta_j,1,\dots\,|\,x_1,\sigma_i,\tau_{ij},\sigma_j,\dots)+(\dots,1,n\,|\,\dots,x_n)\\
 \qquad{} =(2,1,\theta_i,\theta_j,1,n\,|\,x_1,\sigma_i,\tau_{ij},\sigma_j,x_n).\tag*{\qed}
 \end{gather*}\renewcommand{\qed}{}
\end{proof}

\section{Proof of Theorem~\ref{thm:main}}
\label{sec:proof}

In this section we prove our main theorem.
It turns out that the default orientation guarantees a positive sign on all terms of the expansion of
$\lambda$-lengths (Theorem~\ref{thm:main}), and is also preserved by induction.

Our proof has three parts: we first prove the case of single fan triangulations, and then prove the case of zig-zag triangulations. Finally we prove Theorem~\ref{thm:main} in full generality by combining the two cases mentioned above.

Before proving our theorem, we state the following results that will be used in our proofs.

\begin{Proposition} \label{lemma:theta32}
 Let $A$, $\beta,$ and $\Sigma$ be elements in the super algebra $\mathcal A$, which for convenience, we assume is written as
 $\mathcal A = {\mathbb R}\bigl[x_1^{\pm\frac{1}{2}},\dots,x_{2n+3}^{\pm\frac{1}{2}} \,|\, \theta_1,\dots,\theta_{n+1}\bigr]$ as in Definition~{\rm \ref{def:weight}}.
 Further, we assume that $A$ is an even element with a non-zero body,\footnote{The \emph{body} of an element $A$ of super algebra $\mathcal A$ is the constant term when expanded out in terms of the $\theta_i$'s.}
 and that both $\beta$ and $\Sigma$ are odd elements of~$\mathcal A$. Then we have
 \[\sqrt{A+\beta\Sigma}= \sqrt{A}+\frac{\beta}{2\sqrt{A}}\Sigma,\]
 where the square root of $A$ is taken to be the positive square root.
\end{Proposition}

\begin{Remark}\rm
 Since $A$ is an even element, its positive square root is well-defined as the choice such that the body of $\sqrt{A}$ is the positive square root of the body of $A$.
\end{Remark}

\begin{proof}
 Squaring the right-hand side:
 \[\left(\sqrt{A}+\frac{\beta}{2\sqrt{A}}\Sigma\right)^2=A+2\sqrt{A}\cdot \frac{\beta}{2\sqrt{A}}\cdot \Sigma +\left(\frac{\beta}{2\sqrt{A}}\Sigma\right)^2=A+\beta\Sigma.\]
 This clearly equals the square of the left-hand side.
\end{proof}

Using Lemma~\ref{lemma:theta32}, we can rewrite the super Ptolemy relations in a more symmetrical form.
\begin{Proposition}
The super Ptolemy relations described in Figure~{\rm \ref{fig:super_ptolemy}} can be written as follows
\begin{gather}
\theta'\sqrt{ef} =	\theta\sqrt{bd}+\sigma \sqrt{ac},\label{eq:5}\\
\sigma'\sqrt{ef} =\sigma\sqrt{bd}-\theta\sqrt{ac}.\label{eq:6}
\end{gather}
\end{Proposition}
\begin{proof}We have
\begin{gather*}
\theta'\sqrt{ef} =\theta'\sqrt{ab+cd+\sqrt{abcd}\sigma\theta}
 =\theta'\sqrt{ab+cd+\sqrt{abcd}\sigma'\theta'}
 =\theta'\sqrt{ab+cd}
\end{gather*}
and
\begin{gather*}
	\theta' = {{ \theta\sqrt{bd}+\sigma \sqrt{ac}}\over{\sqrt{ac+bd}}},\qquad
	\theta'\sqrt{ac+bd} = \theta\sqrt{bd}+\sigma \sqrt{ac}.
\end{gather*}
	
Putting these two equations together gives equation~\eqref{eq:5}:
\begin{align*}
\theta'\sqrt{ef}= \theta\sqrt{bd}+\sigma \sqrt{ac}.
\end{align*}
Equation~\eqref{eq:6} can be derived in a similar way.
\end{proof}

\subsection{Proof of Theorem~\ref{thm:main} for a single fan}

For sake of readability, in the below, we use $\boxed{ijk}$ to denote the $\mu$-invariant associated to the triple of
ideal points $(i,j,k)$, subject to Remark~\ref{rem:mu-invariants}, while the $\lambda$-length of a pair $(i,j)$ will be
denoted $\lambda_{ij}$. We will also sometimes use $\boxed{ijk}$ to denote the internal vertex of the auxiliary graph associated to $(i,j,k)$, when talking about super $T$-paths.

First, we will prove the main theorem in the case of a single fan triangulation.

\begin{figure}[h] \centering
 \begin{tikzpicture}[decoration={
    markings,
    mark=at position 0.5 with {\arrow{>}}}
    ] 

\tikzstyle{every path}=[draw] 
		\path
    node[
      regular polygon,
      regular polygon sides=8,
      draw=none,
      inner sep=2.1cm,
    ] (T) {}
    %
    (T.corner 1) node[above] {$n$}
    (T.corner 2) node[above] {$1$}
    (T.corner 3) node[left]  {$2$}
    (T.corner 4) node[left]  {$3$}
    (T.corner 5) node[below] {$4$}
    (T.corner 6) node[below] {$5$}
    (T.corner 7) node[right] {}
    (T.corner 8) node[right] {$n-1$}
;

\draw [-] (T.corner 6) to (T.corner 5) to (T.corner 4) to (T.corner 3) to (T.corner 2) to (T.corner 1) to (T.corner 8);
\draw [dashed] (T.corner 8) to (T.corner 7) to (T.corner 6);
\foreach \i in {4,5,6,8}{
    \draw []  (T.corner 2) to (T.corner \i);
}
\draw [dashed] (T.corner 2) to (T.corner 7);

\foreach \x in {1,2,...,7}{
\draw (T.corner \x) node [fill,circle,scale=0.2] {};}

\coordinate (m1) at ($0.3*(T.corner 2)+0.3*(T.corner 4)+0.4*(T.corner 3)$);
\coordinate (m2) at ($0.33*(T.corner 5)+0.33*(T.corner 4)+0.33*(T.corner 2)$);
\coordinate (m3) at ($0.33*(T.corner 5)+0.33*(T.corner 2)+0.33*(T.corner 6)$);
\coordinate (m4) at ($0.33*(T.corner 7)+0.33*(T.corner 2)+0.33*(T.corner 6)$);
\coordinate (m5) at ($0.33*(T.corner 7)+0.33*(T.corner 2)+0.33*(T.corner 8)$);
\coordinate (m6) at ($0.4*(T.corner 1)+0.3*(T.corner 2)+0.3*(T.corner 8)$);

\draw (m1) node [fill,circle,scale=0.3] {};
\draw (m2) node [fill,circle,scale=0.3] {};
\draw (m3) node [fill,circle,scale=0.3] {};
\draw (m4) node [fill,circle,scale=0.3] {};
\draw (m5) node [fill,circle,scale=0.3] {};
\draw (m6) node [fill,circle,scale=0.3] {};

\draw (T.corner 2) -- (m1) node [near end, fill=white,inner sep=0.2] {$\sigma_1$};
\draw (T.corner 2) -- (m2) node [near end, fill=white,inner sep=0.2] {$\sigma_2$};
\draw (T.corner 2) -- (m3) node [near end, fill=white,inner sep=0.2] {$\sigma_3$};
\draw[dashed] (T.corner 2) -- (m4);
\draw[dashed] (T.corner 2) -- (m5);
\draw (T.corner 2) -- (m6) node [near end, fill=white,inner sep=0.2] {$\sigma_{n-2}$};

\end{tikzpicture}
 \caption{Proof of Theorem~\ref{thm:mu_single_fan}.} \label{fig:mu_single_fan}
\end{figure}
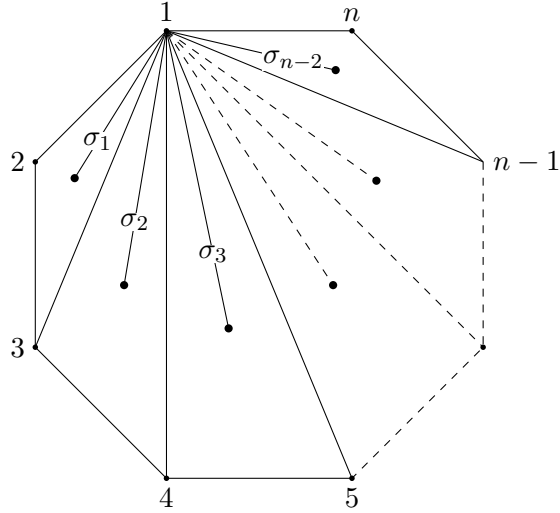

\begin{Theorem} \label{thm:mu_single_fan}
 Consider a single fan triangulation of an $n$-gon as depicted in Figure~{\rm \ref{fig:mu_single_fan}}. We continue our convention of using the default orientation,
 which for a fan triangulation means that arrows on all internal diagonals point away from the fan center.
 After performing the super Ptolemy relations for the flips of arcs $(1,3)$, $(1,4)$, $\dots$, $(1,k-1)$, we have
 \begin{enumerate}\itemsep=0pt
 \item[$(a)$] $\displaystyle \sqrt{\frac{\l2k}{\l12\l1k}} \; \boxed{12k} = \sum_{i=1}^{k-2}\wt(\sigma_i)$,
 \item[$(b)$] $\displaystyle \lambda_{2k} = \sum_{t \in \mathcal{T}_{2,k}} {\rm wt}(t)$,
 \end{enumerate}
 where ${\rm wt}(\sigma_i)$ and ${\rm wt}(t)$ are defined in Definition~{\rm \ref{def:weight}}, and as defined in Definition~{\rm \ref{def:super-T}}, $\mathcal{T}_{2,k}$ denotes the set of super $T$-paths from vertex $2$ to vertex~$k$.
\end{Theorem}

\begin{proof}
 We will induct on $k$. We begin with the case of $k=3$ where statements (a) and (b) follow immediately since
 ${\rm wt}(\sigma_1) = \sqrt{\frac{\lambda_{23}}{\lambda_{12}\lambda_{13}}} \, \boxed{123}$\,, and the unique super $T$-path from vertex $2$ to vertex $3$
 is simply the ordinary $T$-path $t = (2,3)$.

 In general, for $k\geq 4$, after flipping arcs $(1,3),\dots,(1,k-2)$, the next flip of arc $(1,k-1)$ will be inside the following quadrilateral:
 
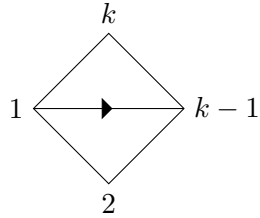
\begin{center}
 \begin{tikzpicture}
 \draw (1,0) -- (0,1) -- (-1,0) -- (0,-1) -- cycle;
 \draw (-1,0) --node {\midarrow} (1,0);

 \draw (-1,0) node[left] {$1$};
 \draw (1,0) node[right] {$k-1$};
 \draw (0,-1) node[below] {$2$};
 \draw (0,1) node[above] {$k$};
 \end{tikzpicture}
 \end{center}
 Note that the $\mu$-invariant $\boxed{1,k-1,k}$ is in the initial triangulation, but $\boxed{1,2,k-1}$ is not\footnote{Except for the special case of $k=4$.}.
 However, we assume by induction that $\boxed{1,2,k-1}$ is given by part $(a)$.

 First we prove part (b).
 The super Ptolemy relation (equation~\eqref{eq:mutation}) says that $\lambda_{2k}$ is given by
 \[
 \lambda_{1,k-1}\lambda_{2k} = \lambda_{12}\lambda_{k-1,k} + \lambda_{1k}\lambda_{2,k-1}
 + \sqrt{\lambda_{12}\lambda_{2,k-1}\lambda_{k-1,k}\lambda_{1k}} \, \boxed{1,2,k-1} \, \boxed{1,k-1,k}\,.
 \]
 After dividing by $\lambda_{1,k-1}$, we get a formula for $\lambda_{2,k}$. On the right-hand side, the first term gives
 $\lambda_{21}\lambda_{1,k-1}^{-1}\lambda_{k-1,k}$, which is clearly an ordinary $T$-path. The second term, by induction, is
 \[ \sum_{t \in \mathcal{T}_{2,k-1}} {\rm wt}(t) \lambda_{k-1,1}^{-1} \lambda_{1k}. \]
 Taking the ordinary $T$-paths in $\mathcal{T}^0_{2,k-1}$, and appending the arcs $(k-1,1)$ and $(1,k)$ give the rest of the ordinary $T$-paths
 in $\mathcal{T}^0_{2k}$. (See Lemma~\ref{lemma:Schiffler} and note that appending arc $(k-1,1)$ as an even step may in fact yield a backtrack that cancels out the final step of an ordinary $T$-path in~$\mathcal{T}^0_{2k}$.) The terms coming from $\mathcal{T}^1_{2,k-1}$, multiplied by~$\frac{\lambda_{1k}}{\lambda_{1,k-1}}$, similarly give the super $T$-paths
 in~$\mathcal{T}^1_{2k}$ which do not involve the triangle~$(1,k-1,k)$ (see Lemma~\ref{lm:super-path-single-fan}).

 The last term, using part (a) and induction on $\boxed{1,2,k-1}$\,, is
 \begin{align*}
 \frac{\sqrt{\lambda_{12}\lambda_{2,k-1}\lambda_{k-1,k}\lambda_{1k}}}{\lambda_{1,k-1}} \, \boxed{1,2,k-1} \, \boxed{1,k-1,k}
 &= \lambda_{21} \sum_{i=1}^{k-3} {\rm wt}(\sigma_i) {\rm wt}(\sigma_{k-2}) \lambda_{1k}.
 \end{align*}
 These are the weights of the super $T$-paths which use the last triangle, namely $(1,k-1,k)$ (again, see Lemma~\ref{lm:super-path-single-fan}).

 Note that the product ${\rm wt}(\sigma_i) {\rm wt}(\sigma_{k-2})$ is in the positive ordering,
 since the two $\mu$-invariants appear in the counter-clockwise order around the fan center.

 Now, we examine part (a). Looking at the same quadrilateral as above, the super Ptolemy relation for $\mu$-invariants (equation~\eqref{eq:5}) says that
 \[ \boxed{12k} \, \sqrt{\lambda_{1,k-1}\lambda_{2k}} = \boxed{1,k-1,k} \, \sqrt{\lambda_{12}\lambda_{k-1,k}} + \boxed{1,2,k-1} \, \sqrt{\lambda_{1k}\lambda_{2,k-1}}. \]

 Dividing by $\sqrt{\l12\lambda_{1,k-1}\l1k}$, commuting the $\lambda$-lengths past the $\mu$-invariants, we get
 \[\sqrt{\l2k\over \l12\l1k }\, \boxed{12k}=\sqrt{\l k{,k-1} \over \l1{,k-1}\l1k }\,\boxed{1,k-1,k} + \sqrt{\l2{,k-1} \over\l12\l1{,k-1}}\,\boxed{1,2,k-1}\,.\]
 The first term on the right-hand side is simply ${\rm wt}(\sigma_{k-2})$. By induction, the second term is $\sum_{i=1}^{k-3} {\rm wt}(\sigma_i)$. Therefore we have
 \[\sqrt{\l2k\over \l12\l1k }\, \boxed{12k}=\wt(\sigma_{k-2}) + \left(\sum_{i=1}^{k-3}\wt(\sigma_i)\right)=\sum_{i=1}^{k-2}\wt(\sigma_{i})\]
 as desired.
\end{proof}

\subsection{Proof of Theorem~\ref{thm:main} for a zig-zag triangulation}

Next, we prove our main theorem for the case of a zig-zag triangulation.

\begin{Theorem} \label{thm:proof_zigzag}
 Consider a zigzag triangulation $T$ of an $n$-gon as depicted in Figure~{\rm \ref{fig:proof_zigzag}}. We consider all the vertices
 except for $1$ and $n$ to be fan centers, so that $c_{i}$ is labelled $i+1$ for $1\leq i\leq n-2$.

 After flipping the arcs $(n-1,n-2),(n-2,n-3),\dots,(k+2,k+1)$, we have
 \begin{enumerate}\itemsep=0pt
 \item[$(a)$] $\displaystyle \sqrt{\lambda_{kn}\lambda_{k+1,n}\over\lambda_{k,k+1}} \, \boxed{k,k+1,n} = \sum_{i=k}^{n-2} \sum_{t\in\mathcal T_{n,i+1}} \wt(t) \wt(\sigma_i)$,
 \item[$(b)$] $\displaystyle \lambda_{k,n} = \sum_{t \in \mathcal{T}_{kn}} {\rm wt}(t)$.
 \end{enumerate}
\end{Theorem}

\begin{figure}[h]\centering
\begin{tikzpicture}[decoration={
    markings,
    mark=at position 0.5 with {\arrow{>}}}
    ]

\tikzstyle{every path}=[draw]
		\path
    node[
      regular polygon,
      regular polygon sides=10,
      draw=none,
      inner sep=1.9cm,
    ] (T) {}
    %
    (T.corner 1) node[above] {$1$}
    (T.corner 2) node[above] {$2$}
    (T.corner 3) node[left] {$4$}
    (T.corner 4) node[left] {$6$}
    (T.corner 5) node[left] {$n-2$}
    (T.corner 6) node[below] {$n$}
    (T.corner 7) node[below] {$n-1$}
    (T.corner 8) node[right] {$7$}
    (T.corner 9) node[right] {$5$}
    (T.corner 10) node[right] {$3$}
;

\draw[-] (T.corner 4) to (T.corner 3) to (T.corner 2) to (T.corner 1) to (T.corner 10) to (T.corner 9) to (T.corner 8) ;
\draw[dashed] (T.corner 8) to (T.corner 7);
\draw[-] (T.corner 7) to (T.corner 6) to (T.corner 5);
\draw[dashed] (T.corner 5) to (T.corner 4);

\draw [postaction={decorate}] (T.corner 2) to node[xshift=0.5em,yshift=0.4em] {{\small $\theta_1$}} (T.corner 10) ;
\draw [postaction={decorate}] (T.corner 10) to node[xshift=-0.3em,yshift=0.9em] {{\small $\theta_2$}} (T.corner 3);
\draw [postaction={decorate}] (T.corner 3) to node[xshift=1em,yshift=0.8em] {{\small $\theta_3$}} (T.corner 9);
\draw [postaction={decorate}] (T.corner 9) to node[xshift=-1.2em,yshift=1em] {{\small $\theta_4$}} (T.corner 4);
\draw [postaction={decorate}] (T.corner 4) to node[xshift=1.5em,yshift=0.6em] {{\small $\theta_5$}} (T.corner 8);
\draw[dashed] (T.corner 8) to (T.corner 5);
\draw[postaction={decorate}] (T.corner 5) to node[xshift=0.2em,yshift=-0.7em] {{\small $\theta_{n-2}$}} (T.corner 7);

\foreach \x in {1,2,...,10}{
\draw (T.corner \x) node [fill,circle,scale=0.2] {};}
\end{tikzpicture}
\;\;\;
\begin{tikzpicture}[decoration={
    markings,
    mark=at position 0.5 with {\arrow{>}}}
    ]

\tikzstyle{every path}=[draw]
		\path
    node[
      regular polygon,
      regular polygon sides=10,
      draw=none,
      inner sep=1.9cm,
    ] (T) {}
    %
    (T.corner 1) node[above] {$1$}
    (T.corner 2) node[above] {$2$}
    (T.corner 3) node[left] {$4$}
    (T.corner 4) node[left] {$6$}
    (T.corner 5) node[left] {$n-2$}
    (T.corner 6) node[below] {$n$}
    (T.corner 7) node[below] {$n-1$}
    (T.corner 8) node[right] {$7$}
    (T.corner 9) node[right] {$5$}
    (T.corner 10) node[right] {$3$}
;

\draw[-] (T.corner 4) to (T.corner 3) to (T.corner 2) to (T.corner 1) to (T.corner 10) to (T.corner 9) to (T.corner 8) ;
\draw[dashed] (T.corner 8) to (T.corner 7);
\draw[-] (T.corner 7) to (T.corner 6) to (T.corner 5);
\draw[dashed] (T.corner 5) to (T.corner 4);

\draw [] (T.corner 2) to  (T.corner 10) ;
\draw [] (T.corner 10) to (T.corner 3);
\draw [] (T.corner 3) to  (T.corner 9);
\draw [] (T.corner 9) to  (T.corner 4);
\draw  (T.corner 4) to (T.corner 8);
\draw[dashed] (T.corner 8) to (T.corner 5);
\draw[] (T.corner 5) to (T.corner 7);

\coordinate (m1) at ($0.3*(T.corner 2)+0.4*(T.corner 1)+0.34*(T.corner 10)$);
\coordinate (m2) at ($0.33*(T.corner 2)+0.33*(T.corner 10)+0.33*(T.corner 3)$);
\coordinate (m3) at ($0.33*(T.corner 10)+0.33*(T.corner 9)+0.33*(T.corner 3)$);
\coordinate (m4) at ($0.33*(T.corner 9)+0.33*(T.corner 4)+0.33*(T.corner 3)$);
\coordinate (m5) at ($0.33*(T.corner 4)+0.33*(T.corner 8)+0.33*(T.corner 9)$);
\coordinate (m6) at ($0.33*(T.corner 4)+0.33*(T.corner 8)+0.33*(T.corner 5)$);
\coordinate (m7) at ($0.33*(T.corner 7)+0.33*(T.corner 8)+0.33*(T.corner 5)$);
\coordinate (m8) at ($0.3*(T.corner 7)+0.34*(T.corner 6)+0.4*(T.corner 5)$);

\foreach \t in {1,...,8}{
 \draw (m\t) node [fill,circle,scale=0.3] {};}

\draw (T.corner 2) -- (m1) node [near end, fill=white,inner sep=0.2] {$\sigma_{1}$};
\draw (T.corner 10)--(m2) node [near end, fill=white,inner sep=0.2] {$\sigma_{2}$};
\draw (T.corner 3)--(m3) node [near end, fill=white,inner sep=0.2] {$\sigma_{3}$};
\draw (T.corner 9)--(m4) node [near end, fill=white,inner sep=0.2] {$\sigma_{4}$};
\draw (T.corner 4)--(m5) node [near end, fill=white,inner sep=0.2] {$\sigma_{5}$};
\draw [dashed] (T.corner 8)--(m6);
\draw [dashed] (T.corner 5)--(m7);
\draw (T.corner 7) -- (m8) node [yshift=-0.14cm,xshift=0.5cm, fill=white,inner sep=0] {$\sigma_{n-2}$} ;

\foreach \x in {1,2,...,10}{
\draw (T.corner \x) node [fill,circle,scale=0.2] {};}
\end{tikzpicture}
\caption{Proof of Theorem~\ref{thm:proof_zigzag}.
{Left:} zig-zag triangulation with the default orientation.
{Right:} the corresponding auxiliary graph.}\label{fig:proof_zigzag}
\end{figure}
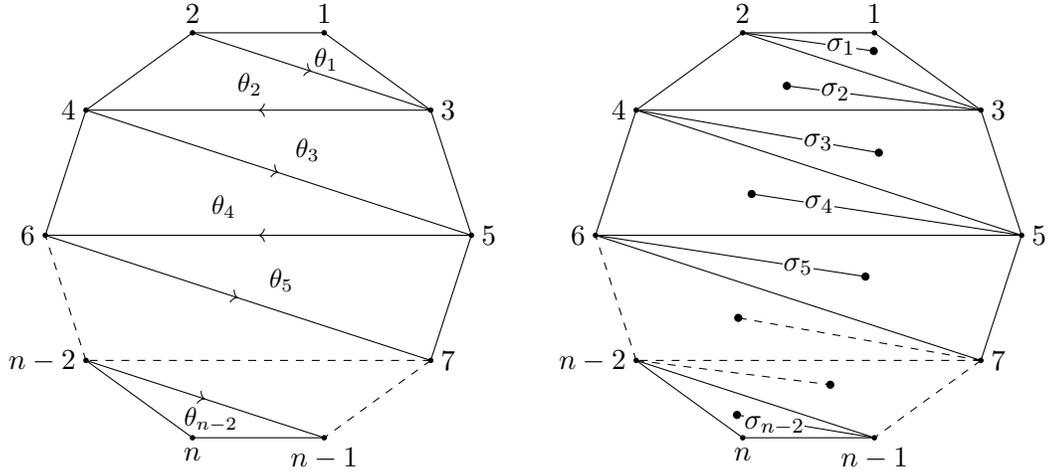

\begin{proof}
 We assume that vertices $n$, $n-1$, and $n-2$ are oriented as in Figure~\ref{fig:proof_zigzag}. The case that they are oriented oppositely is similar.

 We will induct (backwards) on $k$.
 The base case of the induction is the case of a single triangle, when $k=n-2$.
 Since the edge $(n-2,n)$ is already in the triangulation, $\lambda_{n-2,n}$ is already one of the generators. Clearly the only $T$-path
 in this case is the single edge $(n-2,n)$ (when zero flips have been performed). This establishes part (b) for the base case.
 For part (a), the left-hand side is
 \[ \sqrt{\frac{\lambda_{n-2,n}\lambda_{n-1,n}}{\lambda_{n-1,n-2}}} \, \boxed{n-2,n-1,n}\,. \]
 For the right-hand side, there is only a single term in this sum. This is because $i$ only takes the value $n-2$,
 and the only $T$-path from $n-1$ to $n$ is the single edge $(n-1,n)$. Thus the right-hand side is
 \begin{align*}
 \lambda_{n-1,n}   \wt(\sigma_{n-2}) &= \lambda_{n-1,n}   \sqrt{\frac{\lambda_{n-2,n}}{\lambda_{n-2,n-1}\lambda_{n-1,n}}} \, \boxed{n-2,n-1,n} \\
 &= \sqrt{\frac{\lambda_{n-2,n}\lambda_{n-1,n}}{\lambda_{n-2,n-1}}} \, \boxed{n-2,n-1,n}\,.
 \end{align*}
 This establishes part (a) for the base case. We now assume that $1 \leq k \leq n-3$.

 After flipping the arcs $(n-1,n-2)$, $(n-2,n-3), \dots, (k+3,k+2)$,\footnote{For the case of $k=n-3$, no arcs have yet been flipped.} we will have one of the two following quadrilaterals, depending on the parity of $n-k$:
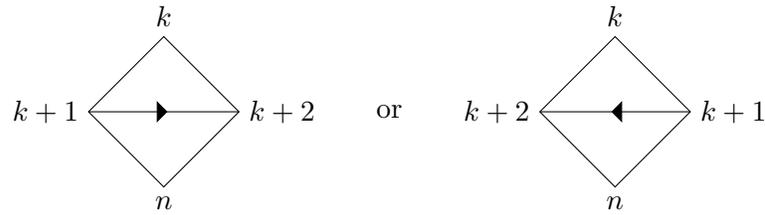
 \begin {center}
 \begin {tikzpicture}
 \draw (1,0) -- (0,1) -- (-1,0) -- (0,-1) -- cycle;
 \draw (-1,0) --node {\midarrow} (1,0);
 \draw (-1,0) node[left] {$k+1$};
 \draw (1,0) node[right] {$k+2$};
 \draw (0,-1) node[below] {$n$};
 \draw (0,1) node[above] {$k$};
 \draw (3,0) node {or};
 \begin {scope}[shift={(6,0)}]
 \draw (1,0) -- (0,1) -- (-1,0) -- (0,-1) -- cycle;
 \draw (1,0) --node {\reflectbox{\midarrow}} (-1,0);
 \draw (-1,0) node[left] {$k+2$};
 \draw (1,0) node[right] {$k+1$};
 \draw (0,-1) node[below] {$n$};
 \draw (0,1) node[above] {$k$};
 \end {scope}
 \end {tikzpicture}
 \end {center}

 Now we flip the edge $(k+2, k+1)$ while applying the Ptolemy relation equation~\eqref{eq:5}. In both cases pictured above, the triangle $(k,k+1,n)$ will play
 the role of $\theta'$ (it will be on the left, looking in the direction of the arrow after the flip).
 And because of the opposite orientations, application of equation~\eqref{eq:5} in both pictures gives
 \begin{gather*}
 \sqrt{\lambda_{kn} \lambda_{k+1,k+2}} \, \boxed{k,k+1,n} \\
 \qquad{} = \sqrt{\lambda_{k+1,n} \lambda_{k,k+2}} \, \boxed{k,k+1,k+2} + \sqrt{\lambda_{k,k+1} \lambda_{k+2,n}} \, \boxed{k+1,k+2,n}\,.
 \end{gather*}

 Multiplying both sides by $\sqrt{\frac{\lambda_{k+1,n}}{\lambda_{k,k+1}\lambda_{k+1,k+2}}}$, we get
 \begin{gather*}
 \begin{split}
 & \sqrt{\frac{\lambda_{kn} \lambda_{k+1,n}}{\lambda_{k,k+1}}} \, \boxed{k,k+1,n}\\
& \qquad{}
 = \lambda_{k+1,n} \sqrt{\frac{\lambda_{k,k+2}}{\lambda_{k,k+1}\lambda_{k+1,k+2}}} \, \boxed{k,k+1,k+2}
 + \sqrt{\frac{\lambda_{k+1,n} \lambda_{k+2,n}}{\lambda_{k+1,k+2}}} \, \boxed{k+1,k+2,n}\,.
 \end{split}
 \end{gather*}

 First we examine the first term on the right-hand side. By induction, $\lambda_{k+1,n}$ is the weighted sum of super $T$-paths from $k+1$ to $n$.
 Notice that $\sqrt{\frac{\lambda_{k,k+2}}{\lambda_{k,k+1} \lambda_{k+1,k+2}}} \, \boxed{k,k+1,k+2}$ is the weight of the $\sigma$-step $\sigma_k$, going from
 vertex $k+1$ into the triangle labelled $\theta_k$.
 Therefore this first term is equal to
 \begin{equation} \label{eq:zigzag_1st_term}
 \sum_{t \in \mathcal{T}_{k+1,n}} \wt(t) \wt(\sigma_k)\,.
 \end{equation}

Next, we focus on the second term of the right-hand side: $\sqrt{\frac{\lambda_{k+1,n} \lambda_{k+2,n}}{\lambda_{k+1,k+2}}} \, \boxed{k+1,k+2,n}\,$.
 By induction, this is equal to
 \begin{equation*} 
 \sqrt{\frac{\lambda_{k+1,n} \lambda_{k+2,n}}{\lambda_{k+1,k+2}}} \, \boxed{k+1,k+2,n} = \sum_{i=k+1}^{n-2} \sum_{t \in \mathcal T_{n,i+1}} \wt(t) \wt(\sigma_i).
 \end{equation*}
 Adding the terms from equation~\eqref{eq:zigzag_1st_term} gives the result for part~(a).

 We should verify that the $\mu$-invariants in each product $\wt(t)\wt(\sigma_k)$ are in the correct positive ordering. {We use the same viewpoint as in Remark~\ref{rem:pos_order}.} By induction, the $\mu$-invariants
 in each term of $\wt(t)$ occur in the positive ordering of the smaller polygon which contains triangles $\theta_{k+1},\dots,\theta_{n-2}$.
 By construction, the positive ordering of the slightly larger polygon (which additionally contains $\theta_k$) is obtained from the smaller by
 placing $\theta_k$ either at the beginning or end of the previous order. But since $\wt(t)$ is an even element, it is central, and so
 $\wt(t)\wt(\sigma_k) = \wt(\sigma_k)\wt(t)$, and we may choose whichever one gives the correct positive ordering.

 Now we will prove part (b). Again after flipping the arcs $(n-1,n-2), (n-2,n-3), \dots$, $(k+3,k+2)$, we are in one of the
 two pictures from before. For the picture on the left, an application of equation~\eqref{eq:mutation}, followed by division on both sides by $\lambda_{k+1,k+2}$, gives
 \begin{gather*}
 \l kn =
 \underbrace{
 \frac{\lambda_{k,k+2} \lambda_{k+1,n}}{\lambda_{k+1,k+2}}}_{\text{part 1}}
 +
 \underbrace{\frac{\lambda_{k,k+1} \lambda_{k+2,n}}{\lambda_{k+1,k+2}}}_{\text{part 2}}\nonumber \\
 \hphantom{\l kn =}{}+
 \underbrace{\sqrt{\frac{\lambda_{k+1,n} \lambda_{k+2,n}}{\lambda_{k+1,k+2}}}\; \boxed{k+1,k+2,n}}_{\text{part 3}}
 \cdot
 \underbrace{\sqrt{\frac{\lambda_{k,k+1} \lambda_{k,k+2}}{\lambda_{k+1,k+2}}}\; \boxed{k,k+1,k+2}}_{\text{part 4}}.
 \end{gather*}
 Note that $\boxed{k+1,k+2,n}$ lies on the right of the oriented diagonal, and hence plays the role of~$\sigma$. Analogously,
 $\boxed{k,k+1,k+2}$ lies on the left and plays the role of $\theta$. Applying equation~\eqref{eq:mutation} to the picture on the right gives the same thing, except that parts~(3) and~(4) appear
 in the opposite order.

 First we will explain that the sum of parts~(1) and~(2) is the weighted sum of all super $T$-paths from $k$ to $n$
 that do not contain a $\tau$-edge which crosses $(k+1,k+2)$.

 Suppose a path $t \in \mathcal{T}_{kn}$ does not contain a $\tau$ edge crossing $(k+1,k+2)$ (i.e., does not have a super step starting with $\theta_k$).
 Then the first step of $t$ is either edge $(k,k+1)$ or $(k,k+2)$, and the remainder of $t$ lies in the smaller polygon below the diagonal $(k+1,k+2)$.
 There are two cases.

 The second edge might be $(k+1,k+2)$, in which case the remainder of $t$ (after removing the first two edges) is a super $T$-path
 in either $\mathcal{T}_{k+1,n}$ or $\mathcal{T}_{k+2,n}$. This is depicted in the left of Figure~\ref{fig:zigzag_cases}.
 Clearly all of these paths occur as terms in parts~(1) and~(2).

 However, there are more terms in parts (1) and~(2);
 namely those in which the denominator $\lambda_{k+1,k+2}$ cancels a contribution in the numerator. This is the
 second case, in which the second step in $t$ is \emph{not} the edge $(k+1,k+2)$. This is depicted in the right of Figure~\ref{fig:zigzag_cases}.
 In this case, replacing the first edge of~$t$ with $(k+1,k+2)$ gives
 a super $T$-path in either $\mathcal{T}_{k+1,n}$ or $\mathcal{T}_{k+2,n}$, and these are the remaining terms in parts~(1) and~(2) in which the denominator cancels.

 \begin{figure}[h]\centering
 \begin{tikzpicture}[decoration={markings,mark=at position 0.5 with {\arrow{>}}}, scale=0.8, every node/.style={scale=0.8}]

\tikzstyle{every path}=[draw] 
		\path
    node[
      regular polygon,
      regular polygon sides=10,
      draw=none,
      inner sep=2cm,
    ] (T) {}
    %
    (T.corner 1) node[above] {$1$}
    (T.corner 3) node[left]  {$k$}
    (T.corner 4) node[left]  {$k+2$}
    (T.corner 9) node[right] {$k+1$}
;

\draw[-] (T.corner 4) to (T.corner 3) to (T.corner 2) to (T.corner 1) to (T.corner 10) to (T.corner 9) to (T.corner 8) ;
\draw[-] (T.corner 8) to (T.corner 7);
\draw[-] (T.corner 7) to (T.corner 6) to (T.corner 5);
\draw[-] (T.corner 5) to (T.corner 4);

\draw [postaction={decorate}] (T.corner 2)  to  (T.corner 10);
\draw [postaction={decorate}] (T.corner 10) to  (T.corner 3);
\draw [postaction={decorate}] (T.corner 3)  to  (T.corner 9);
\draw [postaction={decorate}] (T.corner 9)  to  (T.corner 4);
\draw [postaction={decorate}] (T.corner 4)  to  (T.corner 8);
\draw [postaction={decorate}] (T.corner 8)  to  (T.corner 5);
\draw []                      (T.corner 8)  to  (T.corner 5);
\draw [postaction={decorate}] (T.corner 5)  to  (T.corner 7);

\foreach \x in {1,2,...,10}{
\draw (T.corner \x) node [fill,circle,scale=0.2] {};}

\coordinate (m1) at ($0.3*(T.corner 2)+0.4*(T.corner 1)+0.34*(T.corner 10)$);
\coordinate (m2) at ($0.33*(T.corner 2)+0.33*(T.corner 10)+0.33*(T.corner 3)$);
\coordinate (m3) at ($0.33*(T.corner 10)+0.33*(T.corner 9)+0.33*(T.corner 3)$);
\coordinate (m4) at ($0.33*(T.corner 9)+0.33*(T.corner 4)+0.33*(T.corner 3)$);
\coordinate (m5) at ($0.33*(T.corner 4)+0.33*(T.corner 8)+0.33*(T.corner 9)$);
\coordinate (m6) at ($0.33*(T.corner 4)+0.33*(T.corner 8)+0.33*(T.corner 5)$);
\coordinate (m7) at ($0.33*(T.corner 7)+0.33*(T.corner 8)+0.33*(T.corner 5)$);
\coordinate (m8) at ($0.3*(T.corner 7)+0.34*(T.corner 6)+0.4*(T.corner 5)$);

\draw (T.corner 2) -- (m1);
\draw (T.corner 10)--(m2);
\draw (T.corner 3)--(m3);
\draw (T.corner 9)--(m4);
\draw (T.corner 4)--(m5);
\draw (T.corner 8)--(m6);
\draw (T.corner 5)--(m7);
\draw (T.corner 7) -- (m8);

\foreach \t in {1,...,8}{
 \draw (m\t) node [fill,circle,scale=0.3] {};}

\draw [red, line width = 1.2]  (T.corner 3) -- (T.corner 9);
\draw [blue, line width = 1.2] (T.corner 9) -- (T.corner 4);
\draw [red, line width = 1.2]  (T.corner 4) -- (T.corner 8);
\draw [blue, line width = 1.2] (T.corner 8) -- (m6);
\draw [red, line width = 1.2]  (m6) -- (m7);
\draw [blue, line width = 1.2] (m7) -- (T.corner 5);
\draw [red, line width = 1.2]  (T.corner 5) -- (T.corner 6);

\end{tikzpicture}
\;\;\;
\begin{tikzpicture}[decoration={
    markings,
    mark=at position 0.5 with {\arrow{>}}}, scale=0.8, every node/.style={scale=0.8}
    ] 

\tikzstyle{every path}=[draw] 
		\path
    node[
      regular polygon,
      regular polygon sides=10,
      draw=none,
      inner sep=2cm,
    ] (T) {}
    %
    (T.corner 1) node[above] {$1$}
    (T.corner 3) node[left]  {$k$}
    (T.corner 4) node[left]  {$k+2$}
    (T.corner 9) node[right] {$k+1$}
;

\draw[-] (T.corner 4) to (T.corner 3) to (T.corner 2) to (T.corner 1) to (T.corner 10) to (T.corner 9) to (T.corner 8) ;
\draw[-] (T.corner 8) to (T.corner 7);
\draw[-] (T.corner 7) to (T.corner 6) to (T.corner 5);
\draw[-] (T.corner 5) to (T.corner 4);

\draw [postaction={decorate}] (T.corner 2)  to  (T.corner 10);
\draw [postaction={decorate}] (T.corner 10) to  (T.corner 3);
\draw [postaction={decorate}] (T.corner 3)  to  (T.corner 9);
\draw [postaction={decorate}] (T.corner 9)  to  (T.corner 4);
\draw [postaction={decorate}] (T.corner 4)  to  (T.corner 8);
\draw [postaction={decorate}] (T.corner 8)  to  (T.corner 5);
\draw []                      (T.corner 8)  to  (T.corner 5);
\draw [postaction={decorate}] (T.corner 5)  to  (T.corner 7);

\foreach \x in {1,2,...,10}{
\draw (T.corner \x) node [fill,circle,scale=0.2] {};}

\coordinate (m1) at ($0.3*(T.corner 2)+0.4*(T.corner 1)+0.34*(T.corner 10)$);
\coordinate (m2) at ($0.33*(T.corner 2)+0.33*(T.corner 10)+0.33*(T.corner 3)$);
\coordinate (m3) at ($0.33*(T.corner 10)+0.33*(T.corner 9)+0.33*(T.corner 3)$);
\coordinate (m4) at ($0.33*(T.corner 9)+0.33*(T.corner 4)+0.33*(T.corner 3)$);
\coordinate (m5) at ($0.33*(T.corner 4)+0.33*(T.corner 8)+0.33*(T.corner 9)$);
\coordinate (m6) at ($0.33*(T.corner 4)+0.33*(T.corner 8)+0.33*(T.corner 5)$);
\coordinate (m7) at ($0.33*(T.corner 7)+0.33*(T.corner 8)+0.33*(T.corner 5)$);
\coordinate (m8) at ($0.3*(T.corner 7)+0.34*(T.corner 6)+0.4*(T.corner 5)$);

\draw (T.corner 2) -- (m1);
\draw (T.corner 10)--(m2);
\draw (T.corner 3)--(m3);
\draw (T.corner 9)--(m4);
\draw (T.corner 4)--(m5);
\draw (T.corner 8)--(m6);
\draw (T.corner 5)--(m7);
\draw (T.corner 7) -- (m8);

\foreach \t in {1,...,8}{
 \draw (m\t) node [fill,circle,scale=0.3] {};}

\draw [red, line width = 1.2]  (T.corner 3) -- (T.corner 4);
\draw [blue, line width = 1.2] (T.corner 4) -- (m5);
\draw [red, line width = 1.2]  (m5) -- (m7);
\draw [blue, line width = 1.2] (m7) -- (T.corner 5);
\draw [red, line width = 1.2]  (T.corner 5) -- (T.corner 6);

\end{tikzpicture}
 \caption{{Left:} Removing the first two edges gives a path in $\mathcal{T}_{k+2,n}$.
 {Right:} Replacing the first edge $(k,k+2)$ with $(k+1,k+2)$ gives a path in $\mathcal{T}_{k+1,n}$.} \label{fig:zigzag_cases}
 \end{figure}
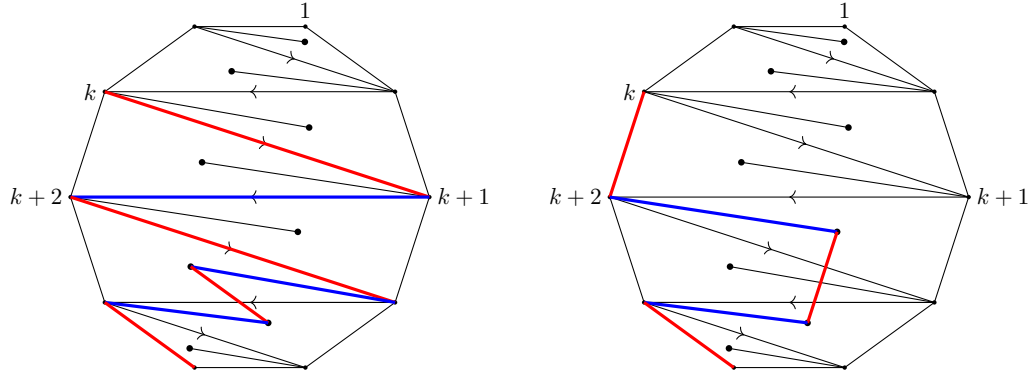

 Now we examine parts (3) and~(4). By induction, part (3) is given by the formula in part~(a):
 \[
 \sqrt{\frac{\lambda_{k+1,n} \lambda_{k+2,n}}{\lambda_{k+1,k+2}}}\, \boxed{k+1,k+2,n}
 = \sum_{i=k+1}^{n-2} \sum_{t \in \mathcal{T}_{i+1,n}} \wt(t) \wt(\sigma_i).
 \]
 Part (4) is equal to $\lambda_{k,k+1} \wt(\sigma_k)$, which is the weight of the first two steps of any super $T$-path $t \in \mathcal{T}_{kn}$ which \emph{does}
 contain a $\tau$-step crossing $(k+1,k+2)$. The terms of part~(3) are all possible ways to complete such a super $T$-path (after joining them by
 the appropriate $\tau$-step, which has weight~1).

\looseness=-1 Again, as in part~(a), we must check that these expressions are written in the correct positive ordering. As we noted above
 in the discussion of part~(a), the positive ordering of the smaller polygon (below the diagonal $(k+1,k+2)$) agrees with
 the positive ordering in the larger polygon. Therefore any factors appearing the formulas obtained above can be assumed to be in
 the correct positive ordering. In parts~(1) and~(2), two of the three factors are single edges in the triangulation, and the terms of the third factor
 (either $\lambda_{k+1,n}$ or $\lambda_{k+2,n}$) appear in positive order by induction. As was discussed in part $(a)$, the terms of
 part~(3) can be written as either $\wt(t)\wt(\sigma_i)$ or $\wt(\sigma_i)\wt(t)$, whichever is correct. Part~(4) only contains a single $\mu$-invariant.
 So all that needs to be checked is that parts~(3) and~(4) occur in the correct order, depending on whether $\theta_k$ comes first or last
 in the positive ordering. But this is precisely the difference between the left and right pictures above (depending on whether vertex~$k$ or~$n$ is on the left of the oriented edge $k+1 \to k+2$), and parts~(3) and~(4) are positioned differently in the two cases.
\end{proof}

\subsection{Proof of Theorem~\ref{thm:main} for generic triangulations}
By a generic triangulation, we mean a polygon in which every triangle has at least one boundary edge. This is the same hypothesis that appears in Proposition~\ref{prop:equiv_rel} because of Remark~\ref{rem:subtriang}. Given a generic triangulation $T$ with $n$ fans, we first apply the flip sequence in Theorem~\ref{thm:mu_single_fan} on each of the fans,
which result in a zig-zag (sub)triangulation $T'$ whose vertices are the fan centers of~$T$ (including $c_0$ and $c_{N+1}$). See Figure~\ref{fig:main_flip_sequence}.

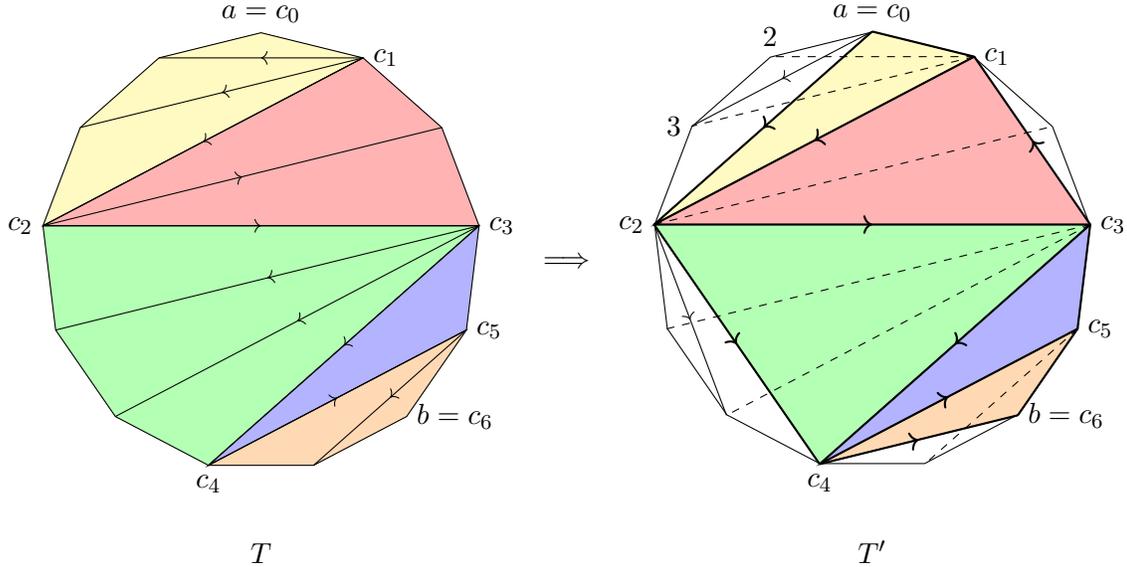
\begin{figure}[h]
\centering
\begin{tikzpicture}[decoration={
    markings,
    mark=at position 0.5 with {\arrow{>}}}
    ] 

\tikzstyle{every path}=[draw] 
		\path
    node[
      regular polygon,
      regular polygon sides=13,
      draw,
      inner sep=2cm,
    ] (T) {}
    %
    (T.corner 1) node[above] {$a=c_0$}
    (T.corner 2) node[above] {}
    (T.corner 3) node[left] {}
    (T.corner 4) node[left] {$c_2$}
    (T.corner 5) node[below] {}
    (T.corner 6) node[left] {}
    (T.corner 7) node[below] {$c_4$}
    (T.corner 8) node[below] {}
    (T.corner 9) node[right] {$b=c_6$}
    (T.corner 10) node[right] {$c_5$}
    (T.corner 11) node[right] {$c_3$}
    (T.corner 12) node[right] {}
    (T.corner 13) node[right] {$c_1$}
;

\coordinate (c0) at (T.corner 1);
\coordinate (c1) at (T.corner 13);
\coordinate (c2) at (T.corner 4);
\coordinate (c3) at (T.corner 11);
\coordinate (c4) at (T.corner 7);
\coordinate (c5) at (T.corner 10);
\coordinate (c6) at (T.corner 9);

\draw[fill=yellow!30!white] (c1)--(T.corner 4)--(T.corner 3)--(T.corner 2)--(T.corner 1)--cycle;
\draw[fill=red!30!white] (T.corner 13)--(T.corner 4)--(T.corner 11)--(T.corner 12)--cycle;

\draw[fill=green!30!white] (T.corner 11)-- (T.corner 4)--(T.corner 5)--(T.corner 6)--(T.corner 7)--cycle;
\draw[fill=blue!30!white] (T.corner 10)--(T.corner 7)--(T.corner 11)--cycle;   
\draw[fill=orange!30!white] (T.corner 10)--(T.corner 9)--(T.corner 8)--(T.corner 7)--cycle;

\draw[postaction={decorate}] (c1)--(T.corner 4);
\draw[postaction={decorate}] (c1)--(T.corner 2);
\draw[postaction={decorate}] (c1)--(T.corner 3);

\draw[postaction={decorate}](c2)--(T.corner 12);
\draw[postaction={decorate}](c2)--(T.corner 11);

\draw[postaction={decorate}](c3)--(T.corner 5);
\draw[postaction={decorate}](c3)--(T.corner 6);
\draw[postaction={decorate}](c3)--(T.corner 7);

\draw[postaction={decorate}](c4)--(c5);

\draw[postaction={decorate}] (T.corner 10)--(T.corner 8);

\node at (0,-4) {$T$};
\end{tikzpicture}
\begin{tikzpicture}
	\node at (0,0) {$\Longrightarrow$};
	\node at (0,-4) {};
\end{tikzpicture}
\begin{tikzpicture}[decoration={
    markings,
    mark=at position 0.5 with {\arrow{>}}}
    ] 

\tikzstyle{every path}=[draw] 
		\path
    node[
      regular polygon,
      regular polygon sides=13,
      draw,
      inner sep=2cm,
    ] (T) {}
    %
    (T.corner 1) node[above] {$a=c_0$}
    (T.corner 2) node[above] {2}
    (T.corner 3) node[left] {3}
    (T.corner 4) node[left] {$c_2$}
    (T.corner 5) node[below] {}
    (T.corner 6) node[left] {}
    (T.corner 7) node[below] {$c_4$}
    (T.corner 8) node[below] {}
    (T.corner 9) node[right] {$b=c_6$}
    (T.corner 10) node[right] {$c_5$}
    (T.corner 11) node[right] {$c_3$}
    (T.corner 12) node[right] {}
    (T.corner 13) node[right] {$c_1$}
;

\coordinate (c0) at (T.corner 1);
\coordinate (c1) at (T.corner 13);
\coordinate (c2) at (T.corner 4);
\coordinate (c3) at (T.corner 11);
\coordinate (c4) at (T.corner 7);
\coordinate (c5) at (T.corner 10);
\coordinate (c6) at (T.corner 9);

\draw[fill=yellow!30!white] (c0)--(c1)--(c2)--(c0);
\draw[fill=red!30!white] (c1)--(c2)--(c3)--(c1);
\draw[fill=green!30!white] (c2)--(c3)--(c4)--(c2);
\draw[fill=orange!30!white] (c4)--(c5)--(c6)--(c4);

\draw[fill=blue!30!white] (T.corner 10)--(T.corner 7)--(T.corner 11)--cycle; 

\draw[thick](c0)--(c1);
\draw[thick](c3)--(c5)--(c6);

\draw[postaction={decorate},thick] (c1)--(c2);
\draw[dashed] (c1)--(T.corner 2);
\draw[dashed] (c1)--(T.corner 3);

\draw[postaction={decorate}] (c0)--(T.corner 3);
\draw[postaction={decorate},thick] (c0)--(c2);

\draw[dashed](c2)--(T.corner 12);
\draw[postaction={decorate},thick](c2)--(c3);
\draw[postaction={decorate},thick](c3)--(c1);

\draw[dashed](c3)--(T.corner 5);
\draw[dashed](c3)--(T.corner 6);
\draw[postaction={decorate},thick](c3)--(c4);

\draw[postaction={decorate}] (c2)--(T.corner 6);
\draw[postaction={decorate},thick](c2)--(c4);

\draw[postaction={decorate},thick] (c4)--(c5);

\draw[dashed] (c5)--(T.corner 8);

\draw[postaction={decorate},thick] (c4)--(c6);

\node at (0,-4) {$T'$};
\end{tikzpicture}
\caption{Flipping edges in each fan, to turn a generic triangulation $T$ into a zig-zag triangulation $T'$.}\label{fig:main_flip_sequence}
\end{figure}

We will then derive our ultimate formula for $\lambda_{ab}$ in $T$ via a combination of Theorem~\ref{thm:mu_single_fan} and Theorem~\ref{thm:proof_zigzag}.
Using Theorem~\ref{thm:proof_zigzag}, we can express $\lambda_{ab}$ in terms of super $T$-paths on $T'$.
This expression uses certain $\lambda$-lengths and $\mu$-invariants that are not in $T$, so we will substitute their expansion
in terms of $T$ using Theorem~\ref{thm:mu_single_fan}.

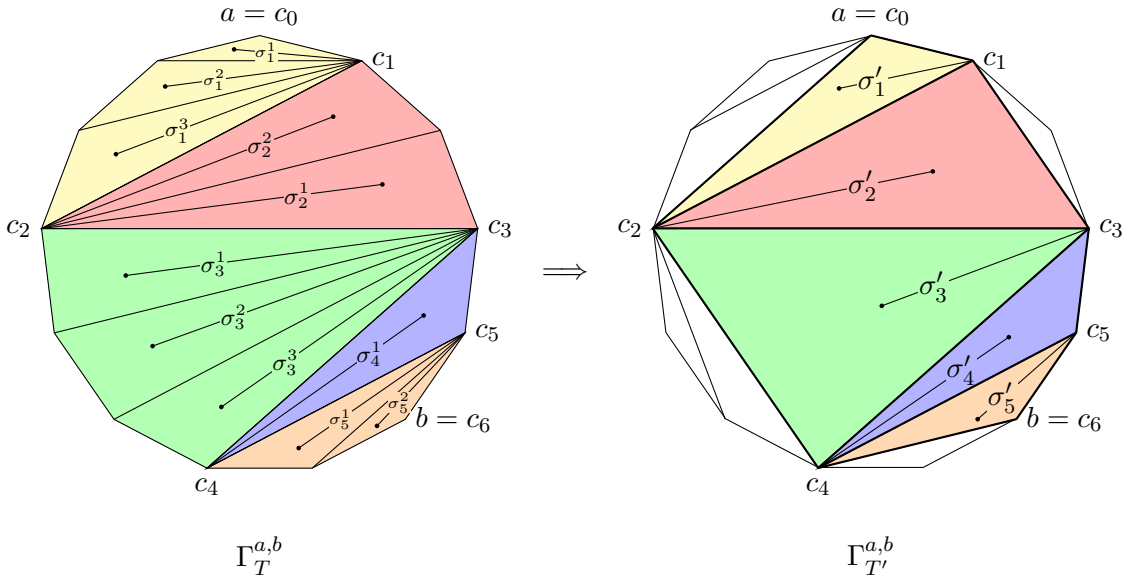
\begin{figure}[h]
\begin{tikzpicture}[decoration={
    markings,
    mark=at position 0.5 with {\arrow{>}}}
    ]

\tikzstyle{every path}=[draw]
		\path
    node[
      regular polygon,
      regular polygon sides=13,
      draw,
      inner sep=2cm,
    ] (T) {}
    %
    (T.corner 1) node[above] {$a=c_0$}
    (T.corner 2) node[above] {}
    (T.corner 3) node[left] {}
    (T.corner 4) node[left] {$c_2$}
    (T.corner 5) node[below] {}
    (T.corner 6) node[left] {}
    (T.corner 7) node[below] {$c_4$}
    (T.corner 8) node[below] {}
    (T.corner 9) node[right] {$b=c_6$}
    (T.corner 10) node[right] {$c_5$}
    (T.corner 11) node[right] {$c_3$}
    (T.corner 12) node[right] {}
    (T.corner 13) node[right] {$c_1$}
;

\coordinate (c0) at (T.corner 1);
\coordinate (c1) at (T.corner 13);
\coordinate (c2) at (T.corner 4);
\coordinate (c3) at (T.corner 11);
\coordinate (c4) at (T.corner 7);
\coordinate (c5) at (T.corner 10);
\coordinate (c6) at (T.corner 9);

\draw[fill=yellow!30!white] (c1)--(T.corner 4)--(T.corner 3)--(T.corner 2)--(T.corner 1)--cycle;
\draw[fill=red!30!white] (T.corner 13)--(T.corner 4)--(T.corner 11)--(T.corner 12)--cycle;

\draw[fill=green!30!white] (T.corner 11)-- (T.corner 4)--(T.corner 5)--(T.corner 6)--(T.corner 7)--cycle;
\draw[fill=blue!30!white] (T.corner 10)--(T.corner 7)--(T.corner 11)--cycle;
\draw[fill=orange!30!white] (T.corner 10)--(T.corner 9)--(T.corner 8)--(T.corner 7)--cycle;

\draw[] (c1)--(T.corner 4);
\draw[] (c1)--(T.corner 2);
\draw[] (c1)--(T.corner 3);

\draw[](c2)--(T.corner 12);
\draw[](c2)--(T.corner 11);

\draw[](c3)--(T.corner 5);
\draw[](c3)--(T.corner 6);
\draw[](c3)--(T.corner 7);

\draw[](c4)--(c5);

\draw[] (T.corner 10)--(T.corner 8);

\coordinate (m11) at ($0.45*(c0)+0.15*(c1)+0.4*(T.corner 2)$);
\coordinate (m52) at ($0.45*(c6)+0.15*(c5)+0.4*(T.corner 8)$);
\coordinate (m12) at ($0.18*(c1)+0.37*(T.corner 3)+0.45*(T.corner 2)$);
\coordinate (m13) at ($0.18*(c1)+0.37*(c2)+0.45*(T.corner 3)$);
\coordinate (m22) at ($0.18*(c2)+0.37*(T.corner 12)+0.45*(c1)$);
\coordinate (m21) at ($0.18*(c2)+0.45*(T.corner 12)+0.37*(c3)$);

\coordinate (m31) at ($0.18*(c3)+0.45*(T.corner 5)+0.37*(c2)$);
\coordinate (m32) at ($0.18*(c3)+0.45*(T.corner 5)+0.37*(T.corner 6)$);
\coordinate (m33) at ($0.18*(c3)+0.45*(c4)+0.37*(T.corner 6)$);

\coordinate (m41) at ($0.18*(c4)+0.4*(c3)+0.42*(c5)$);

\coordinate (m51) at ($0.18*(c5)+0.4*(c4)+0.44*(T.corner 8)$);

\draw (c1)--(m11) node [near end, fill=yellow!30!white,inner sep=0] {\tiny$\sigma_{1}^{1}$};
\draw (c1)--(m12) node [near end, fill=yellow!30!white,inner sep=0.2] {\tiny$\sigma_{1}^{2}$};
\draw (c1)--(m13) node [near end, fill=yellow!30!white,inner sep=0.2] {\scriptsize$\sigma_{1}^{3}$};
\draw (c2)--(m21) node [near end, fill=red!30!white,inner sep=0.2] {\footnotesize$\sigma_{2}^{1}$};
\draw (c2)--(m22) node [near end, fill=red!30!white,inner sep=0.2] {\footnotesize$\sigma_{2}^{2}$};

\draw (c3)--(m31) node [near end, fill=green!30!white,inner sep=0.2] {\footnotesize$\sigma_{3}^{1}$};
\draw (c3)--(m32) node [near end, fill=green!30!white,inner sep=0.2] {\footnotesize$\sigma_{3}^{2}$};
\draw (c3)--(m33) node [near end, fill=green!30!white,inner sep=0.2] {\footnotesize$\sigma_{3}^{3}$};

\draw (c4)--(m41) node [near end, fill=blue!30!white,inner sep=0] {\footnotesize$\sigma_{4}^{1}$};

\draw (c5)--(m51) node [near end, fill=orange!30!white,inner sep=0] {\tiny$\sigma_{5}^{1}$};

\draw (c5)--(m52) node [near end, fill=orange!30!white,inner sep=0] {\tiny$\sigma_{5}^{2}$};

\foreach \t in {11,12,13,21,22,31,32,33,41,51,52}{\draw (m\t) node [fill,circle,scale=0.17]{};}

\node at (0,-4) {$\Gamma_T^{a,b}$};
\end{tikzpicture}
\begin{tikzpicture}
	\node at (0,0) {$\Longrightarrow$};
	\node at (0,-4) {};
\end{tikzpicture}
\begin{tikzpicture}[decoration={
    markings,
    mark=at position 0.5 with {\arrow{>}}}
    ]

\tikzstyle{every path}=[draw]
		\path
    node[
      regular polygon,
      regular polygon sides=13,
      draw,
      inner sep=2cm,
    ] (T) {}
    %
    (T.corner 1) node[above] {$a=c_0$}
    (T.corner 2) node[above] {}
    (T.corner 3) node[left] {}
    (T.corner 4) node[left] {$c_2$}
    (T.corner 5) node[below] {}
    (T.corner 6) node[left] {}
    (T.corner 7) node[below] {$c_4$}
    (T.corner 8) node[below] {}
    (T.corner 9) node[right] {$b=c_6$}
    (T.corner 10) node[right] {$c_5$}
    (T.corner 11) node[right] {$c_3$}
    (T.corner 12) node[right] {}
    (T.corner 13) node[right] {$c_1$}
;

\coordinate (c0) at (T.corner 1);
\coordinate (c1) at (T.corner 13);
\coordinate (c2) at (T.corner 4);
\coordinate (c3) at (T.corner 11);
\coordinate (c4) at (T.corner 7);
\coordinate (c5) at (T.corner 10);
\coordinate (c6) at (T.corner 9);

\draw[fill=yellow!30!white] (c0)--(c1)--(c2)--(c0);
\draw[fill=red!30!white] (c1)--(c2)--(c3)--(c1);
\draw[fill=green!30!white] (c2)--(c3)--(c4)--(c2);
\draw[fill=orange!30!white] (c4)--(c5)--(c6)--(c4);
\draw[fill=blue!30!white] (T.corner 10)--(T.corner 7)--(T.corner 11)--cycle;

\coordinate (m1) at ($0.3*(c2)+0.33*(c1)+0.43*(c0)$) ;
\coordinate (m2) at ($0.3*(c2)+0.33*(c1)+0.43*(c3)$) ;
\coordinate (m3) at ($0.3*(c2)+0.33*(c4)+0.43*(c3)$) ;
\coordinate (m4) at ($0.3*(c5)+0.33*(c4)+0.43*(c3)$) ;
\coordinate (m5) at ($0.3*(c5)+0.33*(c4)+0.43*(c6)$) ;

\draw (c1)--(m1) node [near end, fill=yellow!30!white, inner sep=0.27] {$\sigma_1'$};
\draw (c2)--(m2) node [near end, fill=red!30!white, inner sep=0.27] {$\sigma_2'$};
\draw (c3)--(m3) node [near end, fill=green!30!white, inner sep=0.27] {$\sigma_3'$};
\draw (c4)--(m4) node [near end, fill=blue!30!white, inner sep=0.27] {$\sigma_4'$};
\draw (c5)--(m5) node [near end, fill=orange!30!white, inner sep=0.27] {$\sigma_5'$};

\foreach \t in {1,...,5}{\draw (m\t) node [fill,circle,scale=0.17]{};}

\node at (0,-4) {$\Gamma_{T'}^{a,b}$};

\draw[] (c2)--(T.corner 6);
\draw[thick](c2)--(c4);

\draw[thick] (c4)--(c5);

\draw[thick] (c4)--(c6);
\draw[thick](c0)--(c1);
\draw[thick](c3)--(c5)--(c6);

\draw[thick] (c1)--(c2);

\draw[] (T.corner 3)--(c0);
\draw[thick] (c2)--(c0);

\draw[thick](c2)--(c3);
\draw[thick](c3)--(c1);

\draw[thick](c3)--(c4);

\end{tikzpicture} 
\caption{The auxiliary graphs for $T$ and $T'$, with $\tau$-edges omitted.}\label{fig:main_flip_sequence_labelled}
\end{figure}

We consider the fans of $T$ as subtriangulations, and denote them as $F_1,\dots,F_N$.
We denote the $\mu$-invariants of $T'$ by $\theta_1',\dots,\theta_N'$ and their corresponding $\sigma$-edges to be $\sigma_1',\dots,\sigma_N'$.
We denote the $\mu$-invariants of the $i$-th fan of $T$ by $\theta_i^1,\theta_i^2,\dots$ (ordered counterclockwise around the fan center),
and the corresponding $\sigma$-edges $\sigma_i^1,\sigma_i^2,\dots$. See Figure~\ref{fig:main_flip_sequence_labelled}.

When we substitute super $T$-path expressions for the fans in $T$ into $T'$, we only need to consider two cases:
(1) substitute a boundary edge $(c_{i-1},c_{i+1})$ for $1\leqslant i\leqslant n$,
and (2) substitute a super-step $(\dots,\sigma_i',\tau_{ij},\sigma_j',\dots\,|\,\dots)$,
because every super $T$-path is a concatenation of super steps and complete ordinary $T$-paths.

\begin{enumerate} \itemsep=0pt
 \item[(1)] Suppose we were to replace $\lambda_{c_{i-1},c_{i+1}}$ with its super $T$-path expansion in the $i$-th fan.
 In the super $T$-path of $T'$, the edge $(c_{i-1},c_{i+1})$ must be an odd step because it does not cross the arc $(a,b)$.
 Therefore when we replace it with a super $T$-path in~$F_i$, the indexing agrees up to parity.
Moreover, if an edge crosses the arc $(c_{i-1},c_{i+1})$ in $F_i$, then it must also cross the arc $(a,b)$ in the bigger triangulation~$T$.
 Therefore the axiom (T5) is satisfied, and it is straightforward to verify that all other axioms will follow.

 \item[(2)] We observe that the left-hand side of Theorem~\ref{thm:mu_single_fan}(a)
 equals $\wt(\sigma_i^\prime)$ with respect to $T^\prime$. Using the rest of Theorem~\ref{thm:mu_single_fan}(a), we get the equality
 \[\wt(\sigma_i^\prime)=\sum_{j}\wt\big(\sigma_i^j\big),\]
 i.e., we have that the weight of $\sigma_i^\prime$ in $T^\prime$ is equal to the weighted sum of all $\sigma$-edges in the fan~$F_i$.
 This means that a super step $(\dots,\sigma_i^\prime,\tau_{ij},\sigma_j^\prime,\dots)$ will be expanded into the sum of all
 super steps from a face of $F_i$ to a face of $F_j$. This clearly preserves all axioms of super $T$-paths.
\end{enumerate}

For the converse, we need to prove that every super $T$-path in $T$ can be obtained by such a~substitution.
First, this is clearly true for ordinary $T$-paths, or an ordinary sub-path of a super $T$-path. Therefore we only need to consider the super steps.
Suppose we have a super-step using two $\sigma$-steps $\sigma^*$ and $\sigma^\bullet$, if $\sigma^*$ and $\sigma^\bullet$ are in the same fan,
say $F_i$, then the super step is part of the expansion of $\lambda$-length of $(c_{i-1},c_{i+1})$.
If the $\sigma^*\in F_i$ and $\sigma^\bullet\in F_j$ are in different fans, then the super step came from the super-step
$(\dots,\sigma_i^\prime,\tau_{ij},\sigma_j^\prime,\dots)$ in $T'$.

\section[Super-friezes from super lambda-lengths and mu-invariants]{Super-friezes from super $\boldsymbol{\lambda}$-lengths and $\boldsymbol{\mu}$-invariants}\label{sec:super-friezes}

In this section, we use our formulas for super $\lambda$-lengths and $\mu$-invariants to construct arrays that are variants
of the super-friezes appearing in work of Morier-Genoud, Ovsienko, and Tabachnikov~\cite{M-GOT15}.
Morier-Genoud et al.\ defined a \emph{super-frieze}
to be an array, whose rows alternate so that one row is all even elements, the next is all odd elements, etc. Consider a part of the array,
called an ``\emph{elementary diamond}'', of the form
\[\begin{array}{ccccc}
 & & B & & \\
 & \s{\Xi} & & \s{\Psi} & \\
 A & & & & D \\
 & \s{\Phi} & & \s{\Sigma} & \\
 & & C & &
 \end{array}
\]

Here, Roman letters are even elements and Greek letters are odd elements. The super-frieze rules are
\begin{gather}
AD-BC =1+ \s{\Sigma\Xi},\label{eq:frieze1}\\
AD-BC =1 + \s{\Psi}\s{\Phi},\label{eq:frieze1.5}\\
A\s{\Sigma }-C \s{\Xi} =\s{\Phi},\label{eq:frieze2}\\
B\s{\Sigma }-D \s{\Xi} =\s{\Psi},\label{eq:frieze3}\\
B\s{\Phi}-A\s{\Psi}=\s{\Xi},\label{eq:frieze4}\\
D\s{\Phi}-C\s{\Psi}=\s{\Sigma}.\label{eq:frieze5}
\end{gather}
To define a super-frieze, we only need equation~\eqref{eq:frieze1} or equation~\eqref{eq:frieze1.5} and any two of the four equations \eqref{eq:frieze2}--\eqref{eq:frieze5}. In other words, any two of equations \eqref{eq:frieze2}--\eqref{eq:frieze5} implies the other two, and utilizing these two equations, either of equation~\eqref{eq:frieze1} or equation~\eqref{eq:frieze1.5} implies the other.

We will now observe how the $\lambda$-lengths and $\mu$-invariants satisfy a modified version of these relations.
Put the $\lambda$-lengths in an array so that moving left-to-right along a~row rotates a~diagonal of the polygon
by shifting indices of both endpoints up by 1, and diagonals of the array going south-east have a common first endpoint.
In between these ordinary entries, we put a $\mu$-invariant multiplied by the square-root its two adjacent $\lambda$-lengths, so that $\tilde \mu_{ijk} = \sqrt{\l i j\l j k}\ \boxed{ijk}$ goes in between $\lambda_{ij}$ and $\lambda_{jk}$.
With these conventions, an elementary diamond looks as follows:
\[
 \begin{array}{ccccc}
 & & b & & \\
 & \tilde\theta & & \tilde\sigma' & \\
 e & & & & f \\
 & \tilde\sigma & & \tilde\theta' & \\
 & & d & &
 \end{array}
 =
 \begin{array}{ccccc}
 & & \lambda_{i+1,j} & & \\
 & \tilde \mu_{i,i+1,j} & & \tilde \mu_{i+1,j,j+1} & \\
 \lambda_{ij} & & & & \lambda_{i+1,j+1} \\
 &\tilde \mu_{i,j,j+1} & & \tilde\mu_{i,i+1,j+1} & \\
 & & \lambda_{i,j+1} & &
 \end{array}
\]
\begin{Proposition}\label{prop:frieze-mutation}
Every elementary diamond corresponds to a Ptolemy relation of a quadri\-la\-teral, with two boundary edges having $\lambda$-length $1$.	
\end{Proposition}

\begin{proof}Consider the super flip in the following diagram, where $a=c=1$:
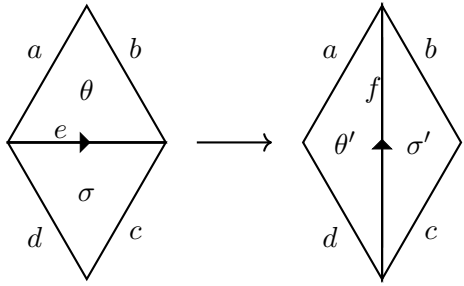
\begin{center}
\begin{tikzpicture}[scale=0.7, baseline, thick]
 \draw (0,0) -- (3,0) -- (60:3) -- cycle;
 \draw (0,0) -- (3,0) -- (-60:3) -- cycle;

 \draw (0,0) -- node{\midarrow} (3,0);

 \draw node[above] at (70:1.5){$a$};
 \draw node[above] at (30:2.8){$b$};
 \draw node[below] at (-30:2.8){$c$};
 \draw node[below=-0.1] at (-70:1.5){$d$};
 \draw node[above] at (1,-0.12){$e$};

 \draw node[left] at (0,0) {};
 \draw node[above] at (60:3) {};
 \draw node[right] at (3,0) {};
 \draw node[below] at (-60:3) {};

 \draw node at (1.5,1){$\theta$};
 \draw node at (1.5,-1){$\sigma$};
\end{tikzpicture}
\begin{tikzpicture}[baseline]
 \draw[->, thick](0,0)--(1,0);
 \node[above] at (0.5,0) {};
\end{tikzpicture}
\begin{tikzpicture}[scale=0.7, baseline, thick,every node/.style={sloped,allow upside down}]
 \draw (0,0)--(60:3)--(-60:3)--cycle;
 \draw (3,0)--(60:3)--(-60:3)--cycle;

 \draw node[above] at (70:1.5) {$a$};
 \draw node[above] at (30:2.8) {$b$};
 \draw node[below] at (-30:2.8) {$c$};
 \draw node[below=-0.1] at (-70:1.5) {$d$};
 \draw node[left] at (1.7,1) {$f$};

 \draw (1.5,-2) --node {\midarrow} (1.5,2);

 \draw node[left] at (0,0) {};
 \draw node[above] at (60:3) {};
 \draw node[right] at (3,0) {};
 \draw node[below] at (-60:3) {};

 \draw node at (0.8,0){$\theta'$};
 \draw node at (2.2,0){$\sigma'$};
\end{tikzpicture}
\end{center}
In the super-diamond, we set $\tilde\theta=\theta\sqrt{be}$, $\tilde\sigma=\sigma\sqrt{ed}$, $\tilde\theta'=\theta'\sqrt{df}$, and $\tilde\sigma' = \sigma'\sqrt{bf}$.
Now, the super Ptolemy relation equation~\eqref{eq:mutation} is
\[ef=1+bd+\sqrt{bd}\sigma\theta.\]
Using equation~\eqref{eq:6}, we substitute $\theta$ with $\sigma\sqrt{bd}-\sigma'\sqrt{ef}$:
\[ef=1+bd+\sqrt{bd}\sigma\big(\sigma\sqrt{bd}-\sigma'\sqrt{ef}\big).\]
Then $\sigma$ squares to zero, so we have
\[ef=1+bd+\sqrt{bdef}\sigma'\sigma.\]
This is exactly equation~\eqref{eq:frieze1.5} in terms of these super-frieze entries, i.e.,
\[ef=1+bd+\tilde\sigma'\tilde\sigma.\]

The other two Ptolemy relations give rise to the desired super-frieze relations as well. The second Ptolemy relation equation~\eqref{eq:5} is \[\theta'\sqrt{ef}=\theta\sqrt{bd}+\sigma. \]
Substitute $\theta$'s with the $\tilde\theta$'s:
\begin{gather*}\tilde\theta'\frac{1}{\sqrt{df} }\sqrt{ef}=\tilde\theta\frac{1}{\sqrt{be}}\sqrt{bd}+\tilde\sigma\frac{1}{\sqrt{de}},\qquad
 \tilde\theta'\sqrt{\frac{e}{d}}= \tilde\theta\sqrt{d\over e}+\tilde\sigma\frac{1}{\sqrt{de}}.\end{gather*}
Then multiply by $\sqrt{de}$ to get
$\tilde\theta'e= \tilde\theta d+\tilde\sigma$.
The third Ptolemy relation equation~\eqref{eq:5} is
\[\sigma'\sqrt{ef}=\sigma\sqrt{bd}-\theta.\]
A similar calculation shows that this is equivalent to
\[\tilde\sigma'e=\tilde\sigma b-\tilde\theta.\]
These relations on the `modified' $\mu$-invariants are exactly the super-frieze relations.
\end{proof}

\begin{Theorem}
 Every super-frieze pattern comes from a decorated super-Teichm\"uller space of a~marked disk.
\end{Theorem}

\begin{proof}
 Take a diagonal of even and odd variables from a super-frieze, and we declare it to be the $\lambda$-lengths
 of a fan triangulation with the default orientation (see Figure~\ref{fig:frieze-triangulation}).

 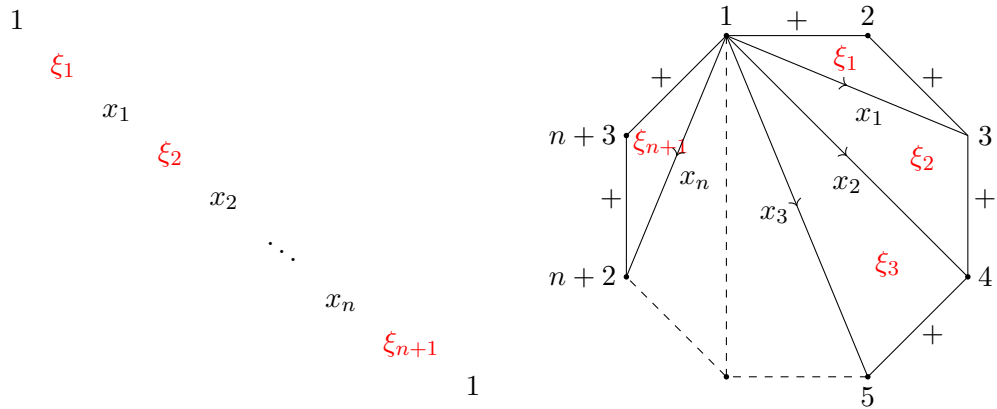
\begin{figure}[h]
 \centering
 \begin{tikzpicture}
 \node at (0,0) {$\displaystyle
 \begin{array}{ccccccccccc}
 1\\[3pt]
 &{\color{red}\xi_1}&&\\[3pt]
 &&x_1\\[3pt]
 &&&{\color{red}\xi_2}&&\\[3pt]
 &&&&x_2&&\\[3pt]
 &&&&&\ddots&&&&\\[3pt]
 &&&&&&x_n\\[3pt]
 &&&&&&&{\color{red}\xi_{n+1}}&&\\[3pt]
 &&&&&&&&1
 \end{array}
 $};
 \end{tikzpicture}
 \begin{tikzpicture}[decoration={
 markings,
 mark=at position 0.5 with {\arrow{>}}}
 ]

 \tikzstyle{every path}=[draw]
 \path
 node[
 regular polygon,
 regular polygon sides=8,
 draw=none,
 inner sep=1.6cm,
 ] (T) {}
 %
 (T.corner 1) node[above] {$2$}
 (T.corner 2) node[above] {$1$}
 (T.corner 3) node[left] {$n+3$}
 (T.corner 4) node[left] {$n+2$}
 (T.corner 5) node[below] {}
 (T.corner 6) node[below] {$5$}
 (T.corner 7) node[right] {$4$}
 (T.corner 8) node[right] {$3$}
 ;
 \draw [-] (T.corner 4) to (T.corner 3) to node[label={[xshift=-0.2cm, yshift=-1.2cm]$\color{red} \xi_{n+1}$}] {} (T.corner 2) to node[label={[xshift=0.65cm, yshift=-0.74cm]$\color{red} \xi_1$}] {} (T.corner 1) to (T.corner 8) to node[label={[xshift=-0.6cm, yshift=0.2cm]$\color{red} \xi_2$}] {} (T.corner 7) to node[label={[xshift=-0.4cm, yshift=0.4cm]$\color{red} \xi_3$}] {} (T.corner 6);
 \draw [dashed] (T.corner 6) to (T.corner 5) to (T.corner 4);

 \draw [postaction={decorate}] (T.corner 2) to node[label={[xshift=-0.3cm, yshift=-0.5cm]$x_3$}] {}(T.corner 6);

 \draw [postaction={decorate}] (T.corner 2) to node[label= below:$x_2$] {} (T.corner 7);
 \draw [postaction={decorate}] (T.corner 2) to node[label={[xshift=0.25cm, yshift=-0.7cm]$x_n$}] {} (T.corner 4);
 \draw [dashed] (T.corner 5) to (T.corner 2);

 \foreach \x in {1,2,...,7}{
 \draw (T.corner \x) node [fill,circle,scale=0.2] {};}

 \draw[postaction={decorate}] (T.corner 12)--(T.corner 4);

 \draw [postaction={decorate}] (T.corner 2) to node[label={[xshift=0.3cm, yshift=-0.8cm]$x_1$}] {}(T.corner 8);

 \draw [-,draw=none] (T.corner 4) to node [label={[xshift=-0.2cm, yshift=-0.3cm]$+$}] {} (T.corner 3) to node [label={[xshift=-0.2cm, yshift=-0.3cm]$+$}] {} (T.corner 2) to node [label={[xshift=0cm, yshift=-0.2cm]$+$}] {} (T.corner 1) to
 node [label={[xshift=0.2cm, yshift=-0.3cm]$+$}] {}
 (T.corner 8) to
 node [label={[xshift=0.23cm, yshift=-0.3cm]$+$}] {} (T.corner 7) to
 node [label={[xshift=0.2cm, yshift=-0.5cm]$+$}] {} (T.corner 6);
 \end{tikzpicture}
 \caption{{Left:} Diagonal of a super-frieze. {Right:} Initial fan triangulation, the $+$ signs records the initial orientation of the spin structure.} \label{fig:frieze-triangulation}
 \end{figure}

 Using Proposition~\ref{prop:frieze-mutation}, the next diagonal to the right corresponds to the triangulation obtained by flipping the edges $x_1,x_2,\dots,x_n$.
 Suppose the next diagonal of the frieze is as follows:\footnote{Using~\cite[Proposition~2.3.1]{M-GOT15}, the top-most non-trivial row of odd variables
 repeats every-other entry while the bottom-most non-trivial row of odd variables alternates in sign every-other entry.}

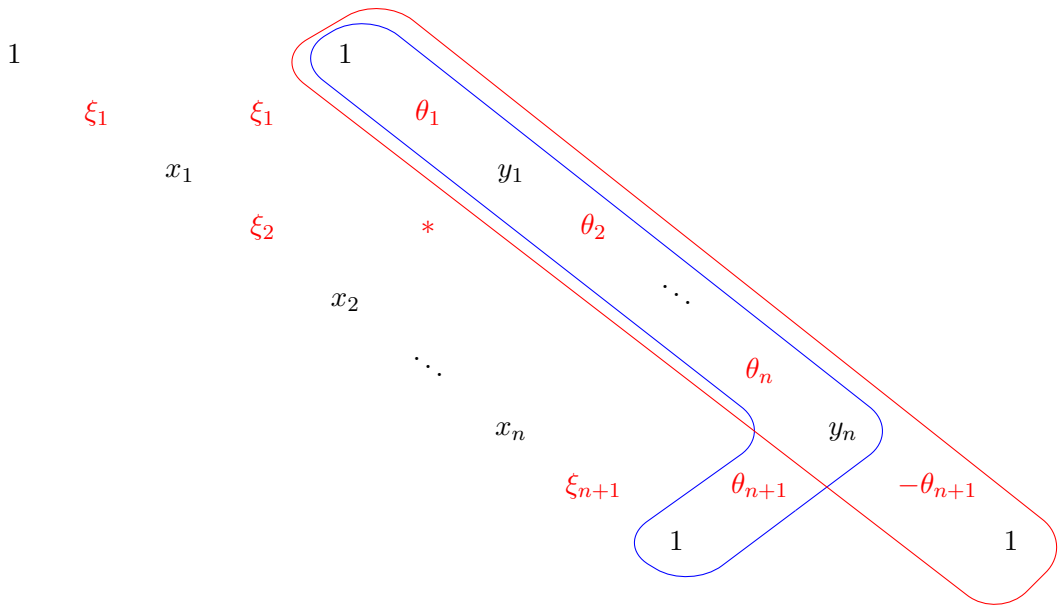
 \begin{tikzpicture}
 \node at (0,0) {$\begin{array}{cccccccccccccccccc}
 {\color{white}\xi_{n+1}}&{\color{white}\xi_{n+1}}&{\color{white}\xi_{n+1}}&{\color{white}\xi_{n+1}}&{\color{white}\xi_{n+1}}&{\color{white}\xi_{n+1}}&{\color{white}\xi_{n+1}}&{\color{white}\xi_{n+1}}&{\color{white}\xi_{n+1}}&{\color{white}\xi_{n+1}}&{\color{white}\xi_{n+1}}
 \\
 1&&&&1\\[8pt]
 &{\color{red}\xi_1}&&\s{\xi_1}&&\s{\theta_1}\\[8pt]
 &&x_1&&&& y_1\\[8pt]
 &&&{\color{red}\xi_2}&&\s{*}&&\s{\theta_2}\\[8pt]
 &&&&x_2&&&&\ddots\\[8pt]
 &&&&&\ddots&&&&\s{\theta_n}\\[8pt]
 &&&&&&x_n&&&&y_n\\[8pt]
 &&&&&&&{\color{red}\xi_{n+1}}&&\s{\theta_{n+1}}&&\s{-\theta_{n+1}}\\[8pt]
 &&&&&&&&1&&&&1
 \end{array}
 $};

 \draw [color=blue,rounded corners=15] (1.75,0.8) to (-1.8,3.6) to (-2.8,3) to (3.6,-2) to (1.5,-3.5) to (2.5,-4.1) to (3.9,-3.05);
 \draw[color=blue,rounded corners=15] (3.9,-3.05) to (5.3,-2) to (1.75,0.8);

 \draw [color=red,rounded corners=15] (-3.05,2.95) -- (6.6,-4.5) -- (7.6,-3.5) -- (-1.6,3.8) -- cycle;
 \end{tikzpicture}

 \looseness=1 Meanwhile, applying super-flips on the edges $x_1,\dots,x_n$ gives us the following
 triangulation:\footnote{Notice that here we apply the clockwise flip sequence, as opposed to the counterclockwise one in Theorem~\ref{thm:mu_single_fan}.
 The reason of switching the convention here is to have the `wrong' arrow on the edge $y_n$ to match up the frieze relations of the last quiddity row.}

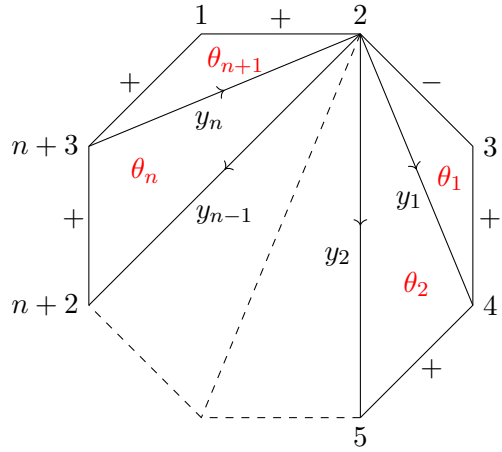
 \begin{center}
 \begin{tikzpicture}[decoration={
 markings,
 mark=at position 0.5 with {\arrow{>}}}
 ]

 \tikzstyle{every path}=[draw]
 \path
 node[
 regular polygon,
 regular polygon sides=8,
 draw=none,
 inner sep=1.8cm,
 ] (T) {}
 %
 (T.corner 1) node[above] {$2$}
 (T.corner 2) node[above] {$1$}
 (T.corner 3) node[left] {$n+3$}
 (T.corner 4) node[left] {$n+2$}
 (T.corner 5) node[below] {}
 (T.corner 6) node[below] {$5$}
 (T.corner 7) node[right] {$4$}
 (T.corner 8) node[right] {$3$}
 ;
 \draw [-] (T.corner 4) to (T.corner 3) to node[label={[xshift=0cm, yshift=-1.5cm]$\color{red} \theta_{n}$}] {} (T.corner 2) to node[label={[xshift=-0.6 cm, yshift=-0.8cm]$\color{red} \theta_{n+1}$}] {} (T.corner 1) to (T.corner 8) to node[label={[xshift=-0.3cm, yshift=0.2cm]$\color{red} \theta_1$}] {} (T.corner 7) to node[label={[xshift=0cm, yshift=0.6cm]$\color{red} \theta_2$}] {} (T.corner 6);
 \draw [dashed] (T.corner 6) to (T.corner 5) to (T.corner 4);

 \draw [postaction={decorate}] (T.corner 1) to node[label={[xshift=-0.1cm, yshift=-0.8cm]$y_1$}] {}(T.corner 7);
 \draw [postaction={decorate}] (T.corner 1) to node[label={[xshift=-0.3cm, yshift=-0.8cm]$y_2$}] {}(T.corner 6);
 \draw [dashed] (T.corner 1) to (T.corner 5);
 \draw [postaction={decorate}] (T.corner 1) to node[label={[xshift=0cm, yshift=-1cm]$y_{n-1}$}] {}(T.corner 4);
 \draw [postaction={decorate}] (T.corner 3) to node[label={[xshift=-0.2cm, yshift=-0.8cm]$y_{n}$}] {}(T.corner 1);

 \draw [-,draw=none] (T.corner 4) to node [label={[xshift=-0.2cm, yshift=-0.3cm]$+$}] {} (T.corner 3) to node [label={[xshift=-0.2cm, yshift=-0.3cm]$+$}] {} (T.corner 2) to node [label={[xshift=0cm, yshift=-0.2cm]$+$}] {} (T.corner 1) to
 node [label={[xshift=0.2cm, yshift=-0.3cm]$-$}] {}
 (T.corner 8) to
 node [label={[xshift=0.23cm, yshift=-0.3cm]$+$}] {} (T.corner 7) to
 node [label={[xshift=0.2cm, yshift=-0.5cm]$+$}] {} (T.corner 6);
 \end{tikzpicture}
 \end{center}

 This corresponds to the blue-circled entries of the above frieze pattern. To obtain the next diagonal,
 we reverse all arrows on the `last' triangle (the triangle corresponding to $\theta_{n+1}$), and negate the $\mu$-invariant.
 This gives us the following triangulation, which corresponds to the red-circled entries of the above frieze:

 
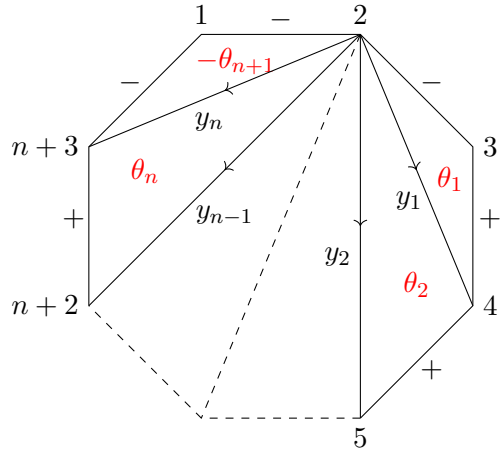
\begin{center}
 \begin{tikzpicture}[decoration={
 markings,
 mark=at position 0.5 with {\arrow{>}}}
 ]

 \tikzstyle{every path}=[draw]
 \path
 node[
 regular polygon,
 regular polygon sides=8,
 draw=none,
 inner sep=1.8cm,
 ] (T) {}
 %
 (T.corner 1) node[above] {$2$}
 (T.corner 2) node[above] {$1$}
 (T.corner 3) node[left] {$n+3$}
 (T.corner 4) node[left] {$n+2$}
 (T.corner 5) node[below] {}
 (T.corner 6) node[below] {$5$}
 (T.corner 7) node[right] {$4$}
 (T.corner 8) node[right] {$3$}
 ;
 \draw [-] (T.corner 4) to (T.corner 3) to node[label={[xshift=0cm, yshift=-1.5cm]$\color{red} \theta_{n}$}] {} (T.corner 2) to node[label={[xshift=-0.6 cm, yshift=-0.8cm]$\color{red}-\theta_{n+1}$}] {} (T.corner 1) to (T.corner 8) to node[label={[xshift=-0.3cm, yshift=0.2cm]$\color{red} \theta_1$}] {} (T.corner 7) to node[label={[xshift=0cm, yshift=0.6cm]$\color{red} \theta_2$}] {} (T.corner 6);
 \draw [dashed] (T.corner 6) to (T.corner 5) to (T.corner 4);

 \draw [postaction={decorate}] (T.corner 1) to node[label={[xshift=-0.1cm, yshift=-0.8cm]$y_1$}] {}(T.corner 7);
 \draw [postaction={decorate}] (T.corner 1) to node[label={[xshift=-0.3cm, yshift=-0.8cm]$y_2$}] {}(T.corner 6);
 \draw [dashed] (T.corner 1) to (T.corner 5);
 \draw [postaction={decorate}] (T.corner 1) to node[label={[xshift=0cm, yshift=-1cm]$y_{n-1}$}] {}(T.corner 4);
 \draw [postaction={decorate}] (T.corner 1) to node[label={[xshift=-0.2cm, yshift=-0.8cm]$y_{n}$}] {}(T.corner 3);

 \draw [-,draw=none] (T.corner 4) to node [label={[xshift=-0.2cm, yshift=-0.3cm]$+$}] {} (T.corner 3) to node [label={[xshift=-0.2cm, yshift=-0.3cm]$-$}] {} (T.corner 2) to node [label={[xshift=0cm, yshift=-0.2cm]$-$}] {} (T.corner 1) to
 node [label={[xshift=0.2cm, yshift=-0.3cm]$-$}] {}
 (T.corner 8) to
 node [label={[xshift=0.23cm, yshift=-0.3cm]$+$}] {} (T.corner 7) to
 node [label={[xshift=0.2cm, yshift=-0.5cm]$+$}] {} (T.corner 6);
 \end{tikzpicture}
 \end{center}

 Now the interior arrows are exactly the `same' as before: all arrows are directed away from the (new) fan center.
 Therefore, inductively, flipping the edges $y_1,\dots,y_n$ and negating the last triangle will give us the next diagonal.

Applying the above operations $n$ times will take us back to the original triangulation, but with different boundary orientations.
 In particular, all boundary arrows are reversed, which is equivalent to reversing the orientation of all triangles and negating all the $\mu$-invariants.

 
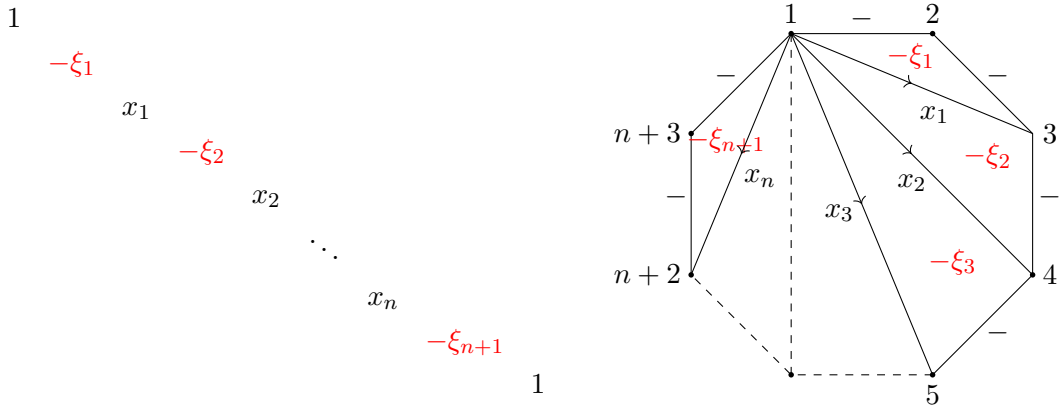
\begin{center}
 \begin{tikzpicture}
	\node at (0,0) {$\displaystyle
 \begin{array}{ccccccccccc}
 1\\[3pt]
 &{\color{red}-\xi_1}&&\\[3pt]
 &&x_1\\[3pt]
 &&&{\color{red}-\xi_2}&&\\[3pt]
 &&&&x_2&&\\[3pt]
 &&&&&\ddots&&&&\\[3pt]
 &&&&&&x_n\\[3pt]
 &&&&&&&{\color{red}-\xi_{n+1}}&&\\[3pt]
 &&&&&&&&1
 \end{array}
 $};
 \end{tikzpicture}
 \begin{tikzpicture}[decoration={
 markings,
 mark=at position 0.5 with {\arrow{>}}}
 ]

 \tikzstyle{every path}=[draw]
		\path
 node[
 regular polygon,
 regular polygon sides=8,
 draw=none,
 inner sep=1.6cm,
 ] (T) {}
 %
 (T.corner 1) node[above] {$2$}
 (T.corner 2) node[above] {$1$}
 (T.corner 3) node[left] {$n+3$}
 (T.corner 4) node[left] {$n+2$}
 (T.corner 5) node[below] {}
 (T.corner 6) node[below] {$5$}
 (T.corner 7) node[right] {$4$}
 (T.corner 8) node[right] {$3$}
 ;
 \draw [-] (T.corner 4) to (T.corner 3) to node[label={[xshift=-0.2cm, yshift=-1.2cm]$\color{red} -\xi_{n+1}$}] {} (T.corner 2) to node[label={[xshift=0.65cm, yshift=-0.74cm]$\color{red} -\xi_1$}] {} (T.corner 1) to (T.corner 8) to node[label={[xshift=-0.6cm, yshift=0.2cm]$\color{red} -\xi_2$}] {} (T.corner 7) to node[label={[xshift=-0.4cm, yshift=0.4cm]$\color{red} -\xi_3$}] {} (T.corner 6);
 \draw [dashed] (T.corner 6) to (T.corner 5) to (T.corner 4);

 \draw [postaction={decorate}] (T.corner 2) to node[label={[xshift=-0.3cm, yshift=-0.5cm]$x_3$}] {}(T.corner 6);

 \draw [postaction={decorate}] (T.corner 2) to node[label= below:$x_2$] {} (T.corner 7);
 \draw [postaction={decorate}] (T.corner 2) to node[label={[xshift=0.25cm, yshift=-0.7cm]$x_n$}] {} (T.corner 4);
 \draw [dashed] (T.corner 5) to (T.corner 2);

 \foreach \x in {1,2,...,7}{
 \draw (T.corner \x) node [fill,circle,scale=0.2] {};}

 \draw[postaction={decorate}] (T.corner 12)--(T.corner 4);

 \draw [postaction={decorate}] (T.corner 2) to node[label={[xshift=0.3cm, yshift=-0.8cm]$x_1$}] {}(T.corner 8);

 \draw [-,draw=none] (T.corner 4) to node [label={[xshift=-0.2cm, yshift=-0.3cm]$-$}] {} (T.corner 3) to node [label={[xshift=-0.2cm, yshift=-0.3cm]$-$}] {} (T.corner 2) to node [label={[xshift=0cm, yshift=-0.2cm]$-$}] {} (T.corner 1) to
 node [label={[xshift=0.2cm, yshift=-0.3cm]$-$}] {}
 (T.corner 8) to
 node [label={[xshift=0.23cm, yshift=-0.3cm]$-$}] {} (T.corner 7) to
 node [label={[xshift=0.2cm, yshift=-0.5cm]$-$}] {} (T.corner 6);
 \end{tikzpicture}
 \end{center}

 Therefore the $n$-th diagonal will have the same even entries and negative odd entries as the first diagonal.
 This explains the glide symmetry of super-frieze patterns.
\end{proof}

\section{Conclusions and future directions}

\subsection[Expansion formulas for mu-invariants]{Expansion formulas for $\boldsymbol{\mu}$-invariants}

In Theorems~\ref{thm:mu_single_fan} and~\ref{thm:proof_zigzag}, we gave formulas for certain types of $\mu$-invariants. However,
these only applied to a subset of such triangles which have at least one side being an arc of the triangulation. The proofs depended on this assumption, and a specific flip sequence,
in order to apply the super Ptolemy relations. This begs the following question, subject to the ambiguity of Remark~\ref{rem:mu-invariants}, and with a specific flip sequence in mind.

\begin{Question}
 What is the correct formula for a $\mu$-invariant of a triangle which has no sides belonging to the triangulation?
\end{Question}

Looking more broadly, we can consider a study of $\lambda$-lenghths and $\mu$-invariants, subject to super Ptolemy relations, for other surfaces.
\begin{Question}
 What is the correct formula for a $\lambda$-length $($or a $\mu$-invariant$)$ for an arc $($resp.\ a triangle$)$ on an annlus, torus, or other surfaces with boundary?
\end{Question}

Note that for the special case of a once-punctured torus with no boundaries, such super structures were studied in~\cite{mcshane}.
The cases of a three-punctured sphere and of a~once-punctured torus were also investigated in \cite[Appendix~B]{ip2018n}, but differed from our setup.
Therein, they had two odd variables (rather than one) for each triangle{, along with an extra family of even variables associated to the edges}.

\subsection{Connections to super cluster algebras and super-friezes}\label{sec:subfrieze}

As we have demonstrated, we have been able to use Penner and Zeitlin's development of super $\lambda$-lengths for decorated super-Teichm\"uller space to obtain explicit formulas for super $\lambda$-lengths on marked disks which involves a construction of super $T$-paths.
In analogy with the classical case, as in \cite{schiffler_08} where weighted generating functions of $T$-paths correspond to cluster variables in cluster algebras of type $A_n$, we wish to investigate how our formulas for super $\lambda$-lengths could aid in the development of \emph{super} cluster algebras (of type $A_n$). Steps towards defining super cluster algebras appeared in work of Ovsienko~\cite{o_15} and separately in the work of Li, Mixco, Ransingh, and Srivastava~\cite{lmrs_17}. These initial steps were followed up by related work such as~\cite{OS19,OT18,sv_19}.

In particular, in \cite{OS19}, Ovsienko and Shapiro define a type of super cluster algebra, motivated by super-frieze patterns. In
their setup, some of the frozen vertices of the quiver correspond to odd variables $\theta_1,\dots,\theta_m$.
There can be paths of length 2 connecting the $\theta_i$ passing through the ``ordinary'' vertices: $\theta_i \to x_k \to \theta_j$.
The mutation rules are the same for all ordinary arrows, and additionally, when mutating at $x_k$,
\begin{enumerate}\itemsep=0pt
 \item[1)] for $\theta_i \to x_k \to \theta_j$, and for $x_k \to x_\ell$, add a 2-path $\theta_i \to x_\ell \to \theta_j$,
 \item[2)] reverse all 2-paths through $x_k$,
 \item[3)] cancel oppositely oriented 2-paths through $x_k$.
\end{enumerate}
The mutation formula for even variables is given by
\[  \mu_k(x_k) =\frac{1}{x_k}\bigg( \prod_{ x_k\to x_\ell} x_\ell+ \prod_{\theta_i \to x_k \to \theta_j} (1 + \theta_i \theta_j) \prod_{x_\ell \to x_k} x_\ell \bigg).\]

Ovsienko and Shapiro noticed that starting with certain initial data, one could construct a~quiver with odd vertices such that
all entries of the super-frieze can be obtained by mutations. Their choice of initial data
consists of all even entries along a NW-SE diagonal, along with the odd entries in the neighboring diagonal. This is pictured below:

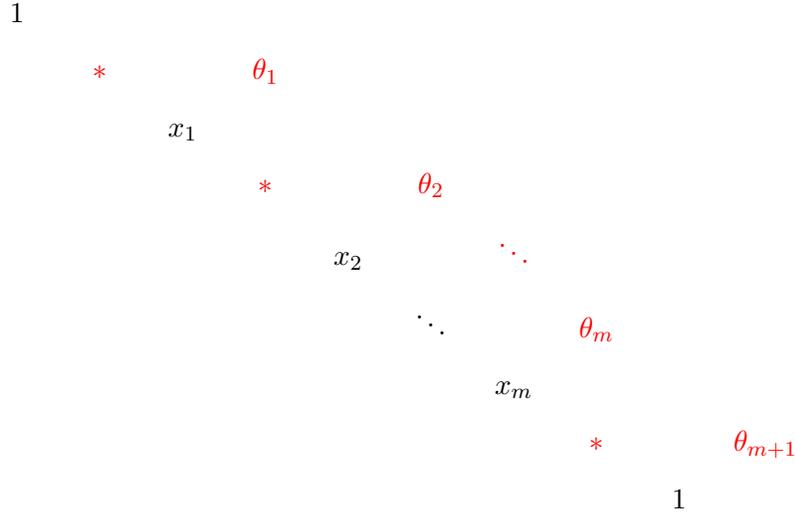
\begin{center}
\(\quad\displaystyle \begin{array}{cccccccccccccccccc}
{\color{white}\xi_{n+1}}&{\color{white}\xi_{n+1}}&{\color{white}\xi_{n+1}}&{\color{white}\xi_{n+1}}&{\color{white}\xi_{n+1}}&{\color{white}\xi_{n+1}}&{\color{white}\xi_{n+1}}&{\color{white}\xi_{n+1}}&{\color{white}\xi_{n+1}}&{\color{white}\xi_{n+1}}&{\color{white}\xi_{n+1}}
\\
1&&&&\\[8pt]
&{\color{red}*}&&\s{\theta_1}&&\\[8pt]
&&x_1&&&&\\[8pt]
&&&{\color{red}*}&&\s{\theta_2}&&\\[8pt]
&&&&x_2&&\s{\ddots} &&\\[8pt]
&&&&&\ddots &&\s{\theta_m}&&\\[8pt]
&&&&&&x_m&&&&\\[8pt]
&&&&&&&{\color{red}*}&&\s{\theta_{m+1}}&&\\[8pt]
&&&&&&&&1&&&&
\end{array}
\)
\end{center}
To this set of initial data, they assign the following quiver:
$$
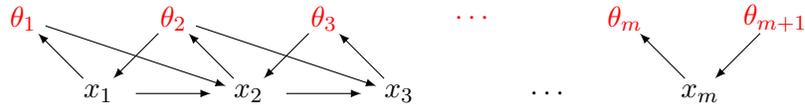
\begin{tikzpicture}
 \draw (0,0) node {$x_1$};
 \draw (2,0) node {$x_2$};
 \draw (4,0) node {$x_3$};
 \draw (6,0) node {$\cdots$};
 \draw (8,0) node {$x_m$};
 \draw[-latex] (0.5,0) -- (1.5,0);
 \draw[-latex] (2.5,0) -- (3.5,0);
 \draw (-1,1) node {\color{red}$\theta_1$};
 \draw (1, 1) node {\color{red}$\theta_2$};
 \draw (3, 1) node {\color{red}$\theta_3$};
 \draw (5, 1) node {\color{red}$\cdots$};
 \draw (7, 1) node {\color{red}$\theta_m$};
 \draw (9, 1) node {\color{red}$\theta_{m+1}$};
 \draw[-latex] (0.8,0.8) -- (0.2,0.2);
 \draw[-latex] (-0.2,0.2) -- (-0.8,0.8);
 \draw[-latex] (2.8,0.8) -- (2.2,0.2);
 \draw[-latex] (1.8,0.2) -- (1.2,0.8);
 \draw[-latex] (3.8,0.2) -- (3.2,0.8);
 \draw[-latex] (8.8,0.8) -- (8.2,0.2);
 \draw[-latex] (7.8,0.2) -- (7.2,0.8);
 \draw[-latex] (-0.7,0.9) -- (1.7,0.1);
 \draw[-latex] (1.3, 0.9) -- (3.7,0.1);
\end{tikzpicture}
$$

\begin{Remark}\rm
In Section \ref{sec:super-friezes}, we described how a super-frieze corresponds to the $\lambda$-lengths and $\mu$-invariants of a triangulated polygon.

Following our construction for the case of an initial fan triangulation, we see that the super-frieze that we construct compares with the construction of Ovsienko and Shapiro via a quiver of even and odd variables as follows:
The initial data of Ovsienko and Shapiro consists of the following:
\begin{itemize}\itemsep=0pt
 \item for even variables, the $\lambda$-lengths of all the diagonals of a fan triangulation,
 \item for odd variables, the $($modified$)$ $\mu$-invariants
 \[ \sqrt{\lambda_{13}} \; \boxed{123}, \sqrt{\lambda_{14}\lambda_{24}} \, \boxed{124},
 \sqrt{\lambda_{15}\lambda_{25}} \, \boxed{125}, \dots, \sqrt{\lambda_{2n}} \, \boxed{12n}\,.
 \]
\end{itemize}
Note that {unlike our usage of collections of $\mu$-invariants}, these $\mu$-invariants correspond to triangles that \emph{do not} belong to the same triangulation.
\end{Remark}

The results of \cite{OS19} therefore show that all $\lambda$-lengths can be expressed in terms of this initial data using sequences of mutations.

On the other hand, our main theorem (Theorem~\ref{thm:main}) shows that all $\lambda$-lengths can be expressed in terms of initial data, where all $\mu$-invariants
come from the same initial triangulation. From the point of view of cluster algebras, it is more natural to have all initial cluster variables coming from
the same triangulation. This leads to the following open question:

\begin{Question}
 Does there exist some modification of the extended quiver mutation from {\rm \cite{OS19}} which realizes the super Ptolemy transformations?
 This would entail $($at least$)$ the following:
 \begin{enumerate}\itemsep=0pt
 \item[$1.$] Specifying which $2$-paths to include for a given triangulation.
 \item[$2.$] Restricting the $2$-path mutation to be compatible with this choice.
 \item[$3.$] Odd variables must change when mutating at an even vertex.
 \end{enumerate}
\end{Question}

Section~5 of~\cite{sv_19} provides yet another conjectural connection to super cluster algebras. In particular, they consider the coordinate ring of the super Grassmannians $G_{r}\big(\mathbb{R}^{n|m}\big)$ of $r$-planes in $(n+m)$-super-space. In the case of $r=2$, $n=4$ or $5$, and $m=1$, these yield a collection of even variables $T^{ab}$ and odd variables $\theta^c$ where $a$ and $b$ can be identified as vertices of an $n$-gon, and the $\theta^c$'s can be identified as the possible triangles inside a quadrilateral or pentagon, respectively. This yields super-Pl\"ucker relations relating these even and odd variables to one another.

\begin{Question}
Is there an algebraic transformation that relates Shemyakova--Voronov's $T^{ab}$'s and $\theta^c$'s of~{\rm \cite{sv_19}} to our $\lambda_{ab}$'s and $\mu$-invariants so that the super-Pl\"ucker relations are satisfied?
\end{Question}

\appendix
\section{The pentagon relation} \label{appendix}

As we highlighted in Remark~\ref{rem:lambda}, under the use of super Ptolemy relations, the $\lambda$-lengths associated to arcs in a polygon are well-defined. To see this, it is sufficient to (1) observe that two consecutive applications of the super Ptolemy relation yields the original arc, and (2) show that the pentagon relation of five alternating applications of the super Ptolemy relation yields the original two arcs.

For (1), we consider the quadrilateral as in Figure~\ref{fig:super_ptolemy}, and flip the arc $e$.
Note that \[ef = ac + bd + \sqrt{abcd} \sigma \theta.\]
After this flip, if we apply the super Ptolemy relation to the diagonal $f$ in the resulting quadrilateral (and denote the resulting diagonal as $g$), we would have
\[fg = bd + ac + \sqrt{abcd} \sigma' \theta'.\]
However, using relation equation~\eqref{eq:mu_cross}, we can rewrite the latter as
\[fg = bd + ac + \sqrt{abcd} \sigma \theta,\] and solving for $g$, we obtain $g = e$ as desired.

The calculations to prove (2) are considerably more involved, and more easily shown using the notation of $\widetilde{\theta}_i$'s
that appeared in Corollary~\ref{cor:laurent}.

We will examine the sequence of flips illustrated in Figure~\ref{fig:pentagon_flips},
and apply the super Ptolemy relations one-by-one.
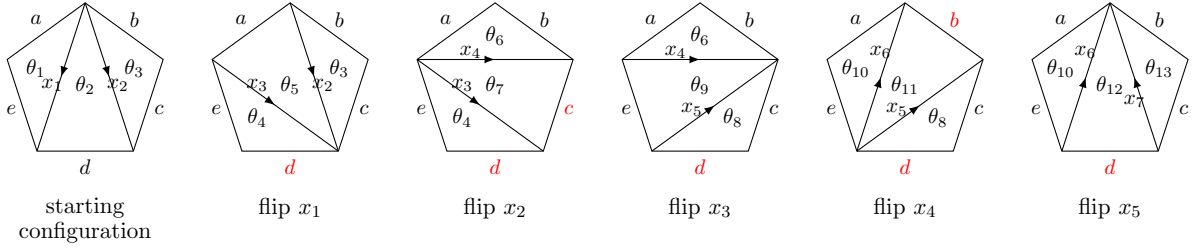
\begin {figure}[h!]
\centering
\begin {tikzpicture}[scale=1.09, every node/.style={scale=0.8}]
 \coordinate (v1) at (0,1);
 \coordinate (v2) at ({cos(pi/2 r + 1*2*pi/5 r)}, {sin(pi/2 r + 1*2*pi/5 r)});
 \coordinate (v3) at ({cos(pi/2 r + 2*2*pi/5 r)}, {sin(pi/2 r + 2*2*pi/5 r)});
 \coordinate (v4) at ({cos(pi/2 r + 3*2*pi/5 r)}, {sin(pi/2 r + 3*2*pi/5 r)});
 \coordinate (v5) at ({cos(pi/2 r + 4*2*pi/5 r)}, {sin(pi/2 r + 4*2*pi/5 r)});

 \draw (v1) -- (v2) -- (v3) -- (v4) -- (v5) -- cycle;

 \draw[->-] (v1) -- (v3);
 \draw[->-] (v1) -- (v4);

 \draw (-0.6,0.8) node {\small $a$};
 \draw (0.6,0.8) node {\small $b$};
 \draw (0.9,-0.3) node {\small $c$};
 \draw (0,-1) node {\small $d$};
 \draw (-0.9,-0.3) node {\small $e$};

 \draw (-0.4,0) node {\small $x_1$};
 \draw (0.4,0) node {\small $x_2$};

 \draw (-0.6,0.2) node {\small $\theta_1$};
 \draw (0,0) node {\small $\theta_2$};
 \draw (0.6,0.2) node {\small $\theta_3$};

 \draw (0,-1.5) node{starting};
 \draw (0,-1.8) node{configuration};


 \begin {scope}[shift={(2.5,0)}]
 \coordinate (v1) at (0,1);
 \coordinate (v2) at ({cos(pi/2 r + 1*2*pi/5 r)}, {sin(pi/2 r + 1*2*pi/5 r)});
 \coordinate (v3) at ({cos(pi/2 r + 2*2*pi/5 r)}, {sin(pi/2 r + 2*2*pi/5 r)});
 \coordinate (v4) at ({cos(pi/2 r + 3*2*pi/5 r)}, {sin(pi/2 r + 3*2*pi/5 r)});
 \coordinate (v5) at ({cos(pi/2 r + 4*2*pi/5 r)}, {sin(pi/2 r + 4*2*pi/5 r)});

 \draw (v1) -- (v2) -- (v3) -- (v4) -- (v5) -- cycle;

 \draw[->-] (v2) -- (v4);
 \draw[->-] (v1) -- (v4);

 \draw (-0.6,0.8) node {\small $a$};
 \draw (0.6,0.8) node {\small $b$};
 \draw (0.9,-0.3) node {\small $c$};
 \draw[red] (0,-1) node {\small $d$};
 \draw (-0.9,-0.3) node {\small $e$};

 \draw (-0.4,0) node {\small $x_3$};
 \draw (0.4,0) node {\small $x_2$};

 \draw (-0.4,-0.4) node {\small $\theta_4$};
 \draw (0,0) node {\small $\theta_5$};
 \draw (0.6,0.2) node {\small $\theta_3$};

 \draw (0,-1.5) node{flip $x_1$};
 \end {scope}


 \begin {scope}[shift={(5,0)}]
 \coordinate (v1) at (0,1);
 \coordinate (v2) at ({cos(pi/2 r + 1*2*pi/5 r)}, {sin(pi/2 r + 1*2*pi/5 r)});
 \coordinate (v3) at ({cos(pi/2 r + 2*2*pi/5 r)}, {sin(pi/2 r + 2*2*pi/5 r)});
 \coordinate (v4) at ({cos(pi/2 r + 3*2*pi/5 r)}, {sin(pi/2 r + 3*2*pi/5 r)});
 \coordinate (v5) at ({cos(pi/2 r + 4*2*pi/5 r)}, {sin(pi/2 r + 4*2*pi/5 r)});

 \draw (v1) -- (v2) -- (v3) -- (v4) -- (v5) -- cycle;

 \draw[->-] (v2) -- (v4);
 \draw[->-] (v2) -- (v5);

 \draw (-0.6,0.8) node {\small $a$};
 \draw (0.6,0.8) node {\small $b$};
 \draw[red] (0.9,-0.3) node {\small $c$};
 \draw[red] (0,-1) node {\small $d$};
 \draw (-0.9,-0.3) node {\small $e$};

 \draw (-0.4,0) node {\small $x_3$};
 \draw (-0.3,0.4) node {\small $x_4$};

 \draw (-0.4,-0.4) node {\small $\theta_4$};
 \draw (0,0) node {\small $\theta_7$};
 \draw (0,0.6) node {\small $\theta_6$};

 \draw (0,-1.5) node{flip $x_2$};
 \end {scope}


 \begin {scope}[shift={(7.5,0)}]
 \coordinate (v1) at (0,1);
 \coordinate (v2) at ({cos(pi/2 r + 1*2*pi/5 r)}, {sin(pi/2 r + 1*2*pi/5 r)});
 \coordinate (v3) at ({cos(pi/2 r + 2*2*pi/5 r)}, {sin(pi/2 r + 2*2*pi/5 r)});
 \coordinate (v4) at ({cos(pi/2 r + 3*2*pi/5 r)}, {sin(pi/2 r + 3*2*pi/5 r)});
 \coordinate (v5) at ({cos(pi/2 r + 4*2*pi/5 r)}, {sin(pi/2 r + 4*2*pi/5 r)});

 \draw (v1) -- (v2) -- (v3) -- (v4) -- (v5) -- cycle;

 \draw[->-] (v3) -- (v5);
 \draw[->-] (v2) -- (v5);

 \draw (-0.6,0.8) node {\small $a$};
 \draw (0.6,0.8) node {\small $b$};
 \draw (0.9,-0.3) node {\small $c$};
 \draw[red] (0,-1) node {\small $d$};
 \draw (-0.9,-0.3) node {\small $e$};

 \draw (-0.1,-0.3) node {\small $x_5$};
 \draw (-0.3,0.4) node {\small $x_4$};

 \draw (0.4,-0.4) node {\small $\theta_8$};
 \draw (0,0) node {\small $\theta_9$};
 \draw (0,0.6) node {\small $\theta_6$};

 \draw (0,-1.5) node{flip $x_3$};
 \end {scope}


 \begin {scope}[shift={(10,0)}]
 \coordinate (v1) at (0,1);
 \coordinate (v2) at ({cos(pi/2 r + 1*2*pi/5 r)}, {sin(pi/2 r + 1*2*pi/5 r)});
 \coordinate (v3) at ({cos(pi/2 r + 2*2*pi/5 r)}, {sin(pi/2 r + 2*2*pi/5 r)});
 \coordinate (v4) at ({cos(pi/2 r + 3*2*pi/5 r)}, {sin(pi/2 r + 3*2*pi/5 r)});
 \coordinate (v5) at ({cos(pi/2 r + 4*2*pi/5 r)}, {sin(pi/2 r + 4*2*pi/5 r)});

 \draw (v1) -- (v2) -- (v3) -- (v4) -- (v5) -- cycle;

 \draw[->-] (v3) -- (v5);
 \draw[->-] (v3) -- (v1);

 \draw (-0.6,0.8) node {\small $a$};
 \draw[red] (0.6,0.8) node {\small $b$};
 \draw (0.9,-0.3) node {\small $c$};
 \draw[red] (0,-1) node {\small $d$};
 \draw (-0.9,-0.3) node {\small $e$};

 \draw (-0.1,-0.3) node {\small $x_5$};
 \draw (-0.3,0.4) node {\small $x_6$};

 \draw (0.4,-0.4) node {\small $\theta_8$};
 \draw (0,0) node {\small $\theta_{11}$};
 \draw (-0.6,0.2) node {\small $\theta_{10}$};

 \draw (0,-1.5) node{flip $x_4$};
 \end {scope}

 \begin{scope}[shift={(12.5,0)}]
 	\coordinate (v1) at (0,1);
 \coordinate (v2) at ({cos(pi/2 r + 1*2*pi/5 r)}, {sin(pi/2 r + 1*2*pi/5 r)});
 \coordinate (v3) at ({cos(pi/2 r + 2*2*pi/5 r)}, {sin(pi/2 r + 2*2*pi/5 r)});
 \coordinate (v4) at ({cos(pi/2 r + 3*2*pi/5 r)}, {sin(pi/2 r + 3*2*pi/5 r)});
 \coordinate (v5) at ({cos(pi/2 r + 4*2*pi/5 r)}, {sin(pi/2 r + 4*2*pi/5 r)});

 \draw (v1) -- (v2) -- (v3) -- (v4) -- (v5) -- cycle;

 \draw[->-] (v3)--(v1);
 \draw[->-] (v4)--(v1);

 \draw (-0.6,0.8) node {\small $a$};
 \draw (0.6,0.8) node {\small $b$};
 \draw (0.9,-0.3) node {\small $c$};
 \draw (0,-1) node {\small ${\color{red} d}$};
 \draw (-0.9,-0.3) node {\small $e$};

 \draw (-0.3,0.4) node {\small $x_6$};
 \draw (0.3,-0.2) node {\small $x_7$};

 \draw (-0.6,0.2) node {\small $\theta_{10}$};
 \draw (0,0) node {\small $\theta_{12}$};
 \draw (0.6,0.2) node {\small $\theta_{13}$};

 \draw (0,-1.5) node{flip $x_5$};
 \end{scope}

\end{tikzpicture}
\caption{Flip sequence to verify the pentagon relation. Red labels indicate reversed orientation.}\label{fig:pentagon_flips}
\end{figure}

Define the ``\emph{modified}'' $\mu$-invariants (as in Corollary~\ref{cor:laurent})
\[ \til_1 = \sqrt{\frac{e}{ax_1}} \theta_1, \qquad \til_2 = \sqrt{\frac{d}{x_1x_2}} \theta_2, \qquad \til_3 =\sqrt{\frac{c}{bx_2}} \theta_3. \]

\textbf{First flip:} After flipping $x_1$, we get $x_3$:
\begin{align*}
 x_3 = \frac{ad + ex_2}{x_1} + \frac{\sqrt{adex_2}}{x_1} \theta_1 \theta_2
 = \frac{ad+ex_2}{x_1} + ax_2 \widetilde{\theta}_1 \widetilde{\theta}_2
\end{align*}
and the new $\theta$'s are
\begin{align*}
 \theta_4 = \frac{\sqrt{ad} \theta_1 - \sqrt{ex_2} \, \theta_2}{\sqrt{x_1x_3}}
 = \frac{1}{\sqrt{dex_3}} \big( ad \widetilde{\theta}_1 - ex_2 \widetilde{\theta}_2 \big)
\end{align*}
and
\begin{align*}
 \theta_5 = \frac{\sqrt{ad} \theta_2 + \sqrt{ex_2} \theta_1}{\sqrt{x_1x_3}}
 = \sqrt{\frac{ax_2}{x_3}} \big( \widetilde{\theta}_1 + \widetilde{\theta}_2 \big).
\end{align*}

\textbf{Second flip:} Flip $x_2$ to get $x_4$:
\begin{align*}
 x_4 = \frac{ac + bx_3}{x_2} + \frac{\sqrt{acbx_3}}{x_2} \theta_5 \theta_3
 = \frac{acx_1+abd+bex_2}{x_1x_2} + ab \big( \widetilde{\theta}_1 \widetilde{\theta}_2 + \widetilde{\theta}_1 \widetilde{\theta}_3 + \widetilde{\theta}_2 \widetilde{\theta}_3 \big).
\end{align*}
The new $\theta$'s are
\begin{align*}
 \theta_6 = \frac{\sqrt{ac} \theta_3 + \sqrt{bx_3} \theta_5}{\sqrt{x_2x_4}}
 = \sqrt{\frac{ab}{x_4}} \big( \widetilde{\theta}_1 + \widetilde{\theta}_2 + \widetilde{\theta}_3 \big)
\end{align*}
and
\begin{align*}
 \theta_7 = \frac{\sqrt{ac} \theta_5 - \sqrt{bx_3} \theta_3}{\sqrt{x_2x_4}}
 = \frac{1}{\sqrt{cx_3x_4}} \big( ac\big(\widetilde{\theta}_1 + \widetilde{\theta}_2\big) - bx_3 \widetilde{\theta}_3 \big).
\end{align*}

{\allowdisplaybreaks \textbf{Third flip:} Flip $x_3$ to get $x_5$:
\begin{align*}
 x_5={} & \frac{ce + dx_4}{x_3} + \frac{\sqrt{cdex_4}}{x_3} \, \theta_4 \theta_7 \\
={} & \frac{cex_1 x_2 + dac x_1 + d a b d + dbex_2}{x_1 x_2 x_3} +
 \frac{dab}{x_3}( \widetilde{\theta}_1\widetilde{\theta}_2 + \widetilde{\theta}_1\widetilde{\theta}_3 +
 \widetilde{\theta}_2 \widetilde{\theta}_3) \\
& {}+ { \frac{1}{x_3^2} (ad \widetilde{\theta}_1 - ex_2 \widetilde{\theta}_2)\left( ac(\widetilde{\theta}_1 + \widetilde{\theta}_2)
 - bx_3 \widetilde{\theta}_3\right) } \\
={} & \left(\frac{ad + ex_2}{x_3}\right)\left(\frac{bd+cx_1}{x_1x_2}\right)
 + \left(\frac{dabx_3 +adac + ex_2 ac}{x_3^2}\right)\widetilde{\theta}_1 \widetilde{\theta}_2 \\
 & {}
 + \left(\frac{dab}{x_3} - \frac{adb}{x_3}\right) \widetilde{\theta}_1 \widetilde{\theta}_3
 + \left(\frac{dab + ex_2 b}{x_3}\right) \widetilde{\theta}_2 \widetilde{\theta}_3 \\
={} &{ \left(\frac{ad + ex_2}{x_3}\right)\left(\frac{bd+cx_1}{x_1x_2}\right)
 + \left(\frac{dab + acx_1}{x_3}\right)\widetilde{\theta}_1 \widetilde{\theta}_2 
 + 
 b x_1 \widetilde{\theta}_2 \widetilde{\theta}_3 }\footnotemark \\
={} & { \left(\frac{ad + ex_2 + a x_1x_2 \widetilde{\theta}_1 \widetilde{\theta}_2}{x_3}\right)\left(\frac{bd+cx_1}{x_1x_2}\right)
 + 
 b x_1 \widetilde{\theta}_2 \widetilde{\theta}_3 } \\
={} & \frac{bd+cx_1}{x_2} + bx_1 \widetilde{\theta}_2 \widetilde{\theta}_3{}^{\ref{footnote1}}
\end{align*}
\footnotetext{\label{footnote1}Here we used $\big(ad + ex_2 + a x_1x_2 \widetilde{\theta}_1\widetilde{\theta}_2\big)/x_3 = x_1$.}
$\!\!$with new $\theta$'s
\begin{align*}
 \theta_8 = \frac{\sqrt{ce} \theta_4 - \sqrt{dx_4} b \theta_7}{\sqrt{x_3x_5}} b
 = \frac{1}{\sqrt{cdx_5}} \left( \frac{-c(ad+ex_2)}{x_3} b \widetilde{\theta}_2 + bd b \widetilde{\theta}_3 \right)
 = \frac{1}{\sqrt{cdx_5}} \big( bd \widetilde{\theta}_3 - cx_1 \widetilde{\theta}_2 \big)
\end{align*}
and
\begin{align*}
 \theta_9 &= \frac{\sqrt{ce} \theta_7 + \sqrt{dx_4} \theta_4}{\sqrt{x_3x_5}}
 = \frac{1}{x_3 \sqrt{ex_4x_5}} \big( a(ec+dx_4) {\widetilde{\theta}_1} + e(ac-x_2x_4) \widetilde{\theta}_2 - be x_3 \widetilde{\theta}_3 \big) \\
 &= {\frac{1}{x_3 \sqrt{ex_4x_5}} \big( \big(a x_3x_5 - aex_2b \til_2\til_3 \big) \widetilde{\theta}_1 - bex_3 \big(\widetilde{\theta}_2 + \widetilde{\theta}_3\big) + eabx_2 \widetilde{\theta}_1 \widetilde{\theta}_2 \widetilde{\theta}_3 \big) }\footnotemark \\
 &= {\frac{1}{\sqrt{ex_4x_5}} \big( a x_5 \widetilde{\theta}_1 - be \big(\widetilde{\theta}_2 + \widetilde{\theta}_3\big) \big). }{}^{\ref{footnote2}}
\end{align*}
\footnotetext{\label{footnote2}Here we used $x_2x_4 = ac + bx_3 + abx_2 \big(\widetilde{\theta}_1+\widetilde{\theta}_2\big)\widetilde{\theta}_3$ and
$x_3x_5 = ce + dx_4 + \big( ac (ad + ex_2)\widetilde{\theta}_1\widetilde{\theta}_2 - ad b x_3 \widetilde{\theta}_1\widetilde{\theta}_1 + ex_2bx_3 \widetilde{\theta}_1\widetilde{\theta}_3\big)/x_3$.}

\textbf{Fourth flip:} Finally, flip $x_4$ to get $x_6$:
\begin{align*}
 x_6 &= \frac{be + ax_5}{x_4} + \frac{\sqrt{abex_5}}{x_4} \theta_9 \theta_6
 \\&= \frac{bex_2 + abd + acx_1}{x_2 x_4}+\frac{abx_1}{x_4} \widetilde{\theta}_2 \widetilde{\theta}_3
 + \frac{ab}{x_4^2} \big( a x_5 \widetilde{\theta}_1 - be \big(\widetilde{\theta}_2 + \widetilde{\theta}_3\big)\big)
 \big(\widetilde{\theta}_1 + \widetilde{\theta}_2 + \widetilde{\theta}_3\big) \\
 &={ \frac{bex_2 + abd + acx_1}{x_2 x_4} + \frac{abx_1}{x_4} \widetilde{\theta}_2 \widetilde{\theta}_3
 + \frac{ab}{x_4} \left( \frac{a x_5 + be}{x_4}\right) \big(\widetilde{\theta}_1 \widetilde{\theta}_2 + \widetilde{\theta}_1 \widetilde{\theta}_3\big) } \\
 &= { x_1 - \frac{abx_1}{x_4} \big(\widetilde{\theta}_1 \widetilde{\theta}_2 + \widetilde{\theta}_1 \widetilde{\theta}_3 + \widetilde{\theta}_2 \widetilde{\theta}_3 \big) + \frac{abx_1}{x_4} \widetilde{\theta}_2 \widetilde{\theta}_3
 + \frac{ab}{x_4} \left( \frac{a x_5 + be}{x_4}\right)
 \big(\widetilde{\theta}_1 \widetilde{\theta}_2 + \widetilde{\theta}_1 \widetilde{\theta}_3\big)} \footnotemark \\
 &= { x_1 - \frac{abx_1}{x_4} \big(\widetilde{\theta}_1 \widetilde{\theta}_2 + \widetilde{\theta}_1 \widetilde{\theta}_3 + \widetilde{\theta}_2 \widetilde{\theta}_3 \big) + \frac{abx_1}{x_4} \widetilde{\theta}_2 \widetilde{\theta}_3} \\
 &\phantom{=} \ {} + { \frac{ab}{x_4} \left(x_1 - \frac{abx_1}{x_4} (\widetilde{\theta}_1 \widetilde{\theta}_2 + \widetilde{\theta}_1 \widetilde{\theta}_3 + \widetilde{\theta}_2 \widetilde{\theta}_3) + \frac{abx_1}{x_4} \widetilde{\theta}_2 \widetilde{\theta}_3
\right)
 \big(\widetilde{\theta}_1 \widetilde{\theta}_2 + \widetilde{\theta}_1 \widetilde{\theta}_3\big)} {}^{\ref{footnote3}} \\
 &= {x_1 - \frac{abx_1}{x_4} \big(\widetilde{\theta}_1 \widetilde{\theta}_2 + \widetilde{\theta}_1 \widetilde{\theta}_3 + \widetilde{\theta}_2 \widetilde{\theta}_3 \big) + \frac{abx_1}{x_4} \widetilde{\theta}_2 \widetilde{\theta}_3
 + \frac{abx_1 }{x_4} \big(\widetilde{\theta}_1 \widetilde{\theta}_2 + \widetilde{\theta}_1 \widetilde{\theta}_3\big)} \\
 &= x_1.
\end{align*}
\footnotetext{\label{footnote3}Here we used $x_1 = {(acx_1+abd+bex_2)}/{x_2x_4} + {ab x_1}\big( \widetilde{\theta}_1 \widetilde{\theta}_2 + \widetilde{\theta}_1 \widetilde{\theta}_3 + \widetilde{\theta}_2 \widetilde{\theta}_3 \big)/{x_4} $.}
$\!\!$The new $\theta$'s are
\begin{align*}
\theta_{10}&=\frac{\theta_6\sqrt{be}+\theta_9\sqrt{ax_5}}{\sqrt{x_4x_6}}\\
&=\sqrt{be\over x_4x_6}\cdot\sqrt{ab\over x_4}\big(\til_1+\til_2+\til_3\big)+\sqrt{ax_5\over x_4x_6}\cdot\sqrt{1\over{ex_4x_5}}\big(ax_5\til_1-be\til_2-be\til_3\big)\\
&=\sqrt{a\over ex_6}\left(be+ax_5\over x_4\right)\til_1\footnotemark
\\
&=\sqrt{a\over ex_6}x_6\til_1\\
&=\sqrt{ax_1\over e}\til_1\footnotemark\\
&=\theta_1
\end{align*}
and
\addtocounter{footnote}{-1}
\footnotetext{Here we used the non-obvious fact that $\theta_9\theta_6\theta_1=0$.}
\addtocounter{footnote}{+1}
\footnotetext{We use the equality $x_6=x_1$ here.}
\begin{align*}
\theta_{11}&=\frac{\theta_9\sqrt{be}-\theta_6\sqrt{ax_5}}{\sqrt{x_4x_6}}\\
&={\sqrt{b}\over x_4\sqrt{x_5x_6}}\big( a x_5 \widetilde{\theta}_1 - be \big(\widetilde{\theta}_2 + \widetilde{\theta}_3\big) \big)-{a\sqrt{bx_5}\over x_4\sqrt{x_6}}\big(\til_1+\til_2+\til_3\big)\\
&=-\sqrt{b\over x_5x_6}\left(be+ax_5\over x_4\right)\big(\til_2+\til_3\big)\footnotemark\\
&=-\sqrt{b x_6\over x_5}\big(\til_2+\til_3\big).
\end{align*}
\footnotetext{Similar to above, here we made use of $\theta_9\theta_6(\theta_2+\theta_3)=0$.}

\textbf{Fifth flip:} Finally, we flip $x_5$ to $x_7$ and get back to the original triangulation (with different orientation):
\begin{align*}
x_7&=	\frac{bd+cx_6}{x_5} + \frac{\sqrt{bcdx_6}}{x_5}\theta_8\theta_{11}\\
&=\frac{bd+cx_6}{x_5}+{\sqrt{bcdx_6}\over x_5}\cdot
\frac{1}{\sqrt{cdx_5}} \big( bd \widetilde{\theta}_3 - cx_1 \widetilde{\theta}_2 \big)\left(-\sqrt{b x_6\over x_5}\big(\til_2+\til_3\big)\right)\\
&=\frac{bd+cx_1}{x_5}+{bx_1\over x_5}\left(\frac{bd+cx_1}{x_5}\right)\til_2\til_3\\
&=\left(\frac{bd+cx_1}{x_5}\right)\left(1+{bx_1\over x_5}\til_2\til_3\right)\\
&=\left(x_2-\frac{bx_1x_2}{x_5}\til_2\til_3\right)\left(1+{bx_1\over x_5}\til_2\til_3\right)\footnotemark\\
&=x_2+\left(-{bx_1x_2\over x_5}+{bx_1x_2\over x_5}\right)\til_2\til_3\\
&=x_2.
\end{align*}
\footnotetext{\label{fn4}We use $x_2x_5=bd+cx_1+bx_1x_2x_5\til_2\til_3$ here.}
$\!\!$Finally, the new $\theta$'s are
\begin{align*}
\theta_{12}&={\theta_{11}\sqrt{bd}+\theta_8\sqrt{cx_1}\over {\sqrt{x_5x_2}}}\\
&=-\left({bd+cx_1}\over{x_5} \right)\sqrt{x_6\over x_7d}\ \til_2\\	
&=-\left(x_2-\frac{bx_1x_2}{x_5}\til_2\til_3\right)\sqrt{x_1\over x_2d}\ \til_2 {}^{\ref{fn4},\ref{fn5} }	\\
&=-\sqrt{x_2x_1\over d}\ \til_2\\
&=-\theta_2,
\\
\theta_{13}&={\theta_{8}\sqrt{bd}-\theta_{11}\sqrt{cx_1}\over {\sqrt{x_5x_2}}}\\
&=\frac{\sqrt{b}}{x_5\sqrt{cx_7}}\big(bd\til_3-cx_1\til_2\big)+\frac{x_1\sqrt{bc}}{x_5\sqrt{x_7}}\big(\til_2+\til_3\big)\\
&=\left(bd+cx_1\over x_5\right)\sqrt{b\over cx_2}\ \til_3 \footnotemark\\
&=\left(x_2-\frac{bx_1x_2}{x_5}\til_2\til_3\right)\sqrt{b\over cx_2}\ \til_3{}^{\ref{fn4} }\\
&=\sqrt{x_2b\over c}\ \til_3\\
&=\theta_3.
\end{align*}
\footnotetext{Here we use $x_7=x_2$.\label{fn5}}

We see in the end that $x_6 = x_1$, $x_7=x_2$. Note also that the edge $x_6$, $x_7$ are oriented opposite of $x_1$, $x_2$, and so the $\lambda_{ij}$'s
are independent of the orientations. We also see that we get the same $\mu$-invariants (up to sign), as $\theta_{10} = \theta_1$,
$\theta_{12} = -\theta_2$, and $\theta_{13} = \theta_3$.

}

\subsection*{Acknowledgements}

The authors would like to thank the support of the NSF grant DMS-1745638 and the University of Minnesota UROP program. We would also like to thank Misha Shapiro and Leonid Chekhov for inspiring conversations, as well as the anonymous referees for their helpful feedback.

\pdfbookmark[1]{References}{ref}
\LastPageEnding

\end{document}